\documentclass[reqno, 11pt, a4paper]{amsart}

\usepackage{
    a4wide,
    amssymb,
    bbm,
    centernot,
    mathtools,
    esint,
    comment
} 
\usepackage{graphicx}
\graphicspath{{./images/}}
\usepackage[cal=euler, scr=boondoxo]{mathalfa}
\usepackage[colorlinks=true, linkcolor=blue,citecolor=blue]{hyperref}
\usepackage{float}
\newtheorem{theorem}{Theorem}[section]
\newtheorem{lemma}[theorem]{Lemma}
\newtheorem{prop}[theorem]{Proposition}
\newtheorem{assumption}[theorem]{Assumption}
\newtheorem{corollary}[theorem]{Corollary}

\theoremstyle{definition}

\theoremstyle{remark}
\newtheorem{remark}[theorem]{Remark}

\numberwithin{equation}{section}

\newcommand{\IND}{\mathbbm{1}}

\newcommand{\capa}{\mathrm{Cap}}
\newcommand{\supp}{\mathrm{supp}}

\newcommand{\De}{\mathrm{d}}

\newcommand{\cA}{\ensuremath{\mathcal A}}

\newcommand{\cD}{\ensuremath{\mathcal D}}
\newcommand{\cE}{\ensuremath{\mathcal E}}

\newcommand{\cG}{\ensuremath{\mathcal G}}
\newcommand{\cH}{\ensuremath{\mathcal H}}
\newcommand{\cI}{\ensuremath{\mathcal I}}

\newcommand{\cL}{\ensuremath{\mathcal L}}

\newcommand{\cQ}{\ensuremath{\mathcal Q}}
\newcommand{\cR}{\ensuremath{\mathcal R}}

\newcommand{\cV}{\ensuremath{\mathcal V}}
\newcommand{\cW}{\ensuremath{\mathcal W}}

\newcommand{\bbE}{\ensuremath{\mathbb E}}

\newcommand{\bbN}{\ensuremath{\mathbb N}}

\newcommand{\bbP}{\ensuremath{\mathbb P}}

\newcommand{\bbR}{\ensuremath{\mathbb R}}

\newcommand{\bbZ}{\ensuremath{\mathbb Z}}
\newcommand{\frr}{\ensuremath{\mathfrak r}}

\begin{document}
\newcounter{cnstcnt}
\newcommand{\newconstant}{%
\refstepcounter{cnstcnt}%
\ensuremath{c_\thecnstcnt}}
\newcommand{\constant}[1]{\ensuremath{c_{\ref{#1}}}}

\title[]{Lower bounds for Bulk deviations for the simple random walk on $\bbZ^d$, $d\geq 3$}

%   author one information

\author{Alberto Chiarini}
\address{Universit\`a degli Studi di Padova}
\curraddr{Department of Mathematics ``Tullio Levi-Civita'', via Trieste 63, 35121, Padova}
\email{chiarini@math.unipd.it}
\thanks{}

%   author two information

\author{Maximilian Nitzschner}
\address{Department of Mathematics, The Hong Kong University of Science and Technology}
\curraddr{Clear Water Bay, Kowloon, Hong Kong}
\email{mnitzschner@ust.hk}
\thanks{}

\begin{abstract}
This article investigates the behavior of the continuous-time simple random walk on $\mathbb{Z}^d$, $d \geq 3$. We derive an asymptotic lower bound on the principal exponential rate of decay for the probability that the average value over a large box of some non-decreasing local function of the field of occupation times of the walk exceeds a given positive value. This bound matches at leading order the corresponding upper bound derived by Sznitman in~\cite{sznitman2019bulk}, and is given in terms of a certain constrained minimum of the Dirichlet energy of functions on $\mathbb{R}^d$ decaying at infinity. Our proof utilizes a version of tilted random walks, a model originally constructed by Li in~\cite{li2017lower} to derive lower bounds on the probability of the event that the trace of a simple random walk disconnects a macroscopic set from an enclosing box. %Our proof brings into play a generic local coupling at a mesoscopic scale of the tilted walk with a Poisson point process of excursions exhibiting the same law as interlacements at a locally constant level, which is of independent interest.
\end{abstract}

\subjclass[2010]{}
\dedicatory{}
\maketitle
%\tableofcontents

\section{Introduction}
\label{sec:introduction}

In this article, we study large-deviation type asymptotics related to the occupation-time field of a continuous-time simple random walk on $\mathbb{Z}^d$, $d \geq 3$. As a main result, we derive a lower bound on the principal exponential rate of decay for the probability that the average in a large box of a certain non-decreasing local function of the occupation-time field of the simple random walk exceeds a strictly positive fraction $\nu$. These lower bounds on the rate match the corresponding upper bounds found in~\cite{sznitman2019bulk},~\cite{sznitman2021excess}, and~\cite{sznitman2023cost}. As one particular example, our result gives insight into the question of how costly it is for a simple random walk to cover a macroscopic fraction of a given box, and provides a near-optimal strategy to obtain such a largely deviant behavior for the walk. \smallskip

The investigation of similar large-deviation type events for random walks, and the related model of \textit{random interlacements}, has received much attention recently, see in particular~\cite{chiarini2020entropic,duminil2023characterization,li2017lower,li2014lowerBound,li2015large,li2022large,nitzschner2017solidification,
sznitman2017disconnection,
sznitman2018macroscopic}, as well as the aforementioned works~\cite{sznitman2021excess,sznitman2019bulk,
sznitman2023cost}. 
The central task in deriving asymptotic lower bounds is the implementation of a suitable \textit{strategy} to enforce the largely deviant event under consideration, which in our set-up brings into play a modified version of \textit{tilted random walks}. The latter were introduced in~\cite{li2017lower} in the context of proving asymptotic lower bounds on the probability that the simple random walk disconnects a (regular) macroscopic set from an enclosing box, and used in~\cite{sznitman2018macroscopic} to produce similar lower bounds concerning macroscopic holes in the connected component of the vacant set of a simple random walk in a large box. The lower bounds on the exponential rates of decay for the bulk-deviation events studied here are given in terms of certain constrained minima for the Dirichlet energy of functions on $\mathbb{R}^d$ decaying at infinity. A remarkable feature is that these rates intrinsically involve the expectation of a local function of random interlacements. The underlying mechanism for this is a \textit{local coupling} of the tilted random walk in mesoscopic boxes with a sequence of excursions having essentially the same law as those generated by random interlacements at a (locally) constant level, which varies in space in a way governed by the square of the minimizer of the corresponding Dirichlet problem. Our findings also relate to the ``Swiss cheese'' picture proposed in~\cite{van2001moderate} to study moderate deviations of the volume of the Wiener sausage, see also~\cite{phetpradap2011intersections} for the adaptation to random walks, and may moreover be compared to the (non-asymptotic) bounds obtained for similar questions concerning the deviant behavior of the range of random walks, see~\cite{asselah2017moderate,asselah2020extracting,asselah2020nature,
asselah2021two,erhard2023uniqueness}. We refer to Remark~\ref{rem:SwissCheese} in Subsection~\ref{subsec:Applications} for a more thorough discussion of this viewpoint.  \smallskip

We will now describe the set-up and our results in a more detailed fashion. We consider the continuous-time simple random walk $(X_t)_{t \geq 0}$ on the integer lattice $\mathbb{Z}^d$, $d \geq 3$, started at a point $x \in \mathbb{Z}^d$, and we denote the governing canonical law by $P_x$. The occupation-time field of the walk corresponds to the total time spent by the random walk in each point of the lattice, and is denoted by $(L_x)_{x \in \mathbb{Z}^d}$. Moreover, we will also need the occupation-time field of continuous-time random interlacements $\mathcal{I}^u$ at level $u \geq 0$, denoted by $(\mathcal{L}^u_x)_{x \in \mathbb{Z}^d}$, and we let $\mathbb{P}$ stand for the governing probability measure and $\mathbb{E}$ for the corresponding expectation. The model of random interlacements was introduced in~\cite{sznitman2010vacant}, and we refer to Section~\ref{sec:preliminaries} for details on the construction relevant to our investigation, as well as to~\cite{cerny2012random,drewitz2014introduction} for a thorough introduction to this model. \smallskip

Similarly as in~\cite{sznitman2019bulk}, we are interested in the behavior of \textit{local functions} of the occupation-time field. Here, a local function is defined as a map $F: [0,\infty)^{B(0,\mathfrak{r})} \rightarrow [0,\infty)$ with some non-negative integer $\mathfrak{r}$ (and $B(0,\mathfrak{r})$ denoting the closed sup-norm ball of radius $\mathfrak{r}$ around the origin) satisfying certain regularity conditions, see Assumption~\ref{eq:RegularityCondF}. Essentially, these conditions require $F$ to be non-decreasing, non-constant, to fulfill $F(0) = 0$, and to admit sub-linear growth. We then consider a set $D \subseteq \mathbb{R}^d$ such that $D$ is either the closure of a smooth, bounded domain containing the origin, or the closure of an open sup-norm ball in $\mathbb{R}^d$ containing the origin, and let 
\begin{equation}
\label{eq:DiscreteBlowup}
D_N = (ND) \cap \mathbb{Z}^d, \qquad N \geq 1
\end{equation}
stand for its discrete blow-up. The set $D$ will act as a model shape, and may for the purposes of this introduction be fixed to $D = [-1,1]^d$. Our primary interest lies in a class of \textit{excess-type events}, defined for fixed $\nu > 0$ and a given local function $F$ by
\begin{equation}
\label{eq:ExcessTypeEvent}
\mathcal{A}^\nu_N(F) = \Big\{\sum_{x \in D_N} F\Big((L_{x+y})_{y \in B(0,\mathfrak{r})}\Big) > \nu |D_N| \Big\}, \qquad N \geq 1.
\end{equation} 
To give some pertinent examples that our result applies to, one may consider 
\begin{equation}
\label{eq:F_examples_introduction}
\begin{split}
F_1(\ell) & = \ell \text{ (with $\mathfrak{r} = 0$)}, \\
F_2(\ell) & = \mathbbm{1}_{\{\ell > 0 \} }\text{ (with $\mathfrak{r} = 0$)},   \\
F_3(\ell) & = \mathbbm{1}_{\{ \text{any path in $B(0,\mathfrak{r})$ from $0$ to $S(0,\mathfrak{r})$ meets some $y$ with $\ell_y > 0$}  \} }\text{ (with $\mathfrak{r} \geq 1$)},
\end{split}
\end{equation}
where $S(0,\mathfrak{r})$ is the sphere of sites at sup-norm distance $\mathfrak{r}$ from the origin. With these choices, the corresponding events $\mathcal{A}^\nu_N(F_1)$, $\mathcal{A}^\nu_N(F_2)$, and $\mathcal{A}^\nu_N(F_3)$ respectively stand for the occurence of an excessive (volume-like) occupation time of the walk within $D_N$, the occurence of an excessive presence of the range within $D_N$, or the occurence of an excessive number of points disconnected from the boundary of enclosing boxes by the trace of the walk within $D_N$. \smallskip

Our main result gives asymptotic lower bounds on the principal decay rate of the probability of the event $\mathcal{A}^\nu_N(F)$ in~\eqref{eq:ExcessTypeEvent}. To state these bounds requires another definition, which crucially brings into play the occupation-time field of random interlacements, namely the function 
\begin{equation}\label{eq:theta}
 \vartheta : [0,\infty) \rightarrow [0,\infty), \qquad u \mapsto \mathbb{E}\Big[F\Big((\mathcal{L}^u_{y})_{y \in B(0,\mathfrak{r})}\Big) \Big]
\end{equation}
(where the requirements made in Assumption~\ref{eq:RegularityCondF} on $F$ ensure that $\vartheta$ is finite, non-decreasing, fulfills $\vartheta(0) = 0$, and is Lipschitz continuous, see Lemma~\ref{lem:PropertiesTheta}). Incidentally, note that if $F  = F_3$ in~\eqref{eq:F_examples_introduction}, $\vartheta(u)$ can be viewed as a finite-range approximation of the probability that the origin is not in the infinite connected component of the vacant set of random interlacements at level $u$ (see~\eqref{eq:theta-0} and below). In general, $\vartheta$ is not necessarily related to the percolation probability. \smallskip
  
As our main result, we will prove in Theorem~\ref{thm:MainThm} that for any local function $F$ fulfilling Assumption~\ref{eq:RegularityCondF} and $\nu \in (0,\vartheta_\infty)$, where $\vartheta_\infty = \lim_{u \rightarrow \infty}\vartheta(u)$, and any $y \in \bbZ^d$, it holds that
\begin{equation}
\label{eq:IntroMainResultLower}
\begin{split}
\liminf_{N \rightarrow \infty} & \frac{1}{N^{d-2}} \log P_y[\cA^\nu_N(F)] \\
&\geq - \inf\bigg\{\frac{1}{2d}\int_{\bbR^d} |\nabla \phi|^2\,\De x\, : \, \phi\in C_0^\infty(\bbR^d)\,,\fint_D \vartheta(\phi^2)\,\De x > \nu\bigg\} \\
& = - \min\bigg\{\frac{1}{2d}\int_{\bbR^d} |\nabla \phi|^2\,\De x\, : \,\phi\in D^{1,2}(\bbR^d)\,,\fint_D \vartheta(\phi^2)\,\De x = \nu\bigg\}
\end{split}
\end{equation}
(where $\fint_D (...) \De x$ stands for the integral $\frac{1}{|D|} \int (...) \De x$, $|D|$ is the Lebesgue measure of $D$, and $D^{1,2}(\bbR^d)$ is the space of locally integrable functions $\phi$ such that $\nabla \phi$, interpreted in a distributional sense, is square-integrable and $\phi$ is vanishing at infinity, as in Chapter 8 $\S$2 of~\cite{lieb2001analysis}). Let us also point out that when $d \geq 5$, one may replace $D^{1,2}(\bbR^d)$ above by the traditional Sobolev space $H^1(\bbR^d)$, see Remark 5.10 in~\cite{sznitman2019bulk}. \smallskip

This result complements a corresponding asymptotic upper bound on the principal rate of decay of $P_y[\cA^\nu_N(F)]$, which was obtained in Corollary 5.11 of~\cite{sznitman2019bulk} when $F$ is bounded, viewing the simple random walk as the ``singular limit as $u \rightarrow 0$'' of random interlacements $\mathcal{I}^u$ at level $u > 0$ and using a coupling argument. In particular, our main result confirms upon combination with the latter that for $\nu \in (0,\vartheta_\infty)$, for bounded $F$, and any $y \in \bbZ^d$,
\begin{equation}
\label{eq:IntroMainResultEquality}
\begin{split}
\lim_{N \rightarrow \infty} & \frac{1}{N^{d-2}} \log P_y[\cA^\nu_N(F)] \\
&= - \inf\bigg\{\frac{1}{2d}\int_{\bbR^d} |\nabla \phi|^2\,\De x\, : \,\phi\in C_0^\infty(\bbR^d)\,,\fint_D \vartheta(\phi^2)\,\De x > \nu\bigg\} \\
& = - \min\bigg\{\frac{1}{2d}\int_{\bbR^d} |\nabla \phi|^2\,\De x\, : \,\phi\in D^{1,2}(\bbR^d)\,,\fint_D \vartheta(\phi^2)\,\De x = \nu\bigg\},
\end{split}
\end{equation}
meaning that the constrained minimization problems on the right-hand side of~\eqref{eq:IntroMainResultLower} (or~\eqref{eq:IntroMainResultEquality}) indeed give the correct principal decay rate for the probability of $\cA^\nu_N(F)$, as discussed in Remark 5.12 of~\cite{sznitman2019bulk}.  This method of obtaining upper bounds on largely deviant events for the random walk upon using random interlacements already appeared in the context of studying \textit{disconnection events} in~\cite{sznitman2015disconnection} and later in~\cite{nitzschner2017solidification,
sznitman2017disconnection} as well as when studying the appearance of large macroscopic holes in the connected component of the vacant set left by the trace of a random walk in~\cite{sznitman2018macroscopic}. Remarkably, finding corresponding lower bounds for the random walk is often not straightforward from results for random interlacements, and relies instead on introducing a well-chosen ``tilting'' of the walk that typically enforces the event under consideration at an (appropriately low) entropic cost. The construction of these \textit{tilted walks}, which will be recalled in detail in Section~\ref{sec:preliminaries} below and plays a key role in this article, was first given in~\cite{li2017lower}. \smallskip  

To provide a specific application of our main result, we consider the choice $F = F_2$ in~\eqref{eq:F_examples_introduction}. In this case, we see that
\begin{equation}
\vartheta(u) \stackrel{\eqref{eq:theta}}{ =} \mathbb{E}\big[\mathbbm{1}_{\{\cL^u_0 > 0 \}} \big] = \bbP[0 \in \mathcal{I}^u ] = 1- \exp\left(-\frac{u}{g(0,0)} \right), \qquad u \geq 0,
\end{equation}
where $g(0,0)$ is the value of the Green function of the simple random walk at two equal sites (see~\eqref{eq:GreenFunction_Def} below), with the second equality following from the explicit characterization of the law of the interlacement set in finite subsets of $\bbZ^d$ (see~\eqref{eq:InterlacementsDefiningProbab} below). Specializing~\eqref{eq:IntroMainResultEquality} to this case, we see that for every $\nu \in (0,1)$ and any $y \in \bbZ^d$,
\begin{equation}
\label{eq:CoveringResultIntro}
\begin{split}
\lim_{N \rightarrow \infty} & \frac{1}{N^{d-2}} \log P_y\big[|D_N \cap \{X_t \,: \, t \geq 0 \}| > \nu |D_N| \big]  \\
&= - \inf\bigg\{\frac{1}{2d}\int_{\bbR^d} |\nabla \phi|^2\,\De x\, : \,\phi\in C_0^\infty(\bbR^d)\,,\fint_D (1 - e^{-\phi^2/g(0,0)})  \,\De x > \nu\bigg\} \\
& = - \min\bigg\{\frac{1}{2d}\int_{\bbR^d} |\nabla \phi|^2\,\De x\, : \,\phi\in D^{1,2}(\bbR^d)\,,\fint_D (1 - e^{-\phi^2/g(0,0)})\,\De x = \nu\bigg\},
\end{split}
\end{equation} 
where we remark that the upper bound follows from~\cite[Corollary 5.11]{sznitman2019bulk}. In~\eqref{eq:CoveringResultIntro}, by providing a matching lower bound we obtain the precise leading-order decay rate for the probability of the~\textit{covering event} under the probability on the left-hand side of~\eqref{eq:CoveringResultIntro}. Similar questions were studied recently in~\cite{asselah2020extracting}, where the authors obtained non-asymptotic upper bounds on related covering-type probabilities. In essence, the method used in the proof of the lower bound in~\eqref{eq:CoveringResultIntro} yields special significance for the model of tilted walks, and the (near-)minimizer (denoted by $\varphi$) to the variational problem above, see Subsection~\ref{subsec:Applications} for more on this. \smallskip

Let us now briefly comment on the proof of~\eqref{eq:IntroMainResultLower}. The fundamental challenge is to introduce an appropriate family of probability measures $\widetilde{P}_{y,N}$ corresponding to the tilted walks. These walks essentially evolve as recurrent walks with a generator given by $\mathscr{L}_N g(x) = \frac{1}{2d}\sum_{|x'-x| = 1} \frac{\varphi_N(x')}{\varphi_N(x)} (g(x') - g(x))$ (with $| \cdot |$ the Euclidean norm), until some deterministic time $S_N$ of order $O(N^d)$, and are then ``released'' to behave as simple random walks, where $\varphi_N = \varphi(\cdot/N)$ and the choice of the function $\varphi$ corresponds to a (near-)minimizer of the variational problem on the right-hand side of~\eqref{eq:IntroMainResultLower}. Notably, the measure $\varphi^2(\cdot/N)$ appears as a reversible measure of the corresponding recurrent walk described above, see also below~\eqref{eq:TiltedWalk}. One then has to argue that this choice of tilted walks renders $\cA^\nu_N(F)$ a ``typical event'', in the sense that $\widetilde{P}_{y,N}[\cA^\nu_N(F)] \rightarrow 1$ as $N$ tends to infinity, and that this is achieved at a (near) minimal entropic cost. The former is done by devising a chain of ``local couplings'' of the occupation times of the tilted walk in mesoscopic boxes of size $M \approx N^{r_1}$ (with $r_1 < 1$) with the occupation times generated by a Poisson process of independent excursions of simple random walks started at the boundaries of these boxes, with an intensity measure proportional to $\varphi^2(\cdot/N)$. Roughly speaking, the occupation time field of the tilted walk in any mesoscopic box of size $M$ will dominate (up to a sufficiently small error probability) the occupation time field of random interlacements at level $\varphi^2(\cdot/N)$. In previous works using a similar approach, see in particular~\cite{li2017lower} and~\cite{sznitman2018macroscopic}, it was enough to show that the traces of these interlacements ``locally create a fence'' around some macroscopic set. In our case however, we need a finer version of such a local coupling attached to an (essentially arbitrary) non-negative  function $\varphi$ with sufficient regularity properties involving also the occupation times. This can be viewed as the central step in this article, and improves on previous couplings of traces of random walks and random interlacements as in~\cite{belius2013gumbel,cerny2016random,li2017lower,teixeira2011fragmentation}. In essence, the chain of couplings constructed in Section~\ref{sec:coupling} provides a \textit{generic mechanism} to show that the occupation time field of tilted walks advantageously dominates that of random interlacements at a locally constant level. We expand on this point in Remark~\ref{rem:LocalCoupling}. 
The fact that random interlacements capture the behavior of the random walks in regimes of ``high density'' is also studied in the recent preprint~\cite{bouchot2024confined}. \smallskip

This article is organized as follows. In Section~\ref{sec:preliminaries}, we introduce further notation and recall useful facts about the simple random walk and tilted walks, potential theory, random interlacements, the change of probability method, and a classical coupling result. In Section~\ref{sec:Variationalproblem}, we establish important properties of the function $\vartheta$ (defined in~\eqref{eq:theta}) in Lemma~\ref{lem:PropertiesTheta} and derive pivotal regularity properties of a (near-)minimizer $\varphi$ attached to the variational problem on the right hand side of~\eqref{eq:IntroMainResultLower}, see Proposition~\ref{prop:quasi-minimizer}. In Section~\ref{sec:estimates}, we state estimates concerning the quasi-stationary distribution of the tilted random walk in Propositions~\ref{prop:ClosenessToQSD} and~\ref{prop:AnalogProp4.7}, which are roughly analogues of the corresponding Propositions 4.5 and 4.7 in~\cite{li2017lower}, with an extension to a more general framework using the (near-)minimizers from the previous section. In Section~\ref{sec:coupling}, we obtain a pivotal series of local couplings of the occupation times of the tilted walk with that of simple random walk excursions (with an intensity measure corresponding essentially to that of random interlacements at a locally constant level). This is done in Propositions~\ref{prop:CouplingIndependentStartedfromSigme} and~\ref{prop:PoissonizationCoupling} and Theorem~\ref{thm:CrucialCoupling}. Finally, in Section~\ref{sec:lowerbound}, we prove our main result in Subsection~\ref{subsec:Maintheorem}, see Theorem~\ref{thm:MainThm}, and give some applications in Subsection~\ref{subsec:Applications}. \smallskip

We conclude the introduction by stating our convention regarding constants. We denote by $C, c, c', \hdots$ positive constants changing from place to place. Numbered constants $c_1, c_2, \hdots$ will
refer to the value corresponding to their first appearance in the text. Dependence on additional parameters is indicated in the notation. All constants may depend implicitly on the
dimension.

\section{Preliminaries}
\label{sec:preliminaries}

In this section we introduce more notation and state some results concerning (continuous-time) random walks, potential theory, random interlacements, as well as some background on tilted walks. We also state a classical inequality involving relative entropy, which is instrumental in obtaining the lower bound in our main Theorem~\ref{thm:MainThm}, as well as a result concerning coupling that will be helpful in Proposition~\ref{prop:CouplingIndependentStartedfromSigme}. Throughout the article, we will always assume that $d \geq 3$. \smallskip

We start with some more notation. We let $\bbN  = \{0,1,2,...\}$ stand for the set of natural numbers. For real numbers $s, t$, we let $s \wedge t$ and $s \vee t$ stand for the minimum and maximum of $s$ and $t$, respectively, and we denote by $\lfloor s \rfloor$ the integer part of $s$, when $s$ is non-negative. Moreover, we let $| \cdot |$ and $| \cdot |_\infty$ stand for the Euclidean and $\ell^\infty$-norms on $\bbR^d$, respectively. For $x \in \bbZ^d$ and $r \geq 0$, we write $B(x,r) = \{y \in \bbZ^d \, : \, |x-y|_\infty \leq r\} \subseteq \bbZ^d$ for the (closed) $\ell^\infty$-ball of radius $r \geq 0$ and center $x \in \bbZ^d$. We denote by $B_r(x)\subseteq\bbR^d$ (resp.~$\overline{B}_r(x)\subseteq\bbR^d$) the Euclidean open ball (resp.~closed ball) of center $x\in \bbR^d$ and radius $r\geq 0$ and set $B_r = B_r(0)$, $\overline{B}_r = \overline{B}_r(0)$. If $x,y \in \bbZ^d$ fulfill $|x - y| = 1$, we call them neighbors and write $x \sim y$. A function $\gamma : \{ 0, \ldots, N \} \rightarrow \bbZ^d$ is called a nearest-neighbor path (of length $N \geq 1$) if $\gamma(i) \sim \gamma(i+1)$ for all $0 \leq i \leq N -1$. For $K \subseteq \bbZ^d$, we let $|K|$ stand for the cardinality of $K$, and we write $K \subset \subset \bbZ^d$ if $|K| < \infty$. Moreover, we write $\partial K = \{ y \in \bbZ^d \setminus K \, : \, y \sim x \text{ for some } x \in K \}$ for the external boundary of $K$, and $\partial_{\mathrm{int}}K = \{y \in K \, : \, y \sim x \text{ for some } x \in \bbZ^d \setminus K \}$ for the internal boundary of $K$. For non-empty $K,L \subseteq \bbZ^d$, we  let $d_\infty(K,L) = \inf\{|x-y|_\infty \, : \, x\in K,y \in L \} $ stand for the $\ell^\infty$-distance between $K$ and $L$, and we also set $d_\infty(x,L) = d_\infty(\{x\},L)$ for $x \in \bbZ^d$. For a set $D \subseteq \bbR^d$ we denote for $\delta > 0$ by $D^\delta$ the closed $\delta$-neighborhood of $D$ and we denote the closure of $D$ by $\overline{D}$. \smallskip 

For a set $U \subseteq \mathbb{Z}^d$, a measure $\beta : U \rightarrow [0,\infty)$, which we identify with its mass function, and functions $f,g : U \rightarrow \bbR$, we define $\langle f, g \rangle_{\ell^2(U,\beta)} = \sum_{x \in U} f(x)g(x)\beta(x)$, whenever the sum converges absolutely. We also write $\|f\|_{\ell^2(U,\beta)} = \langle f,f \rangle_{\ell^2(U,\beta)}^{1/2} = (\sum_{x \in \bbZ^d} f(x)^2 \beta(x) )^{1/2}$, and we call $\ell^2(U,\beta)$ the set of all such functions on $U$ for which $\|f\|_{\ell^2(U,\beta)} < \infty$. When $\beta$ is the counting measure on $U$, we write $\|f\|_{\ell^2(U)}$ and $\ell^2(U)$ as shorthand notation for $\|f\|_{\ell^2(U,\beta)}$ and $\ell^2(U,\beta)$. We also denote by $\mathrm{supp} f = \{x \in \bbZ^d \, : \, f(x) \neq 0 \}$ the support of $f : \bbZ^d \rightarrow \bbR $, and write $f \in C_0(\bbZ^d)$ if $\mathrm{supp} f  \subset \subset \bbZ^d$. For $\varnothing \neq U \subseteq \bbR^d$ open we denote by $C^k(U)$, $k \geq 1$ (resp. ~$C^\infty(U)$) the space of functions that have continuous partial derivatives up to order $k$ (resp.~the space of smooth functions on $U$). The set $C^\infty_0(U)$ is the space of smooth functions on $U$ that vanish outside of a compact set contained in $U$. We write $C^\infty(\overline{U})$ for the subset of functions in $C^\infty(U)$ such that all partial derivatives can be continuously extended to $\overline{U}$. We will use the notation $\nabla q (x) \in \bbR^d$ (resp.~$\Delta q(x) \in \bbR$) for the standard gradient (resp.~Laplace operator) at $x \in U$, for functions $q\in C^1(U)$ (resp.~$q\in C^2(U)$). We denote by $L^k(U)$, $k \geq 1$, the space of functions $q : U \rightarrow \bbR^d$ for which $|q|^k$ has a finite Lebesgue-integral. \smallskip

We now turn to some definitions of path spaces, the continuous-time simple random walk and its potential theory, and continuous-time random interlacements. We let $\widehat{W}_+$ (resp.~$\widehat{W}$) stand for the space of infinite (resp.~bi-infinite) $\bbZ^d \times (0,\infty)$-valued sequences, in which the sequence of first coordinates forms an infinite (resp.~bi-infinite) nearest-neighbor path which spends a finite number of steps in every finite set, and the sequence of the second coordinates has an infinite sum. We let $\widehat{\cW}_+$ stand for the $\sigma$-algebra generated by the first and second coordinate maps, which we denote by $Z_n$ and $\zeta_n$ (with $n \in \bbN$), respectively (the $\sigma$-algebra $\widehat{\mathcal{W}}$ on $\widehat{W}$ is defined analogously). The measure $P_x$ is the law on $(\widehat{W}_+, \widehat{\cW}_+)$ under which the sequence of first coordinates $(Z_n)_{n \geq 0}$ has the law of a discrete-time simple random walk on $\bbZ^d$, starting from $x\in \bbZ^d$, and the sequence of second coordinates $(\zeta_n)_{n \geq 0}$ consists of i.i.d.~exponentially distributed random variables with parameter $1$, independent of $(Z_n)_{n \geq 0}$. We denote the expectation under $P_x$ by $E_x$. If $\beta$ is a measure on $\mathbb{Z}^d$, we denote by $P_\beta$ and $E_\beta$ the measure $\sum_{x \in \mathbb{Z}^d} \beta(x)P_x$ on $(\widehat{W}_+,\widehat{\cW}_+)$ and its corresponding integral. We associate to $\widehat{w} \in \widehat{W}_+$ a continuous-time trajectory $(X_t(\widehat{w}))_{t \geq 0}$, the continuous-time random walk with constant jump rate $1$, by setting
\begin{equation}
\label{eq:DefOfCTSRWalk}
X_t(\widehat{w}) = Z_n(\widehat{w}), \text{ for } t \geq 0, \qquad \text{when }\sum_{\ell = 0}^{n-1} \zeta_\ell(\widehat{w}) \leq t < \sum_{\ell = 0}^n \zeta_\ell(\widehat{w}),
\end{equation}
(where we understand the sum $\sum_{\ell = 0}^{n-1} \zeta_\ell(\widehat{w})$ as zero, when $n = 0$). Given a subset $K \subseteq \bbZ^d$ and a trajectory $\widehat{w} \in \widehat{W}_+$, we write $H_K(\widehat{w}) = \inf\{t \geq 0 \, : \, X_t(\widehat{w}) \in K\}$, $\widetilde{H}_K(\widehat{w}) = \inf\{t \geq \zeta_0 \, : \, X_t(\widehat{w}) \in K\}$, and $T_K(\widehat{w}) = \inf\{t \geq 0\, : \, X_t(\widehat{w}) \notin K\}$ for the entrance, hitting, and exit times of $K$, respectively (where we understand $\inf \varnothing = \infty$), and write $H_x$, $\widetilde{H}_x$, and $T_x$ as a shorthand notation for $H_{\{x\}}$, $\widetilde{H}_{\{x\}}$ and $T_{\{x\}}$, respectively, when $x \in \bbZ^d$.  We let $\Gamma(K)$ stand for the set of right-continuous, piecewise constant functions from $[0,\infty)$ to $K$ that have finitely many jumps in every compact interval. For $A \subseteq [0,\infty)$, we define $X_A = \{ X_t \, : \, t \in A \}$. The occupation-time field of the random walk is denoted by $(L_x)_{x \in \mathbb{Z}^d}$, where 
\begin{equation}
\label{eq:OccupationTimeField}
L_x(\widehat{w}) = \sum_{n \in \bbN} \zeta_n(\widehat{w}) \mathbbm{1}_{\{Z_n(\widehat{w}) = x\}}, \qquad x \in \mathbb{Z}^d,
\end{equation}
records the total time spent in $x$ by the continuous-time random walk.
 \smallskip

We now recall some facts concerning potential theory associated with the simple random walk. For $ \varnothing \neq K \subset\subset \bbZ^d$, we let
\begin{equation}
\label{eq:eqmeasureDef}
e_K(x) = P_x[\widetilde{H}_K = \infty]\mathbbm{1}_K(x), \qquad x \in \bbZ^d,
\end{equation}
stand for the equilibrium measure of the set $K$, and we call its total mass
\begin{equation}
\label{eq:capacityDef}
\capa(K) = \sum_{x \in K} e_K(x),
\end{equation}
the capacity of $K$. We also define the normalized equilibrium measure of $K$ by
\begin{equation}
\label{eq:normalized_eqmeasureDef}
\widetilde{e}_K(x) = \frac{e_K(x)}{\capa(K)}, \qquad x \in K.
\end{equation}
Both $e_K$ and $\widetilde{e}_K$ are supported on $\partial_{\mathrm{int}} K$. Asymptotics of the capacity and equilibrium measure on boxes are well-known, and we will use the bounds (see, e.g.,~\cite[Section 2.2]{lawler2013intersections})
\begin{equation}
\label{eq:capacityControlsBox}
cN^{d-2} \leq \capa(B(0,N)) \leq C N^{d-2}, \qquad  N \geq 1,
\end{equation}
and 
\begin{equation}
\label{eq:LowerBoundEqMeasure}
e_{B(0,N)}(x) \geq \frac{c}{N}, \qquad x \in \partial_{\mathrm{int}}B(0,N), \qquad N \geq 1.
\end{equation}
Moreover, we denote the Green function of the simple random walk by $g(\cdot,\cdot)$, namely
\begin{equation}
\label{eq:GreenFunction_Def}
g(x,y) = E_x\left[\int_0^\infty \mathbbm{1}_{\{X_t =y \}} \mathrm{d}t \right], \qquad x, y\in \mathbb{Z}^d,
\end{equation}
which is  symmetric, non-negative and finite (due to transience), and we record the classical formula for $\varnothing \neq A \subset\subset \bbZ^d$,
\begin{equation}
\label{eq:LastExitDecomposition}
P_x[H_A < \infty] = \sum_{y \in A} g(x,y) e_A(y), \qquad x \in \bbZ^d;
\end{equation}
see, e.g.,~\cite[Lemma 2.1.1]{lawler2013intersections}. Furthermore, one has the following asymptotic behavior of the Green function,
\begin{equation}
\label{eq:Greenfunction_asymptotics}
g(x,y) \sim \frac{c_g}{|x-y|^{d-2}}, \qquad \text{as }|x-y| \rightarrow \infty,
\end{equation}
with the constant $c_g = \frac{d}{2}\Gamma(d/2-1)\pi^{-d/2}$ (see~\cite[Theorem 1.5.4]{lawler2013intersections}). For a function $h : \bbZ^d \rightarrow \bbR$, we denote its discrete Laplacian at $x \in \bbZ^d$ by
\begin{equation}
\Delta_{\bbZ^d}h(x) = \frac{1}{2d} \sum_{y \, : \, y \sim x} \big(h(y) - h(x)\big).
\end{equation}
For later use, we define the discrete Dirichlet form of $h : \bbZ^d \rightarrow \bbR$ by
\begin{equation}
\label{eq:DiscreteDirichletFormDef}
\mathcal{E}_{\bbZ^d}(h,h) = \frac{1}{2}\sum_{x,y \in \bbZ^d, x \sim y} \frac{1}{2d} \big(h(y) - h(x)\big)^2,
\end{equation}
(which may be infinite for general $h$, but is finite if $h \in C_0(\bbZ^d)$). We note that for $h \in C_0(\bbZ^d)$, one has the discrete Gauss-Green identity (see, e.g.,~\cite[Theorem 1.24]{barlow2017random})
\begin{equation}
\label{eq:DiscreteGaussGreen}
\mathcal{E}_{\bbZ^d}(h,h) = - \sum_{x \in \bbZ^d} h(x) \Delta_{\bbZ^d} h(x).
\end{equation}
We also define the Dirichlet form associated with Brownian motion, which is the continuum counterpart of~\eqref{eq:DiscreteDirichletFormDef}, defined for functions $g \in H^1(\bbR^d)$ (the usual Sobolev space over $\bbR^d$) by 
\begin{equation}
\mathcal{E}_{\bbR^d}(g,g) = \frac{1}{2}\int_{\bbR^d} |\nabla g(x)|^2 \De x.
\end{equation}

The capacity of a set  $\varnothing \neq K \subset \subset \bbZ^d$ can be expressed as (see, e.g.,~\cite[Proposition 7.9]{barlow2017random})
\begin{equation}
\label{eq:CapaAlternativeDef}
\capa(K) = \inf\left\{ \mathcal{E}_{\bbZ^d}(h,h) \, : \, h \in C_0(\bbZ^d), \ h(x) = 1 \text{ for all } x \in K  \right\}.
\end{equation}
We now turn to the definition of the tilted random walk, introduced in~\cite{li2017lower}. To that end, let $R > 0$ be an integer and $U^N = (NB_R) \cap \bbZ^d$. For a function $f : \bbZ^d \rightarrow [0, \infty)$ fulfilling
\begin{equation}
\label{eq:Conditions_f}
\begin{minipage}{0.8\linewidth}
\text{i) $f(x) > 0$ if and only if $x \in U^N$,} \\
\text{ii) $f^2$ defines a probability measure, i.e.~$\sum_{x \in \bbZ^d} f^2(x) = 1$}
\end{minipage}
\end{equation}
(both $R$ and $f$ will be chosen later appropriately), we can then consider the stochastic process 
 \begin{equation}
 \label{eq:MartingaleDef}
 M_t = \frac{f(X_{t \wedge T_N})}{f(x)} \exp\left( \int_0^{t \wedge T_N} v(X_s) \mathrm{d}s \right), \qquad t \geq 0,
 \end{equation}
 where $T_N = T_{U^N}$ and 
 \begin{equation}
 \label{eq:vDef}
 v(x) = - \frac{\Delta_{\bbZ^d}f(x)}{f(x)}, \qquad x \in U^N.
 \end{equation}
For any given $T \geq 0$, we can then define the non-negative measure on $(\widehat{W}_+,\widehat{\cW}_+)$ given by
\begin{equation}
\label{eq:MeasureUpToT}
\widehat{P}_{x,T} = M_T P_x, \qquad x \in U^N
\end{equation}
(i.e.~$M_T$ is the Radon-Nikod{\'y}m density of $\widehat{P}_{x,T}$ with respect to $P_x$). It can be shown using classical methods (see, e.g.,~\cite[Lemma 2.5]{li2017lower} and the references therein) that $\widehat{P}_{x,T}$ is a probability measure and under $\widehat{P}_{x,T}$, the process $(X_t)_{t \geq 0}$ up to time $T$ is a Markov chain 
with generator
\begin{align}
\label{eq:GeneratorVSRW}
    \mathscr{L}_f g(x) &=  \frac{1}{2d} \sum_{y \in U^N \, : \, y \sim x} \frac{f(y)}{f(x)}(g(y) - g(x)),\qquad x\in U^N.
\end{align}
The measure $\pi$ on $U^N$, defined by
 \begin{equation}
 \label{eq:piDef}
\pi(x) = f^2(x), \qquad x\in U^N,
\end{equation}
is a reversible measure on $U^N$ of a Markov chain $(\overline{P}_x)_{x \in U^N}$ with generator~\eqref{eq:GeneratorVSRW}. This Markov chain is called the \textit{confined walk}. The laws of the tilted walks are then defined by
\begin{equation}
\label{eq:TiltedWalk}
\widetilde{P}_{y,N}  = \widehat{P}_{y, S_N}, \qquad y \in U^N, N \geq 1,
\end{equation}
with an increasing sequence of positive real numbers $(S_N)_{N \geq 1}$. The times $S_N$ will later be chosen to be approximately $ \|\varphi(\frac{\cdot}{N})  \|_{\ell^2(\bbZ^d)}^2$, where $\varphi$ is a (near-)minimizer of the variational problem in the right-hand side of~\eqref{eq:IntroMainResultLower} constructed in Proposition~\ref{prop:quasi-minimizer}, and therefore will turn out to be of order $O(N^d)$. The function $f$ appearing above will be chosen as $\varphi(\frac{\cdot}{N}) /\|\varphi(\frac{\cdot}{N})  \|_{\ell^2(\bbZ^d)}$. We refer to~\eqref{eq:varphi_N_Def} for the specific choices. Intuitively, one can view the tilted walk as a continuous-time  random walk on $U^N$ with variable speed, whose discrete skeleton is a random walk on $U^N$ (with the nearest-neighbor structure inherited from $\bbZ^d$) equipped with conductances $\left(\frac{1}{2d}\varphi(\frac{x}{N})\varphi(\frac{y}{N}) \right)_{x,y \in U^N, x \sim y}$ and with a variable jump rate, up to time $S_N$, which is ``released'' after $S_N$ and moves like a simple random walk afterwards (see also Remark~\ref{rem:CommentOnConductances}, as well as~\cite[Remark 2.7]{li2017lower} and the discussion below~\cite[(1.49)]{li2014lowerBound}). Note that for every $y \in U^N$, we have that 
\begin{equation}
\label{eq:ConfinedTiltedCoincide}
\text{up to time $S_N$, $\widetilde{P}_{y,N}$ coincides with $\overline{P}_y$,}
\end{equation}
which follows by considering the finite-time marginals and by observing that both processes have the same generator.
\smallskip

 We now introduce some notation concerning continuous-time random interlacements.  The precise definition of the process is not relevant for our application, but can for instance be found in Section 1 of~\cite{sznitman2012isomorphism} or Section 2 of~\cite{chiarini2020entropic}. We also refer to~\cite{cerny2012random, drewitz2014introduction} for more details on (discrete-time) random interlacements. In the construction, one considers the space of bi-infinite trajectories under time-shift, denoted by $\widehat{W}^\ast$, that is $\widehat{W}^\ast = \widehat{W} / \sim$ where $\widehat{w} \sim \widehat{w}'$ if there is a $k \in \bbZ$ such that $\widehat{w} = \widehat{w}'(\cdot + k)$. Moreover, we denote by $\pi^\ast : \widehat{W} \rightarrow \widehat{W}^\ast$ the canonical projection and endow $\widehat{W}^\ast$ with the push-forward $\sigma$-algebra of $\widehat{\cW}$ under $\pi^\ast$. We then define the continuous-time random interlacements as a Poisson point process, defined on some canonical probability space $(\Omega,\mathcal{A},\mathbb{P})$ (with the corresponding expectation denoted by $\mathbb{E}$) with values on $\widehat{W}^\ast\times \bbR_+$, and intensity measure $\widehat{\nu}(\De \widehat{w}^*)\otimes\De u$, where $\widehat{\nu}$ is a certain $\sigma$-finite measure (see~\cite[(1.7)]{sznitman2012isomorphism}). For a realization $\omega=\sum_{i\geq 0} \delta_{(\widehat{w}^*_i, u_i)}$ of this process and $u\geq 0$, we define the random interlacement set at level $u$ as the random subset of $\bbZ^d$ given by
\begin{equation}
    \cI^u(\omega) = \bigcup_{i: u_i\leq u} \mathrm{Range}(\widehat{w}^*_i),
\end{equation}
where for $\widehat{w}^*\in \widehat{W}^*$, $\mathrm{Range}(\widehat{w}^*)$ denotes the set of points in $\bbZ^d$ visited by the first coordinate sequence associated with an arbitrary $\widehat{w}\in\widehat{W}$ such that $\pi^*(\widehat{w}) = \widehat{w}^*$. Note that for any $\varnothing \neq K \subset\subset \bbZ^d$ and $u \geq 0$, we have (see, e.g.,~\cite[Proposition 1.5]{sznitman2010vacant})
\begin{equation}
\label{eq:InterlacementsDefiningProbab}
\mathbb{P}[\cI^u \cap K = \varnothing] = e^{-u\capa(K)}.
\end{equation}
We then define the occupation time at site $x$ and level $u$ of random interlacements, denoted by $\mathcal{L}_{x}^u(\omega)$, as the total time spent at $x$ by all trajectories $\widehat{w}_i^*$ with label $u_i\leq u$ in the cloud $\omega = \sum_{i\geq 0} \delta_{(\widehat{w}_i^*, u_i)}$. Formally, we write:
\begin{equation}
\label{eq:OccTimeDef}
\begin{split}
& \mathcal{L}_{x}^u(\omega) = \sum_{i \geq 0} \sum_{n\in \bbZ} \zeta_n(\widehat{w}_i) \mathbbm{1}_{\{ Z_n(\widehat{w}_i) = x, u_i \leq u \}}, \text{ for } x\in \bbZ^d, u \geq 0, \\
& \text{for }\omega  = \sum_{i\geq 0} \delta_{(\widehat{w}_i^*, u_i)} \in \Omega, \text{ and } \pi^\ast(\widehat{w}_i) = \widehat{w}_i^\ast \text{ for any } i \geq 0.
\end{split}
\end{equation}
It follows from the definition of $\mathcal{L}^u$ and elementary properties of a Poisson point process that under $\mathbb{P}$, $(\mathcal{L}^{u+h}_x - \mathcal{L}^u_x)_{x \in \mathbb{Z}^d}$ is independent of $(\mathcal{L}^u_x)_{x \in \mathbb{Z}^d}$ and the former has the same law as $\mathcal{L}^h$, for every $u \geq 0$, $h > 0$. We also record that 
\begin{equation}
\label{eq:OccupationTimeExpectation}
\bbE[\mathcal{L}^u_x] = u, \qquad x \in \bbZ^d, u \geq 0.
\end{equation}
In the proof of the main result in Section~\ref{sec:lowerbound}, we will make use of a classical change of probability method using the notion of relative entropy. Let $\widetilde{Q}$ and $Q$ be two probability measures on some measurable space $(\mathcal{O}_0,\mathcal{F}_0)$ such that $\widetilde{Q}$ is absolutely continuous with respect to $Q$. The relative entropy of $\widetilde{Q}$ with respect to $Q$ is then defined as
\begin{equation}
\label{eq:RelEntropy}
\cH(\widetilde{Q}|Q) = E_{\widetilde{Q}}\left[\log \frac{\mathrm{d}\widetilde{Q}}{\mathrm{d}Q} \right] = E_{Q}\left[ \frac{\mathrm{d}\widetilde{Q}}{\mathrm{d}Q}\log \frac{\mathrm{d}\widetilde{Q}}{\mathrm{d}Q} \right] \in [0,\infty],
\end{equation}
 where $\frac{\mathrm{d}\widetilde{Q}}{\mathrm{d}Q}$ is the Radon-Nikod{\'y}m derivative of $\widetilde{Q}$ with respect to $Q$ and $E_{\widetilde{Q}}$ and $E_Q$ stand for the expectations under $\widetilde{Q}$ and $Q$, respectively. Then, for any event $A \in \mathcal{F}_0$ with $\widetilde{Q}[A] > 0$, we have
 \begin{equation}
 \label{eq:EntropyIneq}
 \log Q[A] \geq \log \widetilde{Q}[A] - \frac{1}{\widetilde{Q}[A]} \Big\{\cH(\widetilde{Q}|Q)+\frac{1}{\rm{e}}\Big\}, 
 \end{equation}
 see, e.g.,~\cite[p.~76]{deuschel2001large}. \smallskip
 
 Finally, we will need to utilize a standard device relating the total variation distance of probability measures to couplings (which will be instrumental in the proof of Proposition~\ref{prop:CouplingIndependentStartedfromSigme}). Let $Q_1, Q_2$ be two probability measures on some measurable space $(\widehat{\Omega},2^{\widehat{\Omega}})$ with $\widehat{\Omega}$ finite (and $2^{\widehat{\Omega}}$ denoting the power set of $\widehat{\Omega}$), then
 \begin{equation}
 \label{eq:CouplingTV}
 \begin{split}
 \|Q_1 - Q_2\|_{\mathrm{TV}} & \stackrel{\text{def}}{=} \max_{A \subseteq \widehat{\Omega} } \big|Q_1[A] - Q_2[A]\big| \\ 
 & = \frac{1}{2}\sum_{x \in \widehat{\Omega}} |Q_1(x) - Q_2(x)| \\
 & = \inf\Big\{ Q[X_1 \neq X_2] \, : \, (X_1,X_2) \text{ is a coupling of $Q_1$ and $Q_2$} \Big\},
\end{split}
 \end{equation}
where a coupling consists of a probability measure $Q$ on some measurable space $(\widetilde{\Omega}, \widetilde{\mathcal{G}})$ and random variables $X_1,X_2$ with values in $\widehat{\Omega}$ such that the laws of $X_1$ and $X_2$ under $Q$ are $Q_1$ and $Q_2$, respectively (see~\cite[Proposition 4.7, p.~50]{levin2017markov}).
\section{On the variational problem}
\label{sec:Variationalproblem}

The main aim of this section is to analyze the variational problem on the right-hand side of~\eqref{eq:IntroMainResultLower}, which is stated as~\eqref{eq:variational_problem} below, and to construct a smooth, compactly supported quasi-minimizer for it in Proposition~\ref{prop:quasi-minimizer}. We start by introducing some conditions on the local function $F :  [0,\infty)^{B(0,\mathfrak{r})} \to [0,\infty)$ and by recalling the definition of the function $\vartheta$ from~\eqref{eq:theta}. We will then prove some regularity properties of the latter that we will need in the sequel.
\medskip

Throughout the remainder of this article, the following assumption will be in force.
\begin{assumption}
\label{eq:RegularityCondF}
There exists an integer $\mathfrak{r} \geq 0$ and a measurable function $F :  [0,\infty)^{B(0,\mathfrak{r})} \to [0,\infty)$ such that
\begin{equation*}
\begin{minipage}{0.8\linewidth}
\text{i) $F(0) = 0$, and $F$ is not identically equal to $0$;}

\text{ii) $F$ is non-decreasing in each of its arguments;} 

\text{iii) There exists a constant $c_1 = c_1(F) >0$ such that for all $\ell$, $\ell'\in [0,\infty)^{B(0,\mathfrak{r})}$,}

$$F(\ell + \ell') \leq F(\ell) + c_1\bigg(\IND_{\{\ell'\neq 0\}} + \sum_{|x|_\infty\leq \mathfrak{r}} \ell'_x\bigg).$$
\end{minipage}
\end{equation*}
\end{assumption}
\noindent
We observe that Assumption~\ref{eq:RegularityCondF}, iii), is automatically satisfied when $F$ is bounded. \smallskip

We recall the definition of $\vartheta: [0,\infty)\to[0,\infty)$ from~\eqref{eq:theta},
\begin{equation}\label{eq:def theta}
    \vartheta(u) = \bbE\Big[F\Big((\mathcal{L}^u_{y})_{y \in B(0,\mathfrak{r})}\Big) \Big],
\end{equation}
where $F$ fulfills Assumption~\ref{eq:RegularityCondF}.

\begin{lemma}
\label{lem:PropertiesTheta}
    The function $\vartheta$ defined by~\eqref{eq:def theta} with $F$ satisfying i), ii), iii) in Assumption~\ref{eq:RegularityCondF} has the following properties:
    \begin{itemize}
        \item[i)] $\vartheta(0) = 0$;
        \item[ii)]$\vartheta$ is %non-decreasing, 
        Lipschitz continuous, and there exists a constant $c_2 = c_2(F,\mathfrak{r})>0$ such that for all $0\leq u\leq u'$ 
        \begin{equation}
            \vartheta(u'-u) e^{-u \capa(B(0,\mathfrak{r}))} \leq \vartheta(u')-\vartheta(u) \leq c_2(u'-u);
        \end{equation}
        \item[iii)] $\vartheta$ is strictly increasing on $[0,\infty)$. 
    \end{itemize}
\end{lemma}
\begin{proof}
    The proof of i) is immediate from the Assumption~\ref{eq:RegularityCondF}, i) for $F$. We now show that ii) holds. Let $u \in [0,\infty)$ and $h>0$. %First we observe that since $F$ is non-decreasing in each of its arguments by Assumption~\ref{eq:RegularityCondF}, ii):
%    \begin{equation}
%        \vartheta(u+h) = \bbE\Big[F\Big((\mathcal{L}^{u+h}_{y})_{y \in B(0,\mathfrak{r})}\Big) \Big] \geq \bbE\Big[F\Big((\mathcal{L}^{u}_{y})_{y \in B(0,\mathfrak{r})}\Big) \Big] = \vartheta(u),
%    \end{equation}
%    showing that $\vartheta$ is non-decreasing.
    From Assumption~\ref{eq:RegularityCondF}, iii) on $F$ and the fact that $\mathcal{L}^{u+h} = \mathcal{L}^{u} + \widehat{\mathcal{L}}^{h}$ with $\widehat{\mathcal{L}}^{h}$ equal in law to $\mathcal{L}^{h}$ and independent of $\mathcal{L}^u$ (see below~\eqref{eq:OccTimeDef}), we have
    \begin{equation}
        \begin{aligned}
            \vartheta(u+h) & = \bbE\Big[F\Big((\mathcal{L}^u_{y})_{y \in B(0,\mathfrak{r})} + (\widehat{\mathcal{L}}^h_{y})_{y \in B(0,\mathfrak{r})}\Big) \Big]\\
            & \leq \bbE\Big[F\Big((\mathcal{L}^u_{y})_{y \in B(0,\mathfrak{r})}\Big) \Big] + c_1 \Big( \bbP[\cI^h\cap B(0,\mathfrak{r}) \neq \varnothing] +  |B(0,\mathfrak{r})| \bbE[\cL_0^{h}]\Big)\\
            & \stackrel{\eqref{eq:InterlacementsDefiningProbab},\eqref{eq:OccupationTimeExpectation}}{=} \vartheta(u) +  c_1\Big( 1-\exp(-h \capa(B(0,\mathfrak{r})))+ |B(0,\mathfrak{r})|h\Big).
        \end{aligned}
    \end{equation}
    Owing to the inequality $1-e^{-x} \leq x$ for all $x\geq 0$ we thus obtain
    \begin{equation}
        \vartheta(u+h) - \vartheta(u) \leq c_1 (\capa(B(0,\mathfrak{r})) + |B(0,\mathfrak{r})|) h  \stackrel{\text{def}}{=}c_2h.
    \end{equation}
    For the lower bound we observe that for $u \in [0,\infty)$ and $h>0$
    \begin{equation}
        \begin{aligned}
        \vartheta(u+h) &= \bbE\Big[F\Big((\mathcal{L}^{u+h}_{y})_{y \in B(0,\mathfrak{r})}\Big) \Big]\\
        &\geq \vartheta(u) + \bbE\Big[F\Big((\mathcal{L}^{u+h}_{y})_{y \in B(0,\mathfrak{r})}\Big);\, \cL^u_y = 0, \forall y \in B(0,\mathfrak{r}) \Big]\\
        & = \vartheta(u) + \vartheta(h) \bbP[\cI^u\cap B(0,\mathfrak{r}) = \varnothing],
        \end{aligned}
    \end{equation}
and the lower bound follows again upon using~\eqref{eq:InterlacementsDefiningProbab}.    
    
    We now proceed with the proof of iii). Since $F$ is non-decreasing in all of its arguments, and we additionally assume that it is not identically zero, then there must be some $a,b >0$ such that if $\ell_y \geq a$ for all $y\in B(0,\mathfrak{r})$, then $F(\ell) \geq b$. With this in mind, we consider an arbitrary $h>0$ and note that
    \begin{equation}\label{eq:estimate on theta}
        \begin{aligned}
            \vartheta(h) &= \bbE\Big[F\Big((\mathcal{L}^{h}_{y})_{y \in B(0,\mathfrak{r})}\Big) \Big] \geq b \bbP[\cL^h_y \geq a,\forall y\in B(0,\mathfrak{r})] > 0.
        \end{aligned}
    \end{equation}
    Since $h>0$ above is arbitrary,~\eqref{eq:estimate on theta} together with ii) imply that $\vartheta$ is strictly increasing.
\end{proof}

By Lemma~\ref{lem:PropertiesTheta}, we see that $\vartheta$ is strictly increasing and therefore
\begin{equation}
\vartheta_\infty = \lim_{u \rightarrow \infty} \vartheta(u) \in (0,\infty]
\end{equation}
exists. Moreover, if $F$ is bounded, then $\vartheta_\infty < \infty$. In the remainder of this section, we let $\nu\in (0,\vartheta_\infty)$ be fixed and study the following variational problem,
\begin{equation}\label{eq:variational_problem}
        \underset{\phi \in D^{1,2}(\bbR^d)}{\text{minimize}} \quad\frac{1}{2d} \int_{\bbR^d} |\nabla \phi|^2\,\De x,\qquad\text{subject to }  \fint_D \vartheta(\phi^2)\,\De x \geq \nu,
\end{equation}
where $D^{1,2}(\bbR^d)$ is the space of  locally integrable functions $u$ on $\bbR^d$ with finite Dirichlet integral $\int_{\mathbb{R}^d} |\nabla u|^2 \mathrm{d}x < \infty$ and vanishing at infinity in the sense that $|\{ |u| > a \}| < \infty$ for all $a > 0$ (with $|\, \cdot \,|$ denoting the Lebesgue measure), see~\cite[Chapter 8 $\S$2]{lieb2001analysis}. Owing to the fact that $\vartheta$ is Lipschitz continuous, we have that the map $\phi\in L^2(D) \mapsto \fint_D \vartheta(\phi^2)\,\De x \in [0,\infty)$ is continuous. By means of Theorem 8.6 of~\cite{lieb2001analysis} and the lower semi-continuity of the Dirichlet energy it is standard to show that there exists a minimizer $\varphi_\mathrm{min} \in D^{1,2}(\bbR^d)$ to the variational problem in~\eqref{eq:variational_problem}, such that (recall that $D$ is defined above~\eqref{eq:DiscreteBlowup})
\begin{equation}
\label{eq:I_D(nu)Def}
   \cI_D(\nu)\stackrel{\mathrm{def}}{=}   \frac{1}{2d}\int_{\bbR^d} |\nabla \varphi_{\mathrm{min}}|^2\,\De x , \qquad \fint_D \vartheta(\varphi_{\mathrm{min}}^2)\,\De x = \nu.
\end{equation}

\begin{remark}
\label{rem:ChooseMinimizerNonnegative} We note that the operation $\varphi_{\mathrm{min}}\mapsto |\varphi_{\mathrm{min}}|$ only reduces the Dirichlet energy without violating the constraint in~\eqref{eq:variational_problem}. It follows that $\varphi_{\mathrm{min}}$ can be chosen to be non-negative. Furthermore,  we can see that $\varphi_{\min}$ is harmonic in $\bbR^d\setminus D$. In fact, by considering perturbations $\varphi_{\min} + \varepsilon \psi $ with $\psi\in C_0^\infty(\mathbb{R}^d\setminus D$), one obtains $\Delta \varphi_{\min} = 0$ on $\bbR^d\setminus D$ in a weak sense, and a fortiori, by classical elliptic regularity, in a classical sense. The harmonicity together with the fact that $\varphi_{\min} \in D^{1,2}(\mathbb{R}^d)$ and the mean value property of harmonic functions, implies that $\sup_{|x| = R} \varphi_{\min}(x)$ goes to zero as $R$ tends to infinity. In particular, by the maximum principle, and setting $r_D = \sup\{|x|:\, x\in D\}$, one has for all $|x|\geq 2r_D$
\begin{equation}\label{eq:decay}
      \bigg(\frac{2r_D}{|x|}\bigg)^{d-2} \min_{|z| = 2r_D} \varphi_{\min} \leq \varphi_{\min} (x) \leq   \bigg(\frac{2r_D}{|x|}\bigg)^{d-2} \max_{|z| = 2r_D} \varphi_{\min}.
\end{equation}
It follows that when $d \geq 5$, $\varphi_{\min} \in L^2(\bbR^d)$.
\end{remark}

\begin{remark} It is not hard to show that for any $\nu\in [0,\infty)$
\begin{equation}
    \label{eq:smooth_approx}
    \cI_D(\nu) = \inf\bigg\{\frac{1}{2d}\int_{\bbR^d} |\nabla \phi|^2 \,\De x\, : \, \phi\in C_0^\infty(\bbR^d)\,,\fint_D \vartheta(\phi^2)\,\De x > \nu\bigg\},
\end{equation}
with the convention that $\inf \varnothing = \infty$ (see for instance~\cite[Corollary 5.9]{sznitman2019bulk}). 
\end{remark}

\begin{remark} We note that since $\vartheta$ is Lipschitz continuous, it is differentiable almost everywhere and sub-linear. In many interesting situations $\vartheta$ is in fact smooth even if $F$ is discontinuous. As an example consider $F_2(\ell) = \mathbbm{1}_{\{\ell > 0 \} }$, ($\mathfrak{r} =0$), for which $\vartheta(u) = 1 - \exp(-u / g(0,0))$. In the case where the function $\eta$, with $\eta(a) \stackrel{\mathrm{def}}{=} \vartheta(a^2)$ for $a \in \bbR$, fulfills $\eta \in C^2_b(\bbR)$ (i.e.~the subspace of functions in $C^2(\bbR)$ which have bounded derivatives) we can observe as in~\cite[Remark 5.10, 3)]{sznitman2019bulk} that non-negative minimizers $\varphi_{\min}\in D^{1,2}(\bbR^d)$ of~\eqref{eq:variational_problem} satisfy the following semilinear partial differential equation in the weak sense
\begin{equation}
    -\Delta \varphi_{\min} = \lambda \varphi_{\min} \vartheta'(\varphi_{\min}^2) \IND_D,
\end{equation}
for a suitable Lagrange multiplier $\lambda$. In particular, in view of~\eqref{eq:decay} we can even say that a minimizer satisfies
\begin{equation}
    \varphi_{\min} = \lambda \cG \Big[ \varphi_{\min} \vartheta'(\varphi_{\min}^2) \IND_D\Big],
\end{equation}
for a suitable $\lambda>0$, where $\cG = (-\Delta)^{-1}$ is the convolution with the Green function associated with a Brownian motion. 
\end{remark}

\begin{prop}\label{prop:continuityofI} The map $\nu\in [0,\vartheta_\infty)\mapsto \cI_D(\nu) \in [0,\infty)$ is an increasing homeomorphism.    
\end{prop}
\begin{proof} It is immediate that the map is non-decreasing and right-continuous. Furthermore it is straightforward that $\cI_D(0) = 0$. We now argue that the map is left-continuous. Let $\nu\in (0,\vartheta_\infty)$ be fixed and consider a sequence $(\varepsilon_n)_{n \geq 1}$ of positive real numbers with $\varepsilon_n \to 0$ and $\varphi_n\geq 0$, $\varphi_n\in C_0^\infty(\bbR^d)$ such that $\fint_D \vartheta(\varphi_n^2)\,\De x > \nu-\varepsilon_n $ with $\frac{1}{2d}\int_{\bbR^d}|\nabla\varphi_n|^2\De x \to \lim_{\varepsilon \downarrow 0} \cI_D(\nu-\varepsilon) \leq \cI_D(\nu) < \infty$. Since $\nabla \varphi_n$ is bounded in $L^2$, we can extract a weakly convergent subsequence, and thus by Theorem 8.6 of~\cite{lieb2001analysis}, we can extract a subsequence of $\varphi_n$ that converges a.e.\ and in $L^2_{\mathrm{loc}}(\bbR^d)$ to a function $\phi\in D^{1,2}(\bbR^d)$. By the lower semi-continuity of the Dirichlet energy we have $\frac{1}{2d}\int_{\bbR^d} |\nabla\phi|^2\,\De x \leq \lim_{\varepsilon \downarrow 0} \cI_D(\nu-\varepsilon)$, and furthermore by the Lebesgue dominated convergence theorem and the Lipschitz continuity of $\vartheta$ we have 
$\fint_D \vartheta(\phi^2)\,\De x \geq \nu$. It follows that $\cI_D(\nu) \leq \lim_{\varepsilon \downarrow 0} \cI_D(\nu-\varepsilon)$ and the left-continuity follows.

We now show that $\cI_D(\nu)$ is strictly increasing. Let $\nu<\nu'$ and $\varphi'$ a minimizer for $\cI_D(\nu')$. Then $\varphi'$ is not the null function so that for some $a\in [0,1)$ we have
\begin{equation}
    \fint_D \vartheta((a\varphi')^2)\,\De x \geq \nu,\qquad \Longrightarrow \qquad
    \cI_D(\nu) \leq \frac{a^2}{2d} \int_{\bbR^d} |\nabla \varphi'|^2\,\De x < \cI_D(\nu').
\end{equation}
We finish the proof by showing that $\lim_{\nu\uparrow \vartheta_\infty} \cI_D(\nu) = \infty$. We start with the case $\vartheta_\infty < \infty$. As in the proof of the left-continuity, if $\lim_{\nu\uparrow\vartheta_{\infty}}\cI_D(\nu)$ would be finite we could find a function $\phi\in D^{1,2}(\bbR^d)$, $\phi\geq 0$ such that $\fint_D\vartheta(\phi^2)\,\De x = \vartheta_\infty$. This is impossible since $\vartheta(\phi^2) < \vartheta_\infty$ a.e.\ in $D$. Finally, if $\vartheta_\infty = \infty$, the same argument as before, but choosing a sequence $\nu_n\uparrow \infty$, would allow us to find $\phi\in D^{1,2}(\bbR^d)$, $\phi\geq 0$ such that $\fint_D\vartheta(\phi^2)\,\De x = \infty$, this is impossible since it would imply, owing to the Lipschitz continuity of $\vartheta$, that $\phi\notin L^2(D)$.
\end{proof} 

Our main purpose in Proposition~\ref{prop:quasi-minimizer} below is to construct a quasi-minimizer supported on a ball of a large enough radius that is smooth. We start with a lemma that allows us to restrict the minimization problem to a ball of a large enough radius. In the following, we denote for an open set $U \subseteq \bbR^d$ by $H^1_0(U)$ the closure of $C^\infty_0(U)$ in the standard $H^1$ (Sobolev) norm; see, e.g.,~\cite[p.~174]{lieb2001analysis}.

\begin{lemma}\label{lem:min_r} Consider for $\nu\in (0,\vartheta_\infty)$ and any $r>r_D \stackrel{\mathrm{def}}{=} \sup \{|x|:\, x\in D\}$ the minimization problem
\begin{equation}
    \label{eq:variational_problem_r}
        \underset{\phi \in H^{1}_0(B_r)}{\text{minimize}} \quad\frac{1}{2d} \int_{B_r} |\nabla \phi|^2\,\De x,\qquad\text{subject to }  \fint_D \vartheta(\phi^2)\,\De x \geq \nu,
\end{equation}
and denote by $\cI_{D,r}(\nu)$ the minimum of the above problem.
Then,
    \begin{equation}
        \cI_{D,r}(\nu)\downarrow \cI_D(\nu),\qquad \text{as $r\uparrow\infty$}.
    \end{equation}
Furthermore, there exists $\varphi\in H_0^1(B_{r})$, with $\varphi\geq0$ and $\fint_D \vartheta(\varphi^2) \,\De x = \nu$, such that $\cI_{D,r}(\nu) = \frac{1}{2d}\int_{\bbR^d}|\nabla \varphi|^2\,\De x$,  and $\varphi$ is harmonic on $B_r\setminus D$.
\end{lemma}
\begin{proof} It is a standard first variation argument that the minimizer $\varphi$ of~\eqref{eq:variational_problem_r} is harmonic on $B_r\setminus D$, that it can be chosen to be non-negative, and that it saturates the constraint $\fint_{D} \vartheta(\varphi^2)\,\De x = \nu$, in view of the Lipschitz continuity of $\vartheta$.  Furthermore, it is immediate from the setting of the problem that $\cI_{D,r}(\nu)$ is decreasing in $r>2 r_D$ and that $\cI_{D,r}(\nu) \geq \cI_D(\nu)$ for all $r>2 r_D$. Note that if $\varphi_{\min}$ is the minimizer for $\cI_D(\nu)$, we can consider $\psi \varphi_{\min} $, where $\psi$ is a smooth cutoff with values in $[0,1]$ that is supported on $B_r$, equal to one on $B_{r/2}$, and such that $\|\nabla \psi\|_\infty \stackrel{\mathrm{def}}{=} \sup\{ |\nabla\psi(x)| \, : \, x \in \bbR^d \} \leq c/r$. Note that $\psi\varphi_{\min}\in H_0^1(B_r)$, $\fint_D \vartheta((\psi\varphi_{\min})^2)\,\De x = \nu$ if $r > 2r_D$, and (using that for $\varepsilon > 0$, and real numbers $a,b$, one has $(a+b)^2 \leq (1+\varepsilon)a^2 + (1+\frac{1}{\varepsilon})b^2$)
    \begin{equation}\label{eq:cutoff}
    \begin{split}
        \frac{1}{2d}\int_{\bbR^d}|\nabla (\psi\varphi_{\min})|^2\,\De x & \leq \frac{1+\varepsilon}{2d}\int_{\bbR^d}|\nabla \varphi_{\min}|^2\,\De x + \frac{c\left(1 + \frac{1}{\varepsilon} \right)}{r^2} \int_{B_r\setminus B_{r/2}} \varphi_{\min}^2\,\De x, \\
        \end{split}
    \end{equation} 
    for any $\varepsilon > 0$. By~\eqref{eq:decay}, we find that
\begin{equation}
    \label{eq:L2norm}
    \frac{1}{r^2} \int_{B_r\setminus B_{r/2}} \varphi_{\min}^2\,\De x \leq  \left(\max_{|x| = 2r_D} \varphi^2_{\min}(x)\right) \frac{C}{r^2} \int_{r/2}^r s^{3-d} \,\De s \to 0,\qquad \text{as $r\to\infty$}.
\end{equation}
We deduce from~\eqref{eq:cutoff} and~\eqref{eq:L2norm} and using that $\varepsilon > 0$ is arbitrary that $\lim_{r\to\infty}\cI_{D,r}(\nu) \leq \cI_D(\nu) $ and the proof is complete.
\end{proof}

\begin{prop}\label{prop:quasi-minimizer} Fix $\nu\in (0,\vartheta_\infty)$, recall $r_D = \sup \{|x|:\, x\in D\}$,~and consider a fixed $0 < \delta < 1 \wedge (r_D/2)$. Then, for all large $r > 4 r_D$ there exists $\varphi:\bbR^d\to \bbR$ such that:
    \begin{itemize}
        \item[i)] $\varphi\in C^\infty(\overline{B}_{r})$, $\supp\,\varphi = \overline{B}_{r}$, and $\varphi > 0$ on ${B}_{r}$;
        \item[ii)] $\varphi$ is harmonic on $B_{r}\setminus D^\delta $. In particular for all $s\in [r_D + \delta,r)$ and $x\in B_r\setminus \overline{B}_s$
        \begin{equation}\label{eq:harmonic_estimates}
             \frac{r^{2-d} - |x|^{2-d}}{r^{2-d} - s^{2-d}} \Big(\min_{|z| = s} \varphi(z) \Big) \leq \varphi(x) \leq  \frac{r^{2-d} - |x|^{2-d}}{r^{2-d} - s^{2-d}} \Big(\max_{|z| = s} \varphi(z)  \Big);
        \end{equation}
        \item[iii)] It holds that
        \begin{equation}
            \label{eq:last}
            \nu(1 + \delta)\leq \fint_D \vartheta(\varphi^2)\,\De x \leq \nu(1+2\delta),\quad \cI_D(\nu(1+\delta))\leq  \frac{1}{2d} \int_{\bbR^d} |\nabla    \varphi|^2\,\De x \leq \cI_D(\nu(1+2\delta)).
        \end{equation}
    \end{itemize}   
    
\end{prop}
\begin{proof} In view of Lemma~\ref{lem:min_r}, we can take $r>4 r_D$ large enough such that $\cI_{D,r}(\nu(1+3 \delta/2)) < \cI_{D}(\nu(1+2\delta))$ and we let $\psi$ be a minimizer for $\cI_{D,r}(\nu(1+3\delta/2))$ as in Lemma~\ref{lem:min_r}. Fix some $\epsilon = \epsilon(\delta) \in (0,\delta)$ to be chosen small enough and consider a smooth, radially symmetric probability density $\rho_\epsilon$ supported in $B_\epsilon$ (namely such that $\rho_\epsilon(x)$ depends on $x \in \bbR^d$ only through $|x|$). We define (with $\ast$ denoting the usual convolution operator):
    \begin{equation}
       \varphi_\epsilon =  \begin{cases}
            (\psi+\epsilon h_{D,r}) \ast \rho_\epsilon & \text{on } B_{r_D+\delta},\\
            (\psi+\epsilon h_{D,r}) & \text{otherwise},
        \end{cases}
    \end{equation} 
    where $h_{D,r}(x) = W_x[H_D < T_{B_r}]$, with $W_x$ denoting the standard Wiener measure starting from $x \in \bbR^d$, and $H_D$ and $T_{B_r}$ standing for the entrance time into $D$ and the exit time from $T_{B_r}$ for Brownian motion, respectively. We need to show that $\varphi_\epsilon$ satisfies i), ii) and iii). 
    
    We start with i). The fact that $\supp\,\varphi_\epsilon = \overline{B}_r$ and  $\varphi_\epsilon >0$ on $B_r$ is clear from the definition. Moreover, $\varphi_\epsilon$ is infinitely often differentiable in $B_{r_D + \delta} $ by definition, and, since $\psi + \epsilon h_{D,r}$ is harmonic in $B_r\setminus D$, by the mean value property of harmonic functions, we also have that $\varphi_\epsilon \equiv \psi + \epsilon h_{D,r}$ on $B_r \setminus D^\delta$ and thus $\varphi_\epsilon\in C^\infty(\overline{B}_r)$. 
    
    The harmonicity of $\varphi_\epsilon$ in $B_r\setminus D^\delta$ and the maximum principle readily imply ii).

    We are now left showing iii). Recall that the map $\phi\in L^2(D) \mapsto \fint_D \vartheta(\phi^2)\,\De x \in [0,\infty)$ is continuous. Therefore, using that $\varphi_\epsilon \to \psi$ in $L^2(D)$ as $\epsilon\to0$, there exists $\epsilon$ small enough such that 
    \begin{equation}
        \nu(1+\delta)\leq \fint_D \vartheta(\varphi_\epsilon^2)\,\De x \leq \nu(1+2\delta).
    \end{equation}

    In particular, by construction, $\frac{1}{2d} \int_{\bbR^d} |\nabla \varphi_\epsilon|^2\,\De x \geq \cI_D(\nu(1+\delta))$. We are left noticing that that
    \begin{equation}
        \limsup_{\epsilon\to 0}\frac{1}{2d}\int_{\bbR^d} |\nabla \varphi_\epsilon|^2\,\De x \leq \cI_{D,r}(\nu(1+3\delta/2)) < \cI_D(\nu(1+2\delta)).
    \end{equation}
    Therefore, choosing $\epsilon$ small enough we can guarantee that $\frac{1}{2d} \int_{\bbR^d} |\nabla \varphi_\epsilon|^2\,\De x \leq \cI_D(\nu(1+2\delta))$.
\end{proof}

\section{The tilted walk and estimates}
\label{sec:estimates}

In this section we introduce a near-optimal strategy to realize the event $\cA^\nu_N(F)$ introduced in~\eqref{eq:ExcessTypeEvent}. To that end, we first introduce a ``potential'' $\varphi_N = \varphi(\cdot/N)$ that will govern the behavior of the tilted random walk up to a time $S_N \approx \|\varphi_N\|_{\ell^2(\bbZ^d)}^2$ where we let $\varphi$ be a near-minimizer of the variational problem appearing on the right-hand side of~\eqref{eq:IntroMainResultLower}. More precisely, we fix a choice
\begin{equation}
\label{eq:FixedChoiceR}
\delta \in (0,1 \wedge (r_D/2)) \ \text{and} \ \cR \in (4r_D,\infty) \cap \mathbb{Z}, \text{ fulfilling the conditions of Proposition~\ref{prop:quasi-minimizer}}
\end{equation}
(i.e.~$\cR$ is an integer choice of the radius $r > 4r_D$ appearing in Proposition~\ref{prop:quasi-minimizer}). With this, we make the standing assumption throughout Sections~\ref{sec:estimates},~\ref{sec:coupling},  and~\ref{sec:lowerbound}:
\begin{equation}
\label{eq:FixedPhi}
\text{$\varphi$ is a near-minimizer constructed in Proposition~\ref{prop:quasi-minimizer} with $\delta,\cR$ as in~\eqref{eq:FixedChoiceR}.}
\end{equation} 
All constants appearing throughout this section may implicitly depend on $\delta,\cR$, and $\varphi$. Then we define for integer $N \geq 1$
\begin{equation}
\label{eq:varphi_N_Def}
    \varphi_N(x) \stackrel{\text{def}}{=} \varphi\Big(\frac{x}{N}\Big),\quad \text{and} \quad f(x) \stackrel{\text{def}}{=} \frac{\varphi_N(x)}{\|\varphi_N\|_{\ell^2(\bbZ^d)}},\quad x\in \bbZ^d,
\end{equation}
with $U^N = (NB_{\cR }) \cap \bbZ^d$ as above~\eqref{eq:Conditions_f} (with the choice $R = \cR$). Our goal is now to use the tilted walk in~\eqref{eq:TiltedWalk} with the choice $f$ as in~\eqref{eq:varphi_N_Def}. For later use, we will also consider for a fixed $\eta \in (0,\cR/100)$ the set 
\begin{equation}
(U_\eta)^N = (N B_{\cR - \eta}) \cap \bbZ^d.
\end{equation}
This will be convenient, since 
\begin{equation}
\label{eq:min_phi_strictlypositive}
c_3(\eta) = \min_{z \in  \overline{B}_{\cR - \eta}} \varphi(z) > 0,
\end{equation}
by Proposition~\ref{prop:quasi-minimizer}, i). We also set for a given $\varepsilon \in (0,1)$
\begin{equation}\label{eq:time_horizon}
    S_N = (1 + \varepsilon)\|\varphi_N\|_{\ell^2(\bbZ^d)}^2.
\end{equation}
It is immediate to see that as a consequence of Proposition \ref{prop:quasi-minimizer}, for large enough $N$, the function $f$ in~\eqref{eq:varphi_N_Def} fulfills the conditions~\eqref{eq:Conditions_f}. Indeed,~\eqref{eq:Conditions_f}, i) follows from the fact that $\varphi > 0$ on $B_{\cR}$ and $\varphi\vert_{B_{\cR}^c} = 0$, so that $\varphi_N(x) > 0$ if and only if $|x| < N\cR$, and~\eqref{eq:Conditions_f}, ii) follows immediately from the definition~\eqref{eq:varphi_N_Def}. We also recall the definition of the reversible measure $\pi$ in~\eqref{eq:piDef}, in which we fix $f$ as in~\eqref{eq:varphi_N_Def}. \smallskip

In the remainder of this section, we derive pivotal controls on $\varphi_N$, $f$, $\pi$, as well as on properties associated with the tilted walk under this choice of $f$ (and $S_N$). Roughly speaking, the near-minimizer $\varphi$ plays the role of $\widetilde{h}$ constructed in~\cite[Lemma 2.1]{li2017lower} (with $\varphi_N$, $f$, and $\pi$ replacing $h_N = \widetilde{h}(\cdot/N)$, $f = h_N / \|h_N\|_{\ell^2(\bbZ^d)}$, in (2.15) and (2.38) of the same reference, respectively). In the sequel, we will adapt the arguments in~\cite[Sections 2--4]{li2017lower} to our context (in particular, Propositions~\ref{prop:ClosenessToQSD} and~\ref{prop:AnalogProp4.7} below). Some of these modifications are straightforward, and we explain only significant changes in the proof in order to keep the exposition concise. However, a major technical obstruction comes from the fact that we need to derive controls on the tilted walk in essentially \textit{all} mesoscopic sub-boxes of $U^N$ intersecting $D_N$, whereas in~\cite{li2017lower}, precise controls on hitting distributions were only needed for certain subsets of $U^N$ where the corresponding near-minimizer $\widetilde{h}$ of the variational problem attached to the capacity is constant (and equal to one). In order to deal with this, we also develop some comparisons between the equilibrium measure for a simple random walk and that of (an appropriately defined version of) the ``tilted walk in a mesoscopic box'' in Proposition~\ref{prop:ComparisonEqmeasure} below. Roughly speaking, we will heavily rely on the smoothness and boundedness of the near-minimizer $\varphi$, as well as its harmonicity on $B_\cR \setminus D^\delta$ for $\delta$ as in~\eqref{eq:FixedChoiceR} (which implies good quantitative controls on its decay, see~\eqref{eq:harmonic_estimates}). These allow us on one hand to conclude that the conductances associated with the discrete skeleton of the tilted (or confined) walk can be treated as ``essentially constant'' in mesoscopic boxes of size $M \ll N$ (with a quantitative decay rate on their oscillation, see~\eqref{eq:WeightsDontMatterMuch}). On the other hand, this also allows us to leverage standard heat-kernel bounds to obtain good controls on exit times, see Lemma~\ref{lem:HKLemma}. \smallskip

We start with this program by subsuming in the next lemma several technical properties of $\varphi_N$, $f$, and $S_N$ from Lemmas 2.3, 2.10, 2.11 of~\cite{li2017lower}. 
\begin{lemma}
\label{lem:TechnicalProperties}
Let~$\mathcal{R}$ be as in~\eqref{eq:FixedChoiceR} and consider the fixed near-minimizer $\varphi$ as in~\eqref{eq:FixedPhi}. For $\varphi_N$, $f$, $\pi$ and $S_N$ as in~\eqref{eq:varphi_N_Def},~\eqref{eq:piDef},
and~\eqref{eq:time_horizon}, respectively, we have the following properties for $N$ large enough: 
\begin{itemize}
\item[i)] The function $\varphi_N$ fulfills
\begin{equation}
\begin{split}
\label{eq:Phi_N_TechnicalProp}
& cN^{-2} \leq \varphi_N(x) \leq C  \text{ for all }x \in U^N,  \\
& \varphi_N(x) \leq cN^{-1}\text{ for all }x \in \partial_{\mathrm{int}} U^N,  \\
& \varphi_N(x) \geq cN^{-1}\text{ for all }x \in O^N,\\
& cN^d \leq \| \varphi_N \|_{\ell^2(\mathbb{Z}^d)}^2 \leq C N^d,
\end{split}
\end{equation}
where we define $O^N$ as 
\begin{equation}
\begin{split}
O^N & = \Big( U^N \setminus ( \partial_{\mathrm{int}} U^N \cup B_{N\cR/2} ) \Big) \\
& \qquad \cup \{x \in \partial_{\mathrm{int}} U^N \, : \, |y| = N\cR \text{ for all }y \sim x, y \notin U^N \}
\end{split}
\end{equation}
(see Figure~\ref{fig:O_N} for a graphical representation of the set $O^N$).
\begin{figure}[htbp]
\begin{center}
\includegraphics[scale=.45]{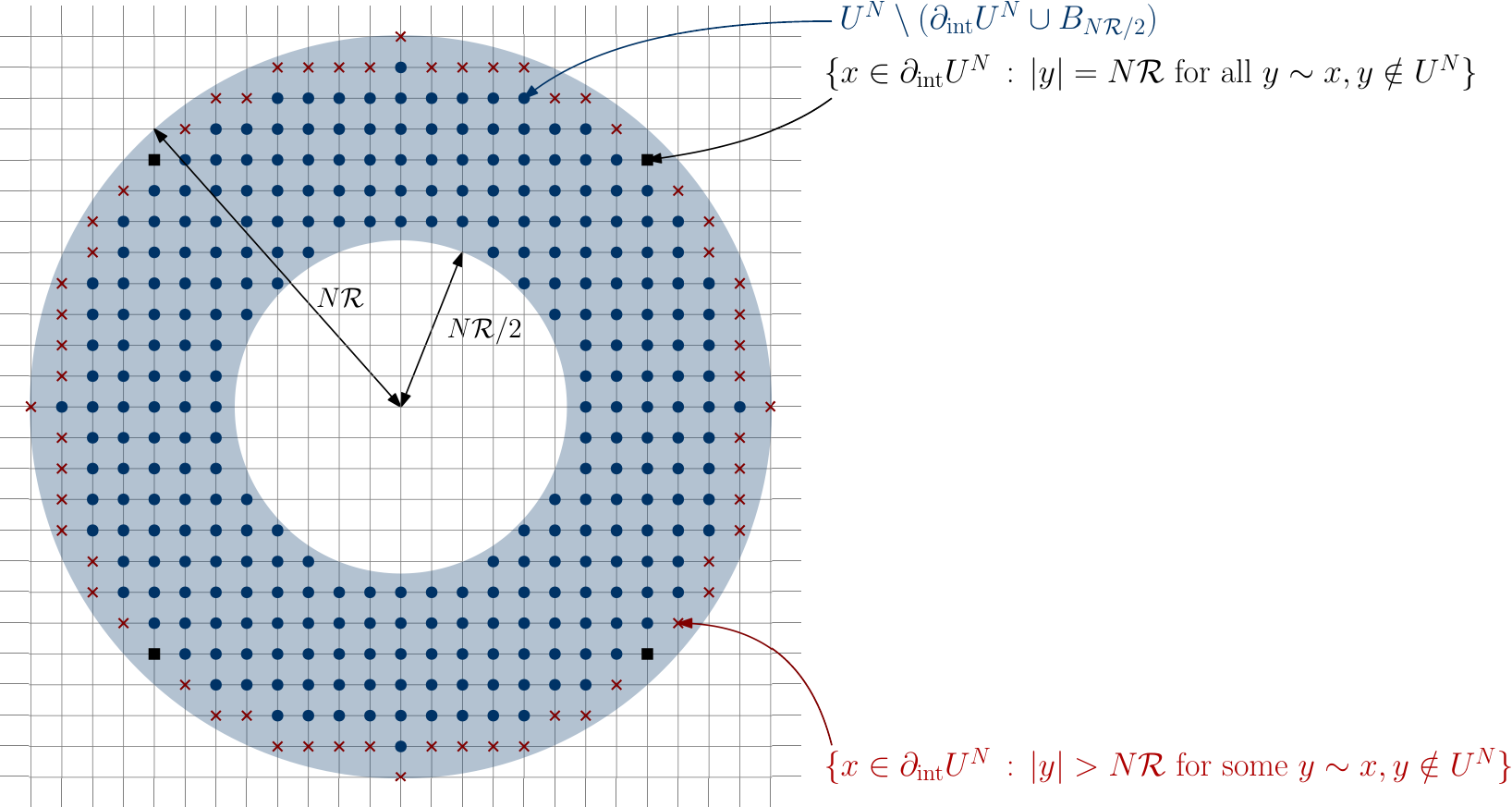}
\end{center}
\caption{An illustration of the sets in the definition of $O^N$: The interior boundary $\partial_{\mathrm{int}}U$ consists of points such that for any nearest neighbor $y$ outside of $U^N$, $|y| = N \mathcal{R}$ holds (in black) and points for which this condition is not valid (in red). In the definition of $O^N$, the former are retained.}
\label{fig:O_N}
\end{figure}
%\begin{figure}[htbp]
%\begin{center}
%\includegraphics[scale=.85]{O_N_with_boundary}
%\end{center}
%\caption{An illustration of the sets in the definition of $O^N$: The interior boundary $\partial_{\mathrm{int}}U$ consists of points such that for any nearest neighbor $y$ outside of $U^N$, $|y| = N \mathcal{R}$ holds (in black) and points for which this condition is not valid (in red). In the definition of $O^N$, the former are retained.}
%\label{fig:O_N}
%\end{figure}
%\begin{figure}[htbp]
%\begin{center}
%\includegraphics[scale=.85]{O_N}
%\end{center}
%\caption{An illustration of the sets in the definition of $O^N$: The interior boundary $\partial_{\mathrm{int}}U$ contains points such that for any nearest neighbor $y$ outside of $U^N$, $|y| = N \mathcal{R}$ holds (in black). In the definition of $O^N$, these points are retained.}
%\label{fig:O_N}
%\end{figure}
\item[ii)] The measure $\pi$ ( = $f^2$) fulfills
\begin{equation}
cN^{-d-4} \leq \pi(x) \leq CN^{-d} \text{ for all }x \in U^N.
\end{equation}
\item[iii)] Define $v_\varphi$ as the function $v$ in~\eqref{eq:vDef} with $f$ chosen as in~\eqref{eq:varphi_N_Def}, namely
\begin{equation}
\label{eq:v_phiDef}
v_\varphi(x) = - \frac{\Delta_{\mathbb{Z}^d} f(x)}{f(x)} \left( =  - \frac{\Delta_{\mathbb{Z}^d} \varphi_N(x)}{\varphi_N(x)}\right), \qquad x \in U^N.
\end{equation}
Then the function $v_\varphi$ fulfills
\begin{equation}
\begin{split}
\label{eq:vUpperBound}
-cN^2 \leq \ & v_\varphi(x) \leq \frac{c'}{N^2},  \qquad x \in U^N, \text{ and}\\
 & v_\varphi(x) \geq - \frac{c(\eta)}{N^2}, \qquad x \in (U_\eta)^N.
 \end{split}
\end{equation}
\item[iv)] The time $S_N$ fulfills
\begin{equation}
\label{eq:S_N_volume_order}
cN^d \leq S_N \leq CN^d.
\end{equation}
\end{itemize}
\end{lemma}
\begin{proof}
We start with i). The upper bound on $\varphi_N$ in the first line of~\eqref{eq:Phi_N_TechnicalProp} is immediate by choosing $C = \max_{z\in\overline{B}_\cR}\varphi(z)$, which is finite by Proposition~\ref{prop:quasi-minimizer}, i). The lower bound in the first line follows as in the proof of claim 1 of~\cite[Lemma 2.3]{li2017lower}, using the fact that $\cR$ is integer, and employing the lower bound in~\eqref{eq:harmonic_estimates} of Proposition~\ref{prop:quasi-minimizer}, ii). Analogously, the proof of the second and third lines of~\eqref{eq:Phi_N_TechnicalProp} follow similarly as in the proof of claims 2 and 3 of~\cite[Lemma 2.3]{li2017lower}. The last line in~\eqref{eq:Phi_N_TechnicalProp} is again immediate from the definition of $\varphi_N$ and the fact that $\varphi \in C^\infty(\overline{B}_{\cR})$ and  $\varphi > 0$ on $B_\cR$; see Proposition~\ref{prop:quasi-minimizer}, i). \smallskip

Item ii) follows upon using the definition of $f$ in~\eqref{eq:varphi_N_Def} and of $\pi$ in~\eqref{eq:piDef} and combining the first and last line of~\eqref{eq:Phi_N_TechnicalProp}. \smallskip

We now turn to iii). The proof of the first line in~\eqref{eq:vUpperBound} corresponds to that of~\cite[Lemma 2.11]{li2017lower}, and we briefly explain the changes in our set-up. The lower bound in the first line follows directly upon combining the lower and upper bounds of~\eqref{eq:Phi_N_TechnicalProp} and using the fact that $\varphi_N \geq 0$, and we now turn to the upper bound. We will now distinguish the analysis into three regions, namely $I^N = B_{N\cR/2}\cap \bbZ^d$, $O^N$, and $S^N = \partial_{\mathrm{int}}U^N\setminus O^N$. 
\smallskip

We start by estimating $v_\varphi$ for $x\in I^N\cup O^N$ and notice that the function $\varphi$ (which replaces $\widetilde{h}$ in the proof of~\cite[Lemma 2.11]{li2017lower}) is still $C^\infty(\overline{B}_{\cR})$. Thus, one can first obtain (using a standard second order Taylor formula) that
\begin{equation}
    \label{eq:DiscreteLaplacianOnPhi_N}
    \Delta_{\bbZ^d} \varphi_N(x) \geq \frac{1}{N^2}\left(\frac{1}{2d} \Delta \varphi\left(\frac{x}{N}\right) - \frac{c}{N}  \right), \qquad x\in I^N \cup O^N.
\end{equation}
On the one hand if $x\in I^N$ we can use that 
\begin{equation}
\label{eq:varphibiggerthanconst}
\varphi(x) \geq \min_{z \in \overline{B}_{\cR/2}} \varphi(z) = c \, (> 0), \qquad x \in B_{\cR/2},
\end{equation}
since $\varphi$ is smooth and strictly positive on $B_{\cR}$ by Proposition~\ref{prop:quasi-minimizer}, i), to find that
\begin{equation}
 - \frac{\Delta_{\mathbb{Z}^d} \varphi_N(x)}{\varphi_N(x)} \leq \frac{|\tfrac{1}{2d}\Delta \varphi(\tfrac{x}{N})-\tfrac{c}{N}|}{c N^2} \leq \frac{c'}{N^2}.
\end{equation}
On the other hand if $x\in O^N$, we can use that $\varphi_N(x) \geq cN^{-1}$ for all $x\in O^N$ and that $\varphi$ is harmonic in $B_\cR\setminus B_{\cR/2}$, so that $\Delta \varphi(x/N) = 0$, to get
\begin{equation}
    - \frac{\Delta_{\mathbb{Z}^d} \varphi_N(x)}{\varphi_N(x)} \leq \frac{-\tfrac{1}{2d}\Delta \varphi(\tfrac{x}{N})+\tfrac{c}{N}}{\varphi_N(x) N^2} \leq \frac{c'}{N^2}.
\end{equation}
Finally, we briefly discuss the case $x\in S^N$. Also in this case one can follow almost verbatim the proof of~\cite[Lemma 2.11]{li2017lower} with the key observation that also in our situation~\cite[Lemma 6.37, p.~136]{gilbarg2001elliptic} applies to $\varphi$, and  
the outer normal derivative of $\varphi$, fulfills $\frac{\partial \varphi}{\partial n}(z)  < -c$ for $z \in \partial B_\cR$. The existence of this derivative follows from Proposition~\ref{prop:quasi-minimizer}, i), and the claimed upper bound follows from~\eqref{eq:harmonic_estimates} since $\varphi$ is bounded from above and below by two functions with constant negative outer normal derivatives on $\partial B_{\cR}$ (see~\cite[(2.49)]{li2017lower} and the discussion above it for a similar argument).
 
 For the second line of~\eqref{eq:vUpperBound}, it suffices to bound (in the same way as~\eqref{eq:DiscreteLaplacianOnPhi_N})
\begin{equation}
\Delta_{\bbZ^d} \varphi_N(x) \leq \frac{1}{N^2}\left(\frac{1}{2d} \Delta \varphi\left(\frac{x}{N}\right) + \frac{c}{N}  \right), \qquad x\in (U_\eta)^N,
\end{equation}
and by~\eqref{eq:min_phi_strictlypositive}, we see that
\begin{equation}
-\frac{\Delta_{\bbZ^d}\varphi_N(x)}{\varphi_N(x)} \geq -\frac{\frac{1}{2d} \Delta \varphi\left(\frac{x}{N}\right) + \frac{c}{N} }{\varphi\left(\frac{x}{N} \right)N^2} \geq -\frac{c(\eta)}{N^2}.
\end{equation}
  \smallskip

Finally, we see that iv) follows immediately from the definition~\eqref{eq:time_horizon} of $S_N$ and the fourth line of~\eqref{eq:Phi_N_TechnicalProp}.
\end{proof}

We will frequently make use of the regeneration time of the confined walk defined by
\begin{equation}
\label{eq:RegenerationTimeDef}
t_\star = N^2 \log^2 N.
\end{equation}
 In essence, the confined walk on $U^N$ at time $t_\star$ will be distributed according to the reversible distribution $\pi$ up to a super-polynomially decaying error term. More precisely, we have the following.
 \begin{lemma}
 \label{lem:TVtoStationary}
 For $t \geq t_\star$, we have
 \begin{equation}
 \sup_{x,y  \in U^N} |\overline{P}_x[X_t = y] - \pi(y)| \leq e^{-c \log^2 N}.
 \end{equation}
 \end{lemma}
\begin{proof}
This follows from Proposition~\ref{lem:TechnicalProperties}, ii), upon adapting~\cite[Proposition 2.12]{li2017lower} to our context. Precisely, we need to show that the spectral gap $\overline{\lambda}$ of the confined walk (by which we mean the smallest strictly positive eigenvalue of $-\mathscr{L}_f$ defined in~\eqref{eq:GeneratorVSRW} with the choice of $f$ as in~\eqref{eq:varphi_N_Def}, viewed as a self-adjoint operator on $\ell^2(U^N,\pi)$) can be bounded below by $cN^{-2}$. We can follow the method of canonical paths exactly as in the proof of~\cite[Proposition 2.12]{li2017lower}, where we only need to show that 
\begin{equation}
\label{eq:AlmostMonotonicity}
\text{For }x,y \in U^N, |x| \leq |y| \text{ implies }\pi(x) \geq c\pi(y).
\end{equation}
To establish~\eqref{eq:AlmostMonotonicity}, note that by~\eqref{eq:harmonic_estimates}, one has
\begin{equation}
\varphi(z) \geq c' \varphi(w) \qquad \text{ for } z,w \in \bbR^d \text{ with } |z| \leq |w| \leq \cR. 
\end{equation}
Indeed, we can choose the constant $c'$ as $\frac{\min_{|z| = s} \varphi(z) }{\max_{|z| = s} \varphi(z)}$ for a fixed $s \in [r_D + \delta,\cR)$.
\end{proof}

\subsection{Heat kernel bounds and capacity estimates for a random walk with slowly varying conductances}

We will now derive important comparisons between the simple random walk and a random walk among certain non-constant conductances, which may be thought of as a version of the conductances a confined walk sees within a mesoscopic box, extended to the entire lattice by periodicity. This will allow us to effectively compare the local behavior of the confined walk (on a mesoscopic scale $N^{r_5} \ll N$, with some $r_5 \in (0,1)$ defined in~\eqref{eq:r_jDef} below) with that of a simple random walk. To this end, we introduce some further notation.
We consider
\begin{equation}
\label{eq:zetaAndx_0}
\qquad x_0 \in D^\delta_N \stackrel{\mathrm{def}}{=} (ND^\delta) \cap \bbZ^d,\qquad \text{with $\delta$ as in~\eqref{eq:FixedChoiceR}}.
\end{equation}
We then introduce real numbers $r_j$, $1 \leq j \leq 5$, with
\begin{equation}
\label{eq:r_jDef}
\begin{split}
& r_j \in (0,\tfrac{1}{4}), \text{ for }1 \leq j \leq 5, \qquad r_{j-1} < r_{j}, \text{ for }2
\leq j \leq 5, \\
& r_1 \cdot (d-2) + r_5 < 1, \text{ and }\\ & r_1 < \frac{d-2}{d-1} r_2,
\end{split}
\end{equation}
as well as the concentric boxes given by
\begin{equation}
\label{eq:AjBoxesDef}
A_j = B(x_0, \lfloor N^{r_j} \rfloor), \text{ for }1 \leq j \leq 5, \qquad A_6 = B\Big(x_0, \Big\lfloor \tfrac{\widetilde{\delta}}{100}N \Big\rfloor \Big),
\end{equation}
for some sufficiently small $\widetilde{\delta} > 0$ guaranteeing that $A_6 \subseteq (U_\eta)^N$ for all $N \geq 1$.  The exponents $(r_j)_{1 \leq j \leq 5}$ and the parameter $\widetilde{\delta}$ will be fixed throughout this article, and we suppress dependence of constants on these quantities in the notation. \smallskip

Let $(\mu_{x,y})_{x,y \in \bbZ^d, x \sim y}$ be symmetric weights on the edges on $\mathbb{Z}^d$, $d \geq 3$ (denoted by $\bbE^d$), with 
\begin{equation}
\label{eq:Ellipticity}
\lambda \leq \mu_{x,y} = \mu_{y,x} \leq \Lambda, \qquad \text{for some } 0 < \lambda < \Lambda < \infty.
\end{equation}
For a given measure $\nu : \bbZ^d \rightarrow (\underline{c},\overline{C})$, we let $P^{\mu;\nu}_x$ denote the law of the (variable speed) continuous-time random walk on $(\mathbb{Z}^d, \mathbb{E}^d, \mu)$ with speed measure $\nu$, i.e.~the continuous-time Markov process starting at $x$ with generator 
\begin{equation}
\label{eq:VeryGeneralGenerator}
\mathscr{L}^{\mu;\nu} g(x) = \sum_{z \sim x} \frac{\mu_{x,y}}{\nu_x}  \big(g(z) - g(x)),
\end{equation} 
for functions $g : \mathbb{Z}^d \rightarrow \mathbb{R}$. Under $P_x^{\mu;\nu}$, the Markov chain $(X_t)_{t \geq 0}$ corresponds to a random walk starting in $x$ with jump probabilities given by $\frac{\mu_{y,z}}{M_y}$, where
\begin{equation}
\label{eq:StandardTotalMass}
M_y = \sum_{z : z \sim y}\mu_{y,z},\qquad y \in \bbZ^d,
\end{equation}
which waits at a point $z \in \bbZ^d$ for  an exponential time with mean $\nu_z/M_z$ (see~\cite[Remark 5.7, p.~137--138]{barlow2017random} for more details on this). We also define the corresponding heat kernel by
\begin{equation}
q_t^{\mu;\nu}(x,y) = \frac{P^{\mu;\nu}_x[X_t = y]}{ \nu_y}, \qquad t \geq 0, x,y \in \bbZ^d,
\end{equation}
as well as the Green function by
\begin{equation}
\label{eq:GeneralGreenFunctionDef}
g^{\mu;\nu}(x,y) = E^{\mu;\nu}\left[\int_0^\infty \mathbbm{1}_{\{X_t = y \}} \mathrm{d}t \right] \cdot \frac{1}{\nu_y} = \int_0^\infty q_t^{\mu;\nu}(x,y)  \mathrm{d}t.
\end{equation}
 For $x_0 \in D^\delta_N$, and $\varphi_N$ as in~\eqref{eq:varphi_N_Def}, we consider the $(2\lfloor N^{r_5}\rfloor + 3)$-periodic extension
\begin{equation}
\label{eq:PeriodicExtensionphi}
\varphi_N^{\mathrm{per}}(x) = \varphi_N(x), \qquad x \in B(x_0,\lfloor N^{r_5}\rfloor+1), \text{ extended periodically to }\bbZ^d
\end{equation}
(i.e.~$\varphi_N$ is specified on all points $y \in \bbZ^d$ with $d_\infty(y,A_5) \leq 1$). With this, we can consider the specific choice of weights
\begin{equation}
\label{eq:NonConstantConductances}
\overline{\mu}_{x,y} = \frac{1}{2d} \varphi_N^{\mathrm{per}}(x)\varphi_N^{\mathrm{per}}(y), \text{ for } x\sim y, x,y \in \bbZ^d,
\end{equation}
and the speed measure
\begin{equation}
\label{eq:ChoiceOfSpeedMeasurePi}
\overline{\nu}_x = (\varphi_N^{\mathrm{per}})^2(x), \qquad x \in \mathbb{Z}^d
\end{equation}
(we suppress the additional dependence of $\overline{\mu}_{x,y}$ on $\varphi$, $N$, and $x_0$). Note that for any $x_0$ as in~\eqref{eq:zetaAndx_0}, the conductances~\eqref{eq:NonConstantConductances} fulfill~\eqref{eq:Ellipticity} with constants $\lambda,\Lambda$ not depending on $N$, upon using ~\eqref{eq:min_phi_strictlypositive} and the first line of~\eqref{eq:Phi_N_TechnicalProp}. Moreover, the speed measure $\overline{\nu}$ in~\eqref{eq:ChoiceOfSpeedMeasurePi} is bounded uniformly from above and below by constants that do not depend on $N$.  Importantly, we have that
\begin{equation}
\label{eq:ConfinedWalk_coincides}
\textit{the laws of $(X_{t \wedge T_{A_j}})_{t \geq 0}$, $1\leq j\leq 5$, under $\overline{P}_x$ and under $P^{\overline{\mu};\overline{\nu}}_x$ with $x \in A_5$ coincide,}
\end{equation}
which follows immediately by verifying that the value of the generator $\mathscr{L}^{\overline{\mu};\overline{\nu}}g$ coincides with $\mathscr{L}_fg$ for functions $g : U^N \rightarrow \mathbb{R}$ that vanish outside of $A_5$. 
Since we primarily work with the choice~\eqref{eq:ChoiceOfSpeedMeasurePi}, we use $P^{\overline{\mu}}_x$ (resp.~$q_t^{\overline{\mu}}(x,y)$,~$g^{\overline{\mu}}(x,y)$) as a shorthand notation for $P^{\overline{\mu};\overline{\nu}}_x$ (resp.~$q_t^{\overline{\mu};\overline{\nu}}(x,y)$,~$g^{\overline{\mu};\overline{\nu}}(x,y)$) for any $x,y \in \bbZ^d$, $t \geq 0$.
\begin{remark}
\label{rem:CommentOnConductances}
Let us briefly comment on our choices of the conductances~\eqref{eq:NonConstantConductances}
and the speed measure~\eqref{eq:ChoiceOfSpeedMeasurePi}.
An alternative choice would be to set
\begin{equation}
\begin{split}
\overline{\mu}^{(f)}_{x,y} & = \frac{1}{2d}f^{\mathrm{per}}(x)f^{\mathrm{per}}(y) \text{ for } x\sim y, x,y \in \mathbb{Z}^d, \\
\overline{\nu}_x^{(f)} & =(f^{\mathrm{per}})^2(x), \text{ for } x \in \mathbb{Z}^d,
\end{split}
\end{equation}
with $f(x) = \varphi_N(x) / \| \varphi_N\|_{\ell^2(\bbZ^d)}$ in~\eqref{eq:varphi_N_Def} specified for $x \in B(x_0,2\lfloor N^{r_5} \rfloor +1)$ and $f^{\mathrm{per}}$ its periodic extension to $\bbZ^d$ (similarly as in~\eqref{eq:PeriodicExtensionphi}). Inserting these choices into the definition~\eqref{eq:VeryGeneralGenerator} yields the same generator as we obtain by using~\eqref{eq:NonConstantConductances} and~\eqref{eq:ChoiceOfSpeedMeasurePi}. However, this choice is slightly less convenient, since both $(\overline{\mu}^{(f)}_{x,y})_{x,y}$ and $\overline{\nu}^{(f)}$ have non-trivial scaling as $N$ increases (see Proposition~\ref{lem:TechnicalProperties}).
\end{remark} 
With the choices as in~\eqref{eq:NonConstantConductances} and~\eqref{eq:ChoiceOfSpeedMeasurePi} above, we can introduce several potential-theoretic quantities attached to the laws $(P^{\overline{\mu}}_x)_{x \in \bbZ^d}$, which can be seen as analogues to the quantities introduced for the simple random walk in Section~\ref{sec:preliminaries}, and we refer to~\cite[Chapter 7]{barlow2017random} for more details on this. We introduce for $\varnothing \neq K \subset\subset\bbZ^d$ the tilted equilibrium measure of $K$,
\begin{equation}
\label{eq:TiltedEqMeasureDefinition}
\overline{e}_K(x) = P^{\overline{\mu}}_x[\widetilde{H}_K = \infty]\varphi_N^{\mathrm{per}}(x) \left(\frac{1}{2d} \sum_{y : y \sim x} \varphi_N^{\mathrm{per}}(y) \right)\mathbbm{1}_K(x), \qquad x \in \bbZ^d.
\end{equation}
 Its total mass
\begin{equation}
\overline{\capa}(K) = \sum_{x \in K}\overline{e}_K(x),
\end{equation}
is called the tilted capacity of $K$. We also define the normalized tilted equilibrium measure of $K$ as
\begin{equation}
\label{eq:TiltedNormalizedEqmeasure}
\widetilde{\overline{e}}_K(x)  = \frac{\overline{e}_K(x)}{\overline{\capa}(K)}, \qquad x \in \bbZ^d.
\end{equation}
These definitions are in analogy to~\eqref{eq:eqmeasureDef},~\eqref{eq:capacityDef}, and~\eqref{eq:normalized_eqmeasureDef}, using the conductances~\eqref{eq:NonConstantConductances} and we will derive controls that allow us to compare these ``tilted'' quantities (up to polynomially decaying error terms) to the respective ones for the simple random walk. To facilitate this, we need a result concerning deviations of the time $T_{A_2}$ from times of order $N^{2r_2}$, which brings into play standard Gaussian heat kernel estimates. We use the abbreviation
\begin{equation}
\label{eq:NotTooLargeExit}
\mathcal{E}_\alpha = \left\{ T_{A_2} \leq N^{2\alpha} \right\},\qquad N \geq 1.
\end{equation}
We now establish that the event $\cE_\alpha$ is exponentially unlikely both for the simple random walk as well as for the walk with conductances $\overline{\mu}$ and speed measure $\overline{\nu}$, both started in $x \in A_1$.
\begin{lemma}
\label{lem:HKLemma}
Let $\alpha \in (r_2,1)$. Then, for every $N \geq 1$, 
\begin{equation}
\label{eq:ExitTimeDeviation}
\begin{split}
\sup_{x \in A_1} P^{\overline{\mu}}_{x}[\mathcal{E}_\alpha^c] & \leq C_4 \exp\Big(- C_5 N^{c_6} \Big), \\
\sup_{x \in A_1} P_{x}[\mathcal{E}_\alpha^c] & \leq C_4' \exp\Big(- C_5' N^{c_6'} \Big).
\end{split}
\end{equation}
(with constants depending on $\alpha$).
\end{lemma}
Before proving Lemma~\ref{lem:HKLemma}, we establish that the continuous-time random walk defined by $(P^{\overline{\mu}}_x)_{x \in \bbZ^d}$ fulfills standard Gaussian heat kernel estimates.
\begin{lemma}
\label{lem:NewHK}
For large enough $N$, one has
\begin{equation}
\label{eq:SpeedmeasureLowerBound}
 c \leq \overline{\nu}(x) \leq C, \qquad x \in \bbZ^d,
\end{equation}
as well as
\begin{equation}
\label{eq:HeatKernelForOurWalk}
\frac{c}{t^{d/2}} e^{-c' \frac{|x - y|^2}{t}} \leq q_t^{\overline{\mu}}(x,y) \leq \frac{C}{t^{d/2} } e^{-C' \frac{|x - y|^2}{t}}, \qquad t \geq 1 \vee \widetilde{c}|x-y|, \ x,y \in \mathbb{Z}^d,
\end{equation}
with constants that may depend on $\varphi$ but do not depend on $N$.
\end{lemma}
\begin{proof}
The first claim~\eqref{eq:SpeedmeasureLowerBound} is immediate by the definition~\eqref{eq:ChoiceOfSpeedMeasurePi},~\eqref{eq:min_phi_strictlypositive}, and~\eqref{eq:Phi_N_TechnicalProp}. We now turn to the proof of~\eqref{eq:HeatKernelForOurWalk}. By~\cite[Theorem 6.28]{barlow2017random}, the discrete heat kernel associated with conductances satisfying~\eqref{eq:Ellipticity} admits standard Gaussian heat-kernel bounds. Note that the conductances $\overline{\mu}_{x,y}$ are indeed uniformly elliptic in the sense of~\eqref{eq:Ellipticity} due to~\eqref{eq:min_phi_strictlypositive} and~\eqref{eq:Phi_N_TechnicalProp} with constants $\lambda,\Lambda$ not depending on $N$. With the choice~\eqref{eq:NonConstantConductances}, the corresponding measure $M$ coincides with
\begin{equation}
\label{eq:NaturalMMeasureDef}
\overline{M}_x \stackrel{\mathrm{def}}{=} \frac{1}{2d} \varphi_N^{\mathrm{per}}(x)\sum_{y: y \sim x} \varphi_N^{\mathrm{per}}(y), \qquad x \in \bbZ^d,
\end{equation}
where $ \varphi_N^{\mathrm{per}}$ is defined in~\eqref{eq:PeriodicExtensionphi}. Note that for all $x \in \bbZ^d$, using the smoothness of $\varphi$ (see Proposition~\ref{prop:quasi-minimizer}, i)), one has
\begin{equation}
\label{eq:M_pi_are_Close}
\left\vert\frac{\overline{M}_x}{(\varphi_N^{\mathrm{per}})^2(x)} - 1\right\vert = \left\vert \frac{\frac{1}{2d} \sum_{y: y \sim x} \varphi_N^{\mathrm{per}}(y) - \varphi_N^{\mathrm{per}}(x)}{\varphi_N^{\mathrm{per}}(x)} \right\vert \leq C N^{r_5 - 1}
\end{equation}
(the additional factor $N^{r_5}$ comes from the periodization). Now, $(X_t)_{t\geq 0}$ under $P^{\overline{\mu}}_x$  corresponds to a continuous-time random walk with variable jump rates given by $(\overline{\nu}_x/\overline{M}_x)_{x \in \bbZ^d}$ (see also the discussion below~\eqref{eq:VeryGeneralGenerator}), which are bounded above and below by~\eqref{eq:M_pi_are_Close}.
We can conclude that~\eqref{eq:HeatKernelForOurWalk} holds from the discrete heat kernel bounds, by using the fact that the jump rates of $(X_t)_{t \geq 0}$ under $P^{\overline{\mu}}_x$ are uniformly bounded above and below.
\end{proof}
We will now prove Lemma~\ref{lem:HKLemma}.
\begin{proof}[Proof of Lemma~\ref{lem:HKLemma}]
Set $\mathcal{A}(z,K) = B(z,2K) \setminus B(z,K)$ for $z\in \bbZ^d$ and $K \geq 1$. We consider a set-up of a general heat kernel $q^{\mu;\nu}_t(\cdot,\cdot)$ and a speed measure $\nu \geq c$ such that
\begin{equation}
\label{eq:LowerBoundHeatKernelAssumption}
q_t^{\mu;\nu}(x,y) \geq \frac{c}{t^{d/2}}e^{-c'\frac{|x-y|^2}{t}}, \qquad t \geq 1 \vee \widetilde{c}|x-y|, \ x,y \in \mathbb{Z}^d.
\end{equation}
Using the lower bound in~\eqref{eq:LowerBoundHeatKernelAssumption}, we see that for any $x \in A_2$, we have 
\begin{equation}
\label{eq:LowerBoundSummedHK}
\begin{split}
P^{\mu;\nu}_{x}[X_{ N^{2r_2}} \in \mathcal{A}(x_0,20N^{r_2}) ] & \geq |\mathcal{A}(x_0,20N^{r_2})| \cdot \inf_{ y \in \mathcal{A}(x_0,20N^{r_2})} q_{N^{2r_2}}^{\mu;\nu}(x,y) \nu(y) \\
& \geq \overline{c}  \ (> 0).
\end{split}
\end{equation}
Now note that 
\begin{equation}
\{T_{A_2} > N^{2\alpha} \} \subseteq \bigcap_{j = 1}^{ \lfloor N^{2(\alpha - r_2 )}\rfloor} \{X_{jN^{2r_2} } \notin \mathcal{A}(x_0,20N^{r_2}),X_{(j-1)N^{2r_2} } \in A_2 \}. 
\end{equation}
Using the simple Markov property repeatedly at times $jN^{2r_2}$, $j = 1,...,\lfloor N^{2(\alpha - r_2)} \rfloor-1$, we obtain that
\begin{equation}
P^{\mu;\nu}_x[T_{A_2} > N^{2\alpha}] \stackrel{\eqref{eq:LowerBoundSummedHK}}{\leq} (1 - \overline{c})^{\lfloor N^{2(\alpha - r_2)} \rfloor}, \qquad x \in A_2.
\end{equation}
The first line of the claim~\eqref{eq:ExitTimeDeviation} follows from the above by choosing the conductances $\mu$ as in~\eqref{eq:NonConstantConductances}, the speed measure $\nu$ as in~\eqref{eq:ChoiceOfSpeedMeasurePi}, and using Lemma~\ref{lem:NewHK}. To prove the claim in the second line of~\eqref{eq:ExitTimeDeviation}, note that $P_x$ with $x \in A_2$, corresponds to the simple random walk with conductances $\mu_{x,y} = \frac{1}{2d}$ for every $x\sim y$, $x,y \sim \mathbb{Z}^d$, and $\nu_x = M_x = 1$ for $x \in \bbZ^d$, which fulfills the lower bound~\eqref{eq:LowerBoundHeatKernelAssumption} again by~\cite[Theorem 6.28]{barlow2017random}.
\end{proof}
The following estimates will be pivotal in dealing with the aforementioned issue that we need to compare the tilted walk on mesoscopic boxes with the simple random walk.
\begin{prop}
\label{prop:ComparisonEqmeasure}
For large enough $N$, we have 
\begin{equation}
\label{eq:ComparisonEqmeasure}
\sup_{x \in \partial_{\mathrm{int}} A_1} \left\vert \frac{P^{\overline{\mu}}_x[\widetilde{H}_{A_1} = \infty]}{e_{A_1}(x)} -1 \right\vert \leq \frac{c}{N^{c'}},
\end{equation}
as well as 
\begin{equation}
\label{eq:CapaComparison}
| \varphi_N^2(x_0) \capa(A_1) - \overline{\capa}(A_1)| \leq \frac{c}{N^{c'}}.
\end{equation}
\end{prop}

\begin{proof}
We first prove~\eqref{eq:CapaComparison}.
To that end, note that we can also write
\begin{equation}
\label{eq:CapaAlternativeDefWithWeights}
\overline{\capa}(A_1) = \inf\bigg\{ \frac{1}{2}\sum_{x \sim y}\overline{\mu}_{x,y} \Big(g(x) - g(y) \Big)^2 \, : \, g \in C_0(\bbZ^d), \ g(x) = 1 \text{ for }x \in A_1 \bigg\},
\end{equation}
in analogy to~\eqref{eq:CapaAlternativeDef}; see, e.g.,~\cite[Proposition 7.9]{barlow2017random}. By the definition of $\overline{\mu}_{x,y}$ in~\eqref{eq:NonConstantConductances}, we see that for every $x,y \in \bbZ^d$, $x \sim y$, we have
\begin{equation}
\label{eq:WeightsDontMatterMuch}
\frac{\varphi_N^2(x_0)}{2d}\left(1 - \frac{c}{N^{1 - r_5}} \right) \leq \overline{\mu}_{x,y} \leq \frac{\varphi_N^2(x_0)}{2d}\left(1 + \frac{C}{N^{1 - r_5}} \right),
\end{equation}
using the smoothness of $\varphi$ (see Proposition~\ref{prop:quasi-minimizer}, i)). We also have the a-priori bounds for $N$ large enough,
\begin{equation}
\label{eq:BoxCapacity}
cN^{r_1(d-2)} \leq \overline{\capa}(A_1) \leq C N^{r_1(d-2)}, \qquad c'N^{r_1(d-2)}\leq \capa(A_1) \leq C'N^{r_1(d-2)},
\end{equation}
which follow from~\eqref{eq:capacityControlsBox} and the bounds in~\eqref{eq:WeightsDontMatterMuch}. The claim~\eqref{eq:CapaComparison} then follows by combining these bounds with~\eqref{eq:CapaAlternativeDef},
~\eqref{eq:CapaAlternativeDefWithWeights},
~\eqref{eq:WeightsDontMatterMuch}, and using the second line of~\eqref{eq:r_jDef}. \smallskip
 
Let us now turn to the proof of~\eqref{eq:ComparisonEqmeasure}, which is more involved. As a first step, we let $\alpha \in (r_2,1)$, and derive a comparison between the laws $\overline{P}_x$ and $P_x$ for $x \in A_1$ conditioned on the event $\mathcal{E}_\alpha$, defined in~\eqref{eq:NotTooLargeExit}, which is typical under both measures by Lemma~\ref{lem:HKLemma} (and~\eqref{eq:ConfinedWalk_coincides}). We derive a bound on the quotient between  $\overline{P}_x[ A | \mathcal{E}_\alpha]$ and $P_x[A | \mathcal{E}_\alpha]$ for any $A \subseteq \Gamma(U^N)$ measurable that only depends on $X_{\cdot \wedge T_{A_2}}$. We claim that for large enough $N$ and any such $A$,
\begin{equation}
\label{eq:RNAbsBound}
\sup_{x \in A_1} \left\vert\frac{\overline{P}_x[ A | \mathcal{E}_\alpha ]}{P_x[A | \mathcal{E}_\alpha]} - 1\right\vert \leq \frac{c_7(\alpha)}{N^{1-r_2}}.
\end{equation}
(where we define the quotient on the left-hand side as $1$ if $P_x[A | \mathcal{E}_\alpha] = \overline{P}_x[ A | \mathcal{E}_\alpha ] = 0$). Indeed, we observe that for $x \in A_1$ and measurable $A \subseteq \Gamma(U^N)$ as above, 
\begin{equation}
\begin{split}
\label{eq:RadonNikodymComparisonCalc}
\overline{P}_x[A \cap \mathcal{E}_\alpha] & = \overline{E}_x[\mathbbm{1}_{ A \cap \mathcal{E}_\alpha}] \stackrel{\eqref{eq:MartingaleDef}}{=} E_x\left[ \mathbbm{1}_A \mathbbm{1}_{\mathcal{E}_\alpha} \cdot  \frac{\varphi_N(X_{T_{A_2}})}{\varphi_N(x)} \exp\left(\int_0^{T_{A_2}} v_\varphi(X_s)\mathrm{d}s \right)  \right],
\end{split}
\end{equation}
where $v_\varphi(y) = - \frac{\Delta_{\mathbb{Z}^d} \varphi_N(y)}{\varphi_N(y)}$ (recall~\eqref{eq:vDef}).  On the event $\mathcal{E}_\alpha$, by definition $T_{A_2}  \leq N^{2\alpha}$. Moreover, since $|X_{T_{A_2}} - x| \leq cN^{r_2}$, $P_x$-a.s.~for $x \in A_1$ we obtain using the smoothness of $\varphi$ (see Proposition~\ref{prop:quasi-minimizer}, i)) that, $P_x$-a.s.,
\begin{equation}
\left\vert  \frac{\varphi_N(X_{T_{A_2}})}{\varphi_N(x)} \right\vert \leq 1 + \frac{c}{N^{1-r_2}}, \qquad x \in A_1.
\end{equation}
Using~\eqref{eq:ExitTimeDeviation}, recalling that $\alpha \in (r_2,1)$, and~\eqref{eq:RadonNikodymComparisonCalc}, it follows that the quotient of conditional probabilities can be bounded above as follows \begin{equation}
\label{eq:RNUpperBound}
\frac{\overline{P}_x[A | \mathcal{E}_\alpha]}{P_x[ A | \mathcal{E}_\alpha]}  \leq \left(1 + \frac{c}{N^{1 - r_2}} \right) \exp\left(\frac{\widetilde{c}}{N^{2(1-\alpha)}} \right)\left(1 + C\exp\Big(-C'N^{c_6 \wedge c_6'} \Big) \right), 
\end{equation}
where we also used that $ v_\varphi(y) \leq cN^{-2}$ for every $y \in U^N$, see~\eqref{eq:vUpperBound} of Lemma~\ref{lem:TechnicalProperties}, iii). Similarly, we obtain a lower bound 
\begin{equation}
\label{eq:RNLowerBound}
\frac{\overline{P}_x[A | \mathcal{E}_\alpha]}{P_x[ A | \mathcal{E}_\alpha]}  \geq \left(1 - \frac{c}{N^{1 - r_2}} \right) \exp\left(\frac{\widetilde{c}}{N^{2(1-\alpha)}} \right)\left(1 - C\exp\Big(-C'N^{c_6 \wedge c_6'} \Big) \right), 
\end{equation}
where we used that $v_\varphi(y) \geq -c'N^{-2}$, for every $y \in (U_\eta)^N$, see again \eqref{eq:vUpperBound} of Lemma~\ref{lem:TechnicalProperties}, iii). Combining~\eqref{eq:RNUpperBound} and~\eqref{eq:RNLowerBound}, we find that~\eqref{eq:RNAbsBound} holds. Next, we show that returning to $A_1$ from the boundary $\partial A_2$ is unlikely for both the simple random walk and the random walk with conductances $\overline{\mu}$. \smallskip

To that end we will use that for large enough $N$,
\begin{equation}
\label{eq:HittingOfBoxEstimateI}
\begin{split}
P_x[H_{A_1} < \infty] & \stackrel{\eqref{eq:LastExitDecomposition}}{=} \sum_{y \in A_1} g(x,y) e_{A_1}(y) \\
& \stackrel{\eqref{eq:Greenfunction_asymptotics}}{\leq} \frac{c}{N^{r_2(d-2)}} \capa(A_1) \stackrel{\eqref{eq:BoxCapacity}}{\leq} cN^{(d-2)(r_1-r_2)}, \qquad x \in \partial A_2.
\end{split}
\end{equation}
Moreover, we also have
\begin{equation}
\begin{split}
\label{eq:LastExitDecompisitionTilted}
P^{\overline{\mu}}_{x}[H_{A_1}< \infty] & = P^{\overline{\mu};\overline{M}}_{x}[H_{A_1}< \infty] = \sum_{y \in A_1} g^{\overline{\mu};\overline{M}}(x,y) \overline{e}_{A_1}(y), \qquad x \in \partial A_2,
\end{split}
\end{equation}
(recall the definition of $\overline{M}$ in~\eqref{eq:NaturalMMeasureDef}) in analogy to~\eqref{eq:LastExitDecomposition}; see, e.g.,~\cite[Corollary
7.4]{barlow2017random}, using also in the first equation the fact that the random walk under $P^{\overline{\mu}}_{x}$ is a (random) time-change of the random walk under $P^{\overline{\mu};\overline{M}}_{x}$, see the discussion below~\eqref{eq:StandardTotalMass}. Moreover, the Green function $g^{\overline{\mu};\overline{M}}$ is defined in~\eqref{eq:GeneralGreenFunctionDef}, and we also have an upper bound
\begin{equation}
\label{eq:TiltedGreenFunctionUpperBound}
g^{\overline{\mu};\overline{M}}(x,y) \leq \frac{C}{(|x-y|\vee 1)^{d-2}}, \qquad x,y \in \bbZ^d
\end{equation}
(see, e.g.,~\cite[Theorem 4.26]{barlow2017random}). Thus, by inserting~\eqref{eq:TiltedGreenFunctionUpperBound} in~\eqref{eq:LastExitDecompisitionTilted}
and using~\eqref{eq:BoxCapacity},  one can conclude similarly as for~\eqref{eq:HittingOfBoxEstimateI} that
\begin{equation}
\label{eq:HittingOfBoxEstimateII}
P^{\overline{\mu}}_{x}[H_{A_1}< \infty] \leq cN^{(d-2)(r_1-r_2)}, \qquad  x \in \partial A_2.
\end{equation}
We now have the required preparations to prove~\eqref{eq:ComparisonEqmeasure}. Let $x\in \partial_{\mathrm{int}}A_1$, then
\begin{equation}
\label{eq:MultipleBoundsEqMeasureQuotient}
\begin{split}
\left\vert \frac{P^{\overline{\mu}}_x[\widetilde{H}_{A_1} = \infty]}{e_{A_1}(x)} -1 \right\vert & = \frac{|P^{\overline{\mu}}_x[\widetilde{H}_{A_1} = \infty] - P_x[\widetilde{H}_{A_1} = \infty]|}{P_x[\widetilde{H}_{A_1} = \infty]} \\
& \leq \frac{|P^{\overline{\mu}}_x[\widetilde{H}_{A_1} > T_{A_2} ] - P_x[\widetilde{H}_{A_1} > T_{A_2}]|}{P_x[\widetilde{H}_{A_1} = \infty]} \\
& + \frac{ \sup_{z \in \partial A_2} P^{\overline{\mu}}_z[\widetilde{H}_{A_1} < \infty] + \sup_{z \in \partial A_2} P_z[\widetilde{H}_{A_1}   < \infty]}{P_x[\widetilde{H}_{A_1} = \infty]} \\
  &  \hspace{-0.76cm} \stackrel{\eqref{eq:LowerBoundEqMeasure},\eqref{eq:HittingOfBoxEstimateI},\eqref{eq:HittingOfBoxEstimateII}}{\leq} \frac{|P^{\overline{\mu}}_x[\widetilde{H}_{A_1} > T_{A_2} ] - P_x[\widetilde{H}_{A_1} > T_{A_2}]|}{P_x[\widetilde{H}_{A_1} = \infty]}  + \frac{c N^{(d-2)(r_1-r_2)}}{N^{-r_1}} \\
& \hspace{-0.12cm} \stackrel{\eqref{eq:r_jDef}}{\leq }  \frac{|P^{\overline{\mu}}_x[\widetilde{H}_{A_1} > T_{A_2} ] - P_x[\widetilde{H}_{A_1} > T_{A_2}]|}{P_x[\widetilde{H}_{A_1} = \infty]} + \frac{c}{N^{c'}},
\end{split}
\end{equation}
using also the strong Markov property at time $T_{A_2}$ ($< \infty$, $P_x$-a.s.~and $P^{\overline{\mu}}_x$-a.s.~due to transience) in the first bound. We now turn to the remaining term in the last line of~\eqref{eq:MultipleBoundsEqMeasureQuotient} and note that the events under the probability only depend on $X_{\cdot \wedge T_{A_2}}$. We can therefore for $x \in \partial_{\mathrm{int}}A_1$ (recall~\eqref{eq:ConfinedWalk_coincides}) use
\begin{equation}
\label{eq:MultipleBoundsEqMeasureQuotientII}
\begin{split}
& \frac{|P^{\overline{\mu}}_x[\widetilde{H}_{A_1} > T_{A_2} ] - P_x[\widetilde{H}_{A_1} > T_{A_2}]|}{P_x[\widetilde{H}_{A_1} = \infty]} \\
& \leq \frac{|P^{\overline{\mu}}_x[\widetilde{H}_{A_1} > T_{A_2}|\mathcal{E}_\alpha ] \cdot P^{\overline{\mu}}_x[\mathcal{E}_\alpha] - P_x[\widetilde{H}_{A_1} > T_{A_2} |\mathcal{E}_\alpha ] \cdot P_x[\mathcal{E}_\alpha] |}{P_x[\widetilde{H}_{A_1} = \infty]} + \frac{P^{\overline{\mu}}_x[\mathcal{E}_\alpha^c] + P_x[\mathcal{E}_\alpha^c]}{P_x[\widetilde{H}_{A_1} = \infty]} \\
& \leq \frac{|P^{\overline{\mu}}_x[\widetilde{H}_{A_1} > T_{A_2}|\mathcal{E}_\alpha ] - P_x[\widetilde{H}_{A_1} > T_{A_2} |\mathcal{E}_\alpha ] |}{P_x[\widetilde{H}_{A_1} = \infty]} +2 \frac{P^{\overline{\mu}}_x[\mathcal{E}_\alpha^c] + P_x[\mathcal{E}_\alpha^c]}{P_x[\widetilde{H}_{A_1} = \infty]} \\
& \hspace{-0.4cm} \stackrel{\eqref{eq:LowerBoundEqMeasure},\eqref{eq:ExitTimeDeviation}}{\leq} \frac{|P^{\overline{\mu}}_x[\widetilde{H}_{A_1} > T_{A_2}|\mathcal{E}_\alpha ] - P_x[\widetilde{H}_{A_1} > T_{A_2} |\mathcal{E}_\alpha ] |}{P_x[\widetilde{H}_{A_1} = \infty]} + \frac{Ce^{-cN^{c'}}}{N^{-r_1}} \\
& \hspace{-0.4cm} \stackrel{\eqref{eq:LowerBoundEqMeasure},\eqref{eq:RNAbsBound}}{\leq} \frac{c\cdot  c_7(\alpha)}{N^{1-r_2} \cdot N^{-r_1}} +  C{N^{-r_1}}e^{-cN^{c'}}.
\end{split}
\end{equation}
By~\eqref{eq:r_jDef}, we have $r_1 + r_2 < \frac{1}{2}$, and therefore the claim follows by combining the terms on the right-hand side of~\eqref{eq:MultipleBoundsEqMeasureQuotient} and~\eqref{eq:MultipleBoundsEqMeasureQuotientII}.
\end{proof}

One can combine the two statements in Proposition~\ref{prop:ComparisonEqmeasure} to obtain a control on the normalized equilibrium measures.

\begin{corollary}
For large enough $N$, we have
\begin{equation}
\label{eq:NormalizedEqMeasureClose}
\sup_{x \in \partial_{\mathrm{int}} A_1} \left\vert \frac{\widetilde{\overline{e}}_{A_1}(x)}{\widetilde{e}_{A_1}(x)} -1 \right\vert \leq \frac{C}{N^{c}}.
\end{equation}
\end{corollary}
\begin{proof}
We take $x \in \partial_{\mathrm{int}}A_1$. We recall the definition~\eqref{eq:NaturalMMeasureDef} of $\overline{M}_x$, $x \in \bbZ^d$, and note that by~\eqref{eq:WeightsDontMatterMuch}, we have
\begin{equation}
\label{eq:MandVarphi2Close}
\sup_{x \in \bbZ^d} \left\vert \overline{M}_{x} - \varphi_N^2(x_0) \right\vert \leq \frac{C}{N^{1-r_5}}.
\end{equation}
Thus, for large enough $N$,
\begin{equation}
\begin{split}
\left\vert \frac{\widetilde{\overline{e}}_{A_1}(x)}{\widetilde{e}_{A_1}(x)} -1 \right\vert & \stackrel{\eqref{eq:TiltedEqMeasureDefinition}}{=} \frac{1}{\overline{\capa}(A_1)} \left\vert \frac{\overline{M}_x P^{\overline{\mu}}_x[\widetilde{H}_{A_1} = \infty] }{e_{A_1}(x)} \cdot \capa(A_1)   - \overline{\capa}(A_1) \right\vert \\
& \leq \frac{\overline{M}_x\capa(A_1)}{\overline{\capa}(A_1)} \cdot \left\vert \frac{P_x^{\overline{\mu}}[\widetilde{H}_{A_1} = \infty] }{e_{A_1}(x)} - 1 \right\vert \\
& + \frac{1}{\overline{\capa}(A_1)}\cdot  \left\vert\overline{M}_x \capa(A_1) - \overline{\capa}(A_1) \right\vert \\
& \stackrel{\eqref{eq:BoxCapacity},\eqref{eq:MandVarphi2Close}}{\leq}C \cdot \left\vert \frac{P_x^{\overline{\mu}}[\widetilde{H}_{A_1} = \infty] }{e_{A_1}(x)} - 1 \right\vert +  \frac{C}{N^{1-r_5}} \\
& + \frac{1}{\overline{\capa}(A_1)}\cdot  \left\vert \varphi_N^2(x_0) \capa(A_1) - \overline{\capa}(A_1) \right\vert  \\
& \stackrel{\eqref{eq:ComparisonEqmeasure},\eqref{eq:CapaComparison},\eqref{eq:BoxCapacity}}{\leq } \frac{C}{N^{c}},
\end{split}
\end{equation}
which concludes the proof.
\end{proof}

\subsection{Quasi-stationary distribution of the confined walk}

We now turn to the definition of the quasi-stationary distribution for the confined walk on $U^N$ and state some fundamental estimates relating it to the distribution of the confined walk conditioned on avoiding a certain set for long enough times. To this end, we first need to introduce some further notation.   We define the semigroup of the confined walk $(\overline{P}_x)_{x \in U^N}$ killed upon entering $A_2$ (recall the definitions in~\eqref{eq:zetaAndx_0}--\eqref{eq:AjBoxesDef}) by setting
\begin{equation}
\label{eq:SemigroupDef}
\mathscr{H}^{U^N \setminus A_2}_t g(x) = \overline{E}_x\big[g(X_t) \mathbbm{1}_{\{H_{A_2} > t \}} \big], \qquad t \geq 0, \ g : U^N \rightarrow \mathbb{R},
\end{equation}
and let $\mathscr{L}^{U^N \setminus A_2}$ stand for its generator. For a function $g :U^N \setminus A_2 \rightarrow \bbR$, one has 
\begin{equation}
\mathscr{L}^{U^N \setminus A_2} g(x) = \mathscr{L}_f \widetilde{g}(x), \qquad x \in U^N\setminus A_2,
\end{equation}
where $\widetilde{g}$ denotes the the extension by $0$ in $A_2$ of the function $g$, i.e.
\begin{equation}
\widetilde{g}(x) = g(x)\mathbbm{1}_{ A_2^c}(x), \qquad x\in U^N,
\end{equation}
and $\mathscr{L}_f$ is defined in~\eqref{eq:GeneratorVSRW} (with $f$ as in~\eqref{eq:varphi_N_Def}). We also set $\pi^{U^N \setminus A_2} = \pi\vert_{U^N \setminus A_2}$, the restriction of the measure $\pi$ to $U^N \setminus A_2$. Both $\mathscr{H}^{U^N \setminus A_2}_t$, $t \geq 0$, and $\mathscr{L}^{U^N \setminus A_2}$ are self-adjoint operators on $\ell^2(U^N \setminus A_2,\pi^{U^N \setminus A_2})$, and one has 
\begin{equation}
\mathscr{H}^{U^N \setminus A_2}_t = \exp\big(t \mathscr{L}^{U^N \setminus A_2}\big), \qquad t \geq 0.
\end{equation}
Moreover, since $U^N\setminus A_2$ is connected, by the Perron-Frobenius theorem, we can associate to the smallest eigenvalue $\lambda_1^{U^N \setminus A_2}$ of the positive definite operator $-\mathscr{L}^{U^N \setminus A_2}$ an eigenfunction $f_1 : U^N \setminus A_2 \rightarrow [0,\infty)$ (i.e.~with non-negative values). We then define, with $\boldsymbol{1}$ denoting the function that is constant and equal to one in every point and $\delta_y = \mathbbm{1}_{\{y\}}$,
\begin{equation}
\label{eq:qsd_Def}
\sigma(y) = \frac{\langle f_1,\delta_y  \rangle_{\ell^2(U^N \setminus A_2,\pi^{U^N \setminus A_2})}}{\langle f_1, \boldsymbol{1} \rangle_{\ell^2(U^N \setminus A_2,\pi^{U^N \setminus A_2})}} \left(  = \frac{f_1(y) \pi(y) }{\sum_{z \in U^N \setminus A_2} f_1(z) \pi(z) } \right), \qquad y \in U^N \setminus A_2,
\end{equation}
which we call the quasi-stationary distribution of $(\overline{P}_x)_{x \in U^N \setminus A_2}$ on $U^N \setminus A_2$.

The following is an adaptation of~\cite[Proposition 4.5]{li2017lower} to our framework, and relates the quasi-stationary distribution of the confined walk quantitatively to its distribution at time $t_\star$ (see~\eqref{eq:RegenerationTimeDef}) starting from $x \in U_N \setminus A_2$, conditioned on the event that it has not entered $A_2$ before this time.

\begin{prop}
\label{prop:ClosenessToQSD}
For $N$ large enough, we have that 
\begin{equation}
\label{eq:ConditionalDistributionCloseToqsd}
\sup_{x,y \in U^N \setminus A_2} \big\vert \overline{P}_x[X_{t_\star} = y| H_{A_2} > t_\star] - \sigma(y) \big\vert \leq \exp\big(-c \log^2 N\big).
\end{equation}
\end{prop}

The proof is given in the Appendix~\ref{sec:Appendix}.

We will now show a uniform comparison between the hitting distribution of $A_1$ from the quasi-stationary distribution $\sigma$ on $U^N \setminus A_2$ and $\widetilde{e}_{A_1}$, the normalized equilibrium measure of the simple random walk. The statement is that of~\cite[Proposition 4.7]{li2017lower}, which itself is roughly an adaptation of~\cite[Lemma 3.10]{teixeira2011fragmentation} to the case of the confined walk. However, special care is needed in our case since the confined walk attached to the minimizer $\varphi_N$ does \textit{not} necessarily admit areas in which the conductances are constant.

\begin{prop}
\label{prop:AnalogProp4.7}
For large $N$, we have
\begin{equation}
\sup_{x \in \partial_{\mathrm{int}}A_1} \left\vert \frac{\overline{P}_\sigma[X_{H_{A_1}} = x]}{\widetilde{e}_{A_1}(x)} -1 \right\vert \leq \frac{1}{N^c}.
\end{equation}
\end{prop}

\begin{proof}
In view of~\eqref{eq:NormalizedEqMeasureClose}, it suffices to show that 
\begin{equation}
\label{eq:ClosenessToTiltedEqmeasure}
\sup_{x \in \partial_{\mathrm{int}}A_1} \left\vert \frac{\overline{P}_\sigma[X_{H_{A_1}} = x]}{\widetilde{\overline{e}}_{A_1}(x)} -1 \right\vert \leq \frac{1}{N^c}.
\end{equation}
%Indeed, suppose that~\eqref{eq:ClosenessToTiltedEqmeasure} holds and let $x \in A_1$. Then, for $N$ large enough,
%\begin{equation}
%\begin{split}
%\left\vert \frac{\overline{P}_\sigma[X_{H_{A_1}} = x]}{\widetilde{e}_{A_1}(x)} -1 \right\vert &\leq \left\vert \frac{\overline{P}_\sigma[X_{H_{A_1}} = x]}{\widetilde{\overline{e}}_{A_1}(x)} -1 \right\vert + \frac{\overline{P}_\sigma[X_{H_{A_1}} = x]}{\widetilde{\overline{e}}_{A_1}(x)} \cdot \left\vert \frac{\widetilde{\overline{e}}_{A_1}(x)}{\widetilde{e}_{A_1}(x)} -1  \right\vert \\
%& \stackrel{\eqref{eq:NormalizedEqMeasureClose},\eqref{eq:ClosenessToTiltedEqmeasure}}{\leq} \frac{1}{N^c} + \left(1 + \frac{1}{N^c}\right)\cdot \frac{C}{N^{c'}},
%\end{split}
%\end{equation}
%and the claim follows. \smallskip
To this end, we first establish that for $N$ large enough, $x \in\partial_{\mathrm{int}}A_1$, and the stopping time
\begin{equation}
\label{eq:V_time_def}
V = \inf\{t \geq t_\star \, : \, X_{[t - t_\star,t]} \cap A_2 = \varnothing \},
\end{equation}
we have the bound
\begin{equation}
\label{eq:V_comparisonNormalizedEqmeasure}
\left\vert \frac{\overline{P}_x[V < \widetilde{H}_{A_1}]}{\sum_{y \in \partial_{\mathrm{int}}A_1 } \overline{P}_y[V < \widetilde{H}_{A_1}] \widetilde{\overline{e}}_{A_1}(x) } - 1  \right\vert \leq \frac{1}{N^c},
\end{equation}
which follows in the same way as~\cite[(4.56)]{li2017lower}. Indeed, one first uses~\eqref{eq:HittingUpperBounds} in place of (3.11) and~\eqref{eq:SomeBoundsEqmeasure} in place of (3.12) of the same reference, giving the bounds
\begin{equation}
\label{eq:TiltedEqmeasureControls}
\left(1 - \frac{1}{N^c}\right) \frac{\overline{e}_{A_1}(y)}{\overline{M}_y} \leq \overline{P}_y[V < \widetilde{H}_{A_1}]  \leq \left(1+ \frac{1}{N^{c'}}\right)\frac{\overline{e}_{A_1}(y)}{\overline{M}_y}, \qquad y \in \partial_{\mathrm{int}}A_1.
\end{equation}
Using the smoothness of $\varphi$ (see Proposition~\ref{prop:quasi-minimizer}, i)) to replace $\overline{M}_x$ by $\overline{M}_y$ uniformly for $x,y \in \partial_{\mathrm{int}}A_1$ up to an error bounded by $N^{-{c''}}$, and then summing over $y \in \partial_{\mathrm{int}}A_1$ yields the comparison, for large enough $N$,
\begin{equation}
\label{eq:TiltedCapacityControls}
\left(1 - \frac{1}{N^c}\right)\overline{M}_x \sum_{y \in \partial_{\mathrm{int}} A_1}  \overline{P}_y[V < \widetilde{H}_{A_1}] \leq \overline{\capa}(A_1) \leq \left(1 + \frac{1}{N^{c'}}\right) \overline{M}_x \sum_{y \in \partial_{\mathrm{int}} A_1} \overline{P}_y[V < \widetilde{H}_{A_1}],
\end{equation}
and the claim~\eqref{eq:V_comparisonNormalizedEqmeasure} follows upon using the definition of the normalized equilibrium measure~\eqref{eq:TiltedNormalizedEqmeasure} and combining~\eqref{eq:TiltedEqmeasureControls}
and~\eqref{eq:TiltedCapacityControls}. \smallskip

To conclude the proof of~\eqref{eq:ClosenessToTiltedEqmeasure}, it now suffices to prove that for every $x \in \partial_{\mathrm{int}}A_1$ and $N$ large enough 
\begin{equation}
\label{eq:Appendixli2017MainClaim}
\left\vert \overline{P}_x[V < \widetilde{H}_{A_1}] - \overline{P}_\sigma[X_{H_{A_1}} =x]\sum_{y \in \partial_{\mathrm{int}} A_1 } \overline{P}_y[V < \widetilde{H}_{A_1}] \right\vert \leq \frac{C}{N^{1-r_1}}.
\end{equation}
We admit~\eqref{eq:Appendixli2017MainClaim} for the time being and explain how the proof proceeds given this. First note that upon combining~\eqref{eq:LowerBoundEqMeasure},
~\eqref{eq:ComparisonEqmeasure}, and~\eqref{eq:TiltedCapacityControls}, we obtain for every $x \in \partial_{\mathrm{int}}A_1$ and $N$ large enough (and using also that $c \leq \overline{M}_x \leq C$ by the definition~\eqref{eq:NaturalMMeasureDef},~\eqref{eq:min_phi_strictlypositive}, and~\eqref{eq:Phi_N_TechnicalProp}),
\begin{equation}
\begin{split}
\sum_{y \in \partial_{\mathrm{int}}A_1 } \overline{P}_y[V < \widetilde{H}_{A_1}] \widetilde{\overline{e}}_{A_1}(x) & \stackrel{\eqref{eq:TiltedCapacityControls}}{\geq} \frac{C \, \overline{e}_{A_1}(x)}{\left(1 + \frac{1}{N^c} \right)}  \stackrel{\eqref{eq:ComparisonEqmeasure}}{\geq} C'\left(1 - \frac{1}{N^{c'}} \right)e_{A_1}(x)  \stackrel{\eqref{eq:LowerBoundEqMeasure}}{\geq} \frac{C''}{N^{r_1}}.
\end{split}
\end{equation}
One can therefore divide, for large enough $N$, the inequality~\eqref{eq:Appendixli2017MainClaim} by $\sum_{y \in \partial_{\mathrm{int}}A_1 } \overline{P}_y[V < \widetilde{H}_{A_1}]\widetilde{\overline{e}}_{A_1}(x) $ and infer that
\begin{equation}
\begin{split}
\left\vert \frac{\overline{P}_x[V < \widetilde{H}_{A_1}]}{\sum_{y \in \partial_{\mathrm{int}}A_1 } \overline{P}_y[V < \widetilde{H}_{A_1}] \widetilde{\overline{e}}_{A_1}(x) } - \frac{\overline{P}_\sigma[X_{H_{A_1}} = x]}{\widetilde{\overline{e}}_{A_1}(x)}  \right\vert & \leq \frac{C}{N^{1-2r_1}}  \stackrel{\eqref{eq:r_jDef}}{\leq} \frac{C}{N^{c'}},
\end{split}
\end{equation}
and combining this with~\eqref{eq:V_comparisonNormalizedEqmeasure} establishes~\eqref{eq:ClosenessToTiltedEqmeasure} and therefore the main statement. \smallskip

To show~\eqref{eq:Appendixli2017MainClaim}, we will argue as in~\cite[Appendix]{li2017lower}, but some care is required due to the fact that $\pi$ is not constant on $\partial_{\mathrm{int}}A_1$. Crucially, we have that
\begin{equation}
\label{eq:ReversibilityWithAdditionalFactor}
\sum_{y \in \partial_{\mathrm{int}}A_1 \setminus \{x\}} \frac{\pi(y)}{\pi(x)} \overline{P}_x[V < \widetilde{H}_{A_1}, X_{ \widetilde{H}_{A_1}} = y] = \sum_{y \in \partial_{\mathrm{int}}A_1 \setminus \{x\}}  \overline{P}_{y}[V < \widetilde{H}_{A_1}, X_{ \widetilde{H}_{A_1}} = x],
\end{equation} 
for all $y \in \partial_{\mathrm{int}}A_1$, replacing~\cite[(A.2)]{li2017lower} (in our context, we have the extra factor $\frac{\pi(y)}{\pi(x)}$ appearing, which was equal to $1$ in~\cite[Appendix]{li2017lower}). However, upon using the smoothness of $\varphi$ again, we see that for large enough $N$,
\begin{equation}
\label{eq:StationaryMeasureFlat}
\sup_{x,y \in \partial_{\mathrm{int}}A_1}\left\vert \frac{\pi(y)}{\pi(x)} - 1 \right\vert \leq \frac{c}{N^{1 - r_1}},
\end{equation}
and therefore
\begin{equation}
\label{eq:DifferenceBetweenWithAndWithoutPi}
\begin{split}
& \left\vert \sum_{y \in \partial_{\mathrm{int}}A_1 \setminus \{x\}} \frac{\pi(y)}{\pi(x)} \overline{P}_x[V < \widetilde{H}_{A_1}, X_{\widetilde{H}_{A_1}} = y] - \sum_{y \in \partial_{\mathrm{int}}A_1 \setminus \{x\}}  \overline{P}_x[V < \widetilde{H}_{A_1}, X_{\widetilde{H}_{A_1}} = y] \right\vert  \stackrel{\eqref{eq:StationaryMeasureFlat}}{\leq} \frac{c}{N^{1 - r_1}}.
\end{split}
\end{equation}
We can then directly follow~\cite[(A.4) -- (A.8)]{li2017lower}, using~\eqref{eq:ConditionalDistributionCloseToqsd}  to obtain for every $x',y' \in \partial_{\mathrm{int}}A_1$ that
\begin{equation}
\label{eq:AlmostIndependenceStatementqsd}
\left\vert \overline{P}_{x'}[V < \widetilde{H}_{A_1},X_{\widetilde{H}_{A_1}} = y']- \overline{P}_{x'}[V < \widetilde{H}_{A_1}] \overline{P}_\sigma[X_{\widetilde{H}_{A_1}} = y'] \right\vert \leq e^{-c\log^2N}.
\end{equation}
We therefore obtain (recall that $|A_1| \leq cN^{r_1d}$) for large enough $N$
\begin{equation}
\label{eq:AlmostDone}
\begin{split}
& \left\vert \overline{P}_x[V < \widetilde{H}_{A_1}]\overline{P}_\sigma[X_{\widetilde{H}_{A_1}} \neq x ] - \sum_{y \in \partial_{\mathrm{int}}A_1 \setminus \{x\}} \overline{P}_y[V < \widetilde{H}_{A_1}]\overline{P}_\sigma[X_{\widetilde{H}_{A_1}} = x ]  \right\vert \\
& \stackrel{\eqref{eq:AlmostIndependenceStatementqsd}}{\leq} \left\vert \sum_{y \in \partial_{\mathrm{int}}A_1 \setminus \{x\}}  \overline{P}_x[V < \widetilde{H}_{A_1}, X_{\widetilde{H}_{A_1}} = y] - \sum_{y \in \partial_{\mathrm{int}}A_1 \setminus \{x\}}  \overline{P}_y[V < \widetilde{H}_{A_1}, X_{\widetilde{H}_{A_1}} = x]  \right\vert \\
& \qquad + e^{-c'\log^2N} \\
& \stackrel{\eqref{eq:ReversibilityWithAdditionalFactor},\eqref{eq:DifferenceBetweenWithAndWithoutPi}}{\leq} \frac{c}{N^{1-r_1}}.
\end{split}
\end{equation}
The claim~\eqref{eq:Appendixli2017MainClaim} then follows by adding and subtracting the term $\overline{P}_x[V < \widetilde{H}_{A_1}]\overline{P}_\sigma[X_{\widetilde{H}_{A_1}} = x ]$ under the absolute value on the right-hand side of~\eqref{eq:AlmostDone}.
\end{proof}

\section{Coupling of the occupation times of excursions}
\label{sec:coupling}

The purpose of this section is to introduce a pivotal series of (local) couplings for the occupation times of the tilted walk with the occupation times of a random number of independent simple random walk excursions in certain mesoscopic boxes. This series of couplings will be instrumental in Section~\ref{sec:lowerbound} to demonstrate that under the tilted measure $\widetilde{P}_{y,N}$, $y \in \bbZ^d$, the probability of the event  $\cA^\nu_N(F)$ tends to one, as $N$ tends to infinity. Similar couplings for the trace of the tilted walk with the traces of random interlacements have been devised in~\cite{li2017lower}, adapting a previous coupling for the trace of the simple random walk on the torus with random interlacements of a similar type in~\cite{teixeira2011fragmentation}. For our purposes, we need to adapt the framework of~\cite{li2017lower} further, in particular by coupling the occupation time field instead of the trace of the tilted walk. \smallskip

Throughout this section, we work again with a fixed near-minimizer $\varphi$ as in~\eqref{eq:FixedPhi} and fixed $\delta,\cR$ as in~\eqref{eq:FixedChoiceR}, and all constants may depend implicitly on these choices. We also recall the definitions of the boxes $A_j$, $1 \leq j \leq 6$ in~\eqref{eq:AjBoxesDef}. To start with, we define successive entrance times into the box $A_1$ centered at $x_0 \in D^\delta_N$ (see~\eqref{eq:zetaAndx_0}), which are separated by intervals in which the tilted walk spends a substantial time outside $A_2$. Specifically, recall the definition of the regeneration time in~\eqref{eq:RegenerationTimeDef} and the stopping time $V$ in~\eqref{eq:V_time_def}. These are now used to define the aforementioned successive entrance times $R_k$ and stopping times $V_k$ as follows:
\begin{equation}
\label{eq:Stopping-times-R-V-def}
\begin{split}
R_1 & = H_{A_1}, \qquad V_1 = R_1 + V \circ \vartheta_{R_1}, \\
R_j & = V_{j-1} + H_{A_1} \circ \vartheta_{V_{j-1}}, \qquad V_j = R_j + V \circ \vartheta_{R_j}, \text{ for }j \geq 2.
\end{split}
\end{equation}
The values of the (tilted) walk within the time intervals $[R_j,V_j)$ will be called ``long excursions''. \smallskip
\begin{figure}[htbp]
\begin{center}
\includegraphics[scale=.7]{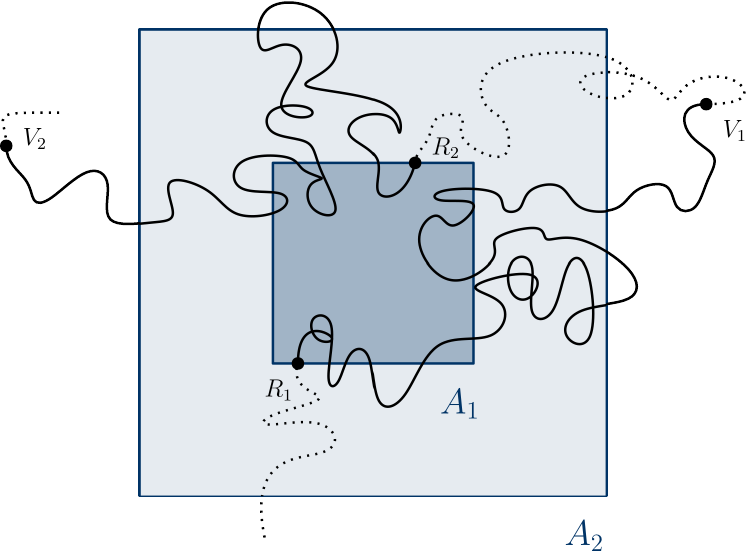}
\end{center}
\caption{An illustration of the first two (long) excursions of the tilted walk in and out of $A_1$.}
\label{fig:excursions}
\end{figure}

We define the quantity
\begin{equation}
\label{eq:JDef}
J = \left\lfloor (1 + \tfrac{\varepsilon}{2})\varphi_N^2(x_0) \capa(A_1) \right\rfloor,
\end{equation}
where $\varepsilon$ is defined in~\eqref{eq:time_horizon} and $\varphi_N$ is defined as in~\eqref{eq:varphi_N_Def}. We remark that by~\eqref{eq:CapaComparison}, the value $\varphi_N^2(x_0)\capa(A_1)$ is close to $\overline{\capa}(A_1)$ for $N$ large. The number $J$ will be used to control the number of (long) excursions to $A_1$ that the tilted walk typically performs before time $S_N$. This is the content of the following statement.
\begin{prop}
\label{prop:TiltedWalkExcursionNumber}
For large enough $N$, and $y \in U^N$, one has
\begin{equation}
\label{eq:EnoughTrajectories}
\overline{P}_{y}[R_J \geq S_N] \leq \exp(-N^c).
\end{equation}
\end{prop}
\begin{proof} The proof is similar to~\cite[Appendix]{li2017lower} (see also the proof of~\cite[Lemma 4.3]{teixeira2011fragmentation}), and could be skipped upon first reading. Observe that $\overline{P}_y$-a.s.\ 
\begin{equation}
    \begin{aligned}
        \{R_J \geq S_N\} &\subseteq  \Big\{ H_{A_1} + H_{A_1}\circ \vartheta_{V_1} + \ldots + H_{A_1}\circ \vartheta_{V_{J-1}} \geq \bigg(1-\frac{\varepsilon}{100}\bigg) S_N\Big\}\\
        &\cup \Big\{V\circ \vartheta_{R_1} + \ldots + V\circ \vartheta_{R_{J-1}} \geq \frac{\varepsilon}{100} S_N\Big\}.    
    \end{aligned}
\end{equation}
 We want to give now an upper bound on the respective probabilities. We start defining the quantity
 \begin{equation}
    t_N = \sup_{y\in U_N} \overline{E}_y[H_{A_1}],
 \end{equation}   
i.e.~the maximum of the expected time to enter $A_1$ starting from any $y\in U_N$. 
By means of a repeated use of the strong Markov property and the exponential Chebyshev inequality (at times $V_1,\ldots,V_{J-1}$ and $R_1,\ldots,R_{J-1}$) we have for any $\vartheta>0$
 \begin{equation}
 \label{eq:TermsToBoundKhasminskii}
    \begin{aligned}
        \overline{P}_y[R_J \geq S_N] &\leq \overline{P}_y\Big[H_{A_1} + H_{A_1}\circ \vartheta_{V_1} + \ldots + H_{A_1}\circ \vartheta_{V_{J-1}} \geq \bigg(1-\frac{\varepsilon}{100}\bigg) S_N\Big]\\&\qquad+ \overline{P}_y\Big[V\circ \vartheta_{R_1} + \ldots + V\circ \vartheta_{R_{J-1}} \geq \frac{\varepsilon}{100} S_N\Big]\\
        &\leq \exp\Big(-\vartheta\Big( 1- \frac{\varepsilon}{100}\Big) \frac{S_N}{t_N}\Big) \Big(\sup_{z\in U_N} \overline{E}_z\Big[\exp\Big(\vartheta \frac{H_{A_1}}{t_N}\Big)\Big]\Big)^J\\
        &\qquad+\exp\Big(-  \frac{\varepsilon}{100} \frac{S_N}{t_N}\Big) \Big(\sup_{z\in {A_1}} \overline{E}_z\Big[\exp\Big( \frac{V}{t_N}\Big)\Big]\Big)^J.
    \end{aligned}
 \end{equation}
 The first term will be handled by means of Khas'minskii's inequality (see, e.g.,~\cite[Lemma 2.1, p.8]{sznitman1998brownian}), which implies that for any $B \subseteq U^N$ 
 \begin{equation}
 \label{eq:Khasminskii}
 \sup_{z \in U^N} \overline{E}_z\big[H_{B}^n \big] \leq n! \sup_{z \in U^N} \overline{E}_z[H_{B}]^n,  \qquad n \geq 1.
 \end{equation}
 This implies that for $\vartheta \in (0,\frac{1}{2})$, one has
 \begin{equation}
 \label{eq:UniformKhasminski-A}
 \sup_{z\in U_N} \overline{E}_z\Big[\exp\Big(\vartheta \frac{H_{A_1}}{t_N}\Big)\Big] \leq \sum_{j =0 }^\infty \frac{\vartheta^j}{j!t_N^j}\sup_{z \in U^N} \overline{E}_z\big[H_{A_1}^j \big]  \stackrel{\eqref{eq:Khasminskii}}{\leq} \sum_{j = 0}^\infty \vartheta^j = \frac{1}{1-\vartheta}.
 \end{equation}
 We now bound the second term on the right-hand side  of~\eqref{eq:TermsToBoundKhasminskii}. Proceeding as in~\cite[(A.16)--(A.17)]{li2017lower} or~\cite[(4.17)]{teixeira2011fragmentation} and the display above, we have for $N$ large enough
 \begin{equation}
 \label{eq:KhasminskiiLemmaMainStep}
 \begin{split}
 \sup_{z\in {A_2}} \overline{E}_z\Big[\exp\Big( \frac{V}{t_N}\Big)\Big] & \leq \sup_{z \in A_2} \overline{E}_z\Big[e^{\frac{T_{A_3} + t_\star}{t_N}} \Big]\left(1+ \sup_{z \in U^N \setminus A_3} \overline{P}_z[H_{A_2} \leq t_\star] \sup_{z \in A_2}\overline{E}_z\Big[e^{\frac{V}{t_N}}\Big]  \right) \\
 & \stackrel{\eqref{eq:HittingUpperBounds}}{ \leq} \sup_{z \in A_2} \overline{E}_z\Big[e^{\frac{T_{A_3} + t_\star}{t_N}} \Big]\left( 1 + \frac{1}{N^c}\sup_{z \in A_2}\overline{E}_z\Big[e^{\frac{V}{t_N}}\Big] \right).
 \end{split}
 \end{equation}
 We then have, for $N$ large enough,
 \begin{equation}
 \frac{1}{t_N} \leq \frac{1}{\overline{E}_\pi[H_{A_1}]} \stackrel{\eqref{eq:ImportantHittingTimeEstimate}}{\leq} \frac{1}{N^{d-(d-2)r_1}}.
 \end{equation}
 Also, for every $x \in A_3$,
 \begin{equation}
 \overline{E}_x[T_{A_3}]  \stackrel{\eqref{eq:ConfinedWalk_coincides}}{=} E^{\overline{\mu}}_x[T_{A_3}] \leq cN^{2r_3},
 \end{equation}
 which follows from~\cite[Lemmas 5.20, 5.22]{barlow2017random}, since we have the Gaussian heat kernel bounds from Lemma~\ref{lem:NewHK}. Therefore, we find for $N$ large enough,
 \begin{equation}
 \frac{\sup_{x \in A_3} \overline{E}_x[T_{A_3}]}{t_N} \leq \frac{c}{N^{d-2r_3-(d-2)r_1}} \leq \frac{1}{N^{c'}},
 \end{equation}
and another application of Khas'minskii's inequality~\eqref{eq:Khasminskii} gives
\begin{equation}
\sup_{x \in A_2} \overline{E}_x\left[\exp\left(\frac{T_{A_3}}{t_N} \right) \right] \leq \frac{1}{1 - N^{-c'}}.
\end{equation}
Proceeding as in~\cite[(A.21)--(A.23)]{li2017lower} or~\cite[(4.19)]{teixeira2011fragmentation} then gives (since $\frac{t_\star}{t_N} \leq cN^{-c'}$) that, for $N$ large enough,
\begin{equation}
 \label{eq:UniformKhasminski-B}
\sup_{x \in A_2} \overline{E}_x\left[\exp\left(\frac{V}{t_N}\right) \right]  \leq \exp\left(\frac{1}{N^c}\right).
\end{equation}
Inserting~\eqref{eq:UniformKhasminski-A} and~\eqref{eq:UniformKhasminski-B} into~\eqref{eq:TermsToBoundKhasminskii} and bounding $(1-\vartheta)^{-1} \leq e^{\vartheta + 2\vartheta^2}$ for $\vartheta \in (0,\frac{1}{2})$, we find that
\begin{equation}
\label{eq:A25equiv}
\begin{split}
\overline{P}_{y}[R_J \geq S_N] & \leq \exp\left(-\vartheta\left(1-\frac{\varepsilon}{100}\right)\frac{S_N}{t_N} + (\vartheta + 2\vartheta^2)J \right) + \exp\left(-\frac{\varepsilon}{100} \frac{S_N}{t_N} + \frac{J}{N^c} \right)
\end{split}
\end{equation}
By using~\eqref{eq:fA1Bound} of Lemma~\ref{lem:Morehittingtimebounds}, we see that for $N$ large enough,
\begin{equation}
\frac{\overline{E}_x[H_{A_1}]}{\overline{E}_\pi[H_{A_1}]} \stackrel{\eqref{eq:fA1Def}}{=} 1 - f_{A_1}(x) \leq 1 + \frac{1}{N^c} \leq \frac{1}{1-\varepsilon/100}, \qquad x \in U^N.
\end{equation}
It follows that, for large enough $N$,
\begin{equation}
\frac{S_N}{t_N} \geq \left(1-\frac{\varepsilon}{100}\right) \frac{S_N}{\overline{E}_\pi[H_{A_1}]}  \stackrel{\eqref{eq:InverseHittingEstimates}}{\geq} \left(1-\frac{\varepsilon}{50}\right)(1+\varepsilon)\varphi_N^2(x_0)\capa(A_1).
\end{equation}
One can then choose $\vartheta$  sufficiently small to ensure that for large enough $N$
\begin{equation}
\label{eq:A28equiv}
-\vartheta\left(1-\frac{\varepsilon}{100} \right) \frac{S_N}{t_N} + (\vartheta + 2\vartheta^2)J \leq -N^c
\end{equation}
(recalling the definition of $J$ in~\eqref{eq:JDef}). It also holds that, for large enough $N$
\begin{equation}
\label{eq:A29equiv}
-\frac{\varepsilon}{100} \frac{S_N}{t_N} + \frac{J}{N^c} \leq -N^{c'},
\end{equation}
and inserting~\eqref{eq:A28equiv} and~\eqref{eq:A29equiv} into~\eqref{eq:A25equiv} shows the claim.
\end{proof}

We now start with the construction of the first coupling. Roughly speaking, we will demonstrate that the field of occupation times corresponding to $J-1$ long excursions (counting from the second to the $J$-th long excursion) of the tilted walk in $A_2$ dominates (with overwhelming probability) a field of occupation times created by $J-1$ \textit{independent} long excursions, each one started from the quasi-stationary distribution $\sigma$. We let
\begin{equation}
L^j_x = \int_{R_j}^{V_j} \mathbbm{1}_{\{ X_t = x\}}\mathrm{d}t, \ j \geq 1, x \in A_2,
\end{equation}
stand for the occupation times obtained by considering the excursions $X_{[R_j,V_j)} \cap A_2$ of the trajectory of a random walk $(X_t)_{t \geq 0}$ in $A_2$. The statement of the following result is similar to~\cite[Proposition 5.2]{li2017lower}, itself based on an argument in~\cite[Lemma 4.2]{teixeira2011fragmentation}, however for our purposes a coupling of the local times is necessary, whereas in~\cite{li2017lower}, a coupling of the traces was sufficient.

\begin{prop}
\label{prop:CouplingIndependentStartedfromSigme}
For large $N$, and $y \in U^N$, there exists a probability space $(\mathcal{O}_1,\mathcal{F}_1,Q_1)$ on which we can define two random fields $(K_x)_{x \in A_2}$ and $(\widetilde{K}_x)_{x \in A_2}$ with values in $[0,\infty)^{A_2}$, such that $(K_x)_{x \in A_2}$ has the same law as $\left(\sum_{j = 2}^J L^j_x \right)_{x \in A_2}$ under $\overline{P}_y$, and $(\widetilde{K}_x)_{x \in A_2}$ has the same law as $\left(\sum_{j = 2}^J \widetilde{L}^j_x \right)_{x \in A_2}$, where for $j \geq 2$, the fields $(\widetilde{L}^j_x)_{x \in A_2}$ are distributed as i.i.d.~copies of
\begin{equation}
\left(\int_{R_1}^{V_1} \mathbbm{1}_{\{ Y_t = x\}}\mathrm{d}t \right)_{x \in A_2},
\end{equation}
with $Y$ having the same law as $X$ under $\overline{P}_\sigma$, such that
\begin{equation}
\label{eq:FirstCoupling}
Q_1\left[(K_x)_{x \in A_2} \neq (\widetilde{K}_x)_{x \in A_2} \right] \leq \exp\left(-c \log^2 N \right).
\end{equation}
\end{prop}

\begin{proof}
Let $z \in U^N \setminus A_2$. By~\eqref{eq:CouplingTV} and Proposition~\ref{prop:ClosenessToQSD},
there exists a coupling $\mathcal{Q}_z$ of two random variables $\Gamma$ and $\Sigma$ such that under $\mathcal{Q}_z$, $\Gamma$ and $\Sigma$ have the laws of $X_{t_\ast}$ under $\overline{P}_z[\,  \cdot\, | H_{A_2} > t_{\ast}]$ and $\sigma$ on $U^N \setminus A_2$, respectively, and 
\begin{equation}
\label{eq:TypicalCase}
\max_{z \in U^N \setminus A_2} \mathcal{Q}_z[\Gamma \neq \Sigma] \leq \exp\left(-c \log^2 N\right).
\end{equation}
We then introduce the index $M$ of the last step of a path in $A_2$ before time $V$, which is given by
\begin{equation}
M = \sup\bigg\{ n \geq 0 \, : \, \sum_{\ell = 0}^{n-1} \zeta_\ell \leq V \text{ and }Z_n \in A_2 \bigg\}
\end{equation}
(see above~\eqref{eq:DefOfCTSRWalk} for the definitions of $\zeta_\ell$ and $Z_n$). For the long excursions of the path, we define the subsequent indices of the corresponding ``last steps'' of the $j$-th long excursion in $A_2$ as 
\begin{equation}
M_j = M \circ \vartheta_{R_j} + n_j, \qquad \text{where }\sum_{\ell = 0}^{n_j -1} \zeta_\ell = R_j, \ j \geq 1.
\end{equation}
We choose $x_1^+ \in \partial A_2$ with the law $\overline{P}_y[Z_{M_1 +1} = \cdot]$. For $j \geq 1$, once $x_j^+$ is chosen, we choose points $x_{j+1}$ and $\widetilde{x}_{j+1}$ in $U^N \setminus A_2$ according to the joint law of $(\Gamma,\Sigma)$ under $\mathcal{Q}_{x_j^+}$, i.e.~according to $\mathcal{Q}_{x_j^+}[\Gamma = \cdot, \Sigma = \cdot ]$. 
The proof now differs from that of~\cite[Proposition 5.2]{li2017lower}. If $x_{j+1}$ and $\widetilde{x}_{j+1}$ coincide (which is typical due to~\eqref{eq:TypicalCase}), we choose $\mathcal{J}_{j+1} = \widetilde{\mathcal{J}}_{j+1} \in [0,\infty)^{A_2}$ as well as $x_{j+1}^+ = \widetilde{x}_{j+1}^+ \in \partial A_2$ according to $\overline{P}_{x_{j+1}}[(L^1_x)_{x \in A_2}  \in \cdot, Z_{M_1+1} = \cdot]$ (this is a probability measure on the space $[0,\infty)^{A_2} \times \partial A_2$, equipped with the $\sigma$-algebra $\mathcal{B}([0,\infty))^{\otimes {A_2}} \otimes 2^{\partial A_2}$). Otherwise, namely if  $x_{j+1}$ and $\widetilde{x}_{j+1}$ differ (which is the atypical case, and can be interpreted as the coupling failing at step $j+1$), we choose $(\mathcal{J}_{j+1},x_{j+1}^+)$ and $(\widetilde{\mathcal{J}}_{j+1},\widetilde{x}_{j+1}^+)$ independently from each other, and according to the laws $\overline{P}_{x_{i+1}}[(L^1_x)_{x \in A_2}  \in \cdot, Z_{M_1+1} = \cdot]$ and $\overline{P}_{\widetilde{x}_{i+1}}[(L^1_x)_{x \in A_2}  \in \cdot, Z_{M_1+1} = \cdot]$, respectively. Formally, we can write $Q_1$ as in the proof of~\cite[Lemma 4.2]{teixeira2011fragmentation} (setting $\boldsymbol{x} = (x_2,x_2^+,\ldots,x_J,x_J^+)$ and $\widetilde{\boldsymbol{x}} = (\widetilde{x}_2,\widetilde{x}_2^+,\ldots,\widetilde{x}_J,\widetilde{x}_J^+) \in ((U^N\setminus A_2) \times \partial A_2)^{J-1}$), defining further two $((U^N\setminus A_2) \times \partial A_2)^{J-1}$-valued random variables $\mathcal{X}$ and $\widetilde{\mathcal{X}}$ as
\begin{equation}
\begin{split}
& Q_1\left[ \mathcal{J}_j \in U_j,\widetilde{\mathcal{J}}_j \in \widetilde{U}_j, \text{for } 2 \leq j \leq J, \mathcal{X} = \boldsymbol{x},\widetilde{\mathcal{X}} = \widetilde{\boldsymbol{x}} \right] \\
& = \sum_{x_1^+ \in \partial A_2} \overline{P}_y[Z_{M_1+1} = x_1^+] \prod_{j = 1}^{J-1} \bigg( \mathcal{Q}_{x_{j}^+}[\Gamma = x_{j+1},\Sigma  = \widetilde{x}_{j+1}] \\
& \qquad \times \Big(\mathbbm{1}_{\{x_{j+1} = \widetilde{x}_{j+1} \}} \overline{P}_{x_{j+1}}\left[(L_x^1)_{x \in A_2} \in U_{j+1}, Z_{M_1+1} = x_{j+1}^+ \right] \mathbbm{1}_{\{x_{j+1}^+ = \widetilde{x}_{j+1}^+, \mathcal{J}_{j+1} = \widetilde{\mathcal{J}}_{j+1} \}} \\
& \qquad + \mathbbm{1}_{\{x_{j+1} \neq \widetilde{x}_{j+1} \}} \overline{P}_{x_{j+1}}\left[(L_x^1)_{x \in A_2} \in U_{j+1}, Z_{M_1+1} = x_{j+1}^+ \right] \\
& \qquad \times \overline{P}_{\widetilde{x}_{j+1}}\left[(L_x^1)_{x \in A_2} \in \widetilde{U}_{j+1}, Z_{M_1+1} = \widetilde{x}_{j+1}^+ \right] \Big) \bigg),
\end{split}
\end{equation}
for any $U_2 \times\widetilde{U}_2\times\ldots\times U_J\times\widetilde{U}_J \in (\mathcal{B}([0,\infty))^{\otimes A_2})^{\otimes 2(J-1)}$. The fields $(K_x)_{x \in A_2}$ and $(\widetilde{K}_x)_{x \in A_2}$ are then defined as 
\begin{equation}
K_x = \sum_{j = 2}^J (\mathcal{J}_j)_x, \qquad \widetilde{K}_x = \sum_{j = 2}^J (\widetilde{\mathcal{J}}_j)_x, \qquad x \in A_2.
\end{equation}
One can then check inductively (similarly as in the proof of~\cite[Lemma 4.2]{teixeira2011fragmentation}) that $Q_1$ defines a coupling as required such that $(K_x)_{x \in A_2}$ and $(\widetilde{K}_x)_{x \in A_2}$ have the required laws, and that the probability that the coupling does not succeed can be bounded above as
\begin{equation}
\begin{split}
Q_1\left[(K_x)_{x \in A_2} \neq (\widetilde{K}_x)_{x \in A_2} \right] & \leq (J-1) \max_{x \in U^N \setminus A_2}\mathcal{Q}_x[\Gamma \neq \Sigma] \\
& \hspace{-0.85cm} \stackrel{\eqref{eq:capacityControlsBox},\eqref{eq:Phi_N_TechnicalProp},\eqref{eq:JDef}}{\leq} cN^{d-2} e^{-c\log^2N} \leq e^{-c'\log^2N},
\end{split}
\end{equation}
giving us the claim.
\end{proof}
As a next crucial step, to continue with the series of couplings, we aim at replacing the trajectories of (independent) excursions of tilted walks started from $\sigma$ by (independent) excursions of \textit{simple} random walk trajectories started from $\widetilde{e}_{A_1}$. It will be convenient to only consider the part of any given excursion before the first exit time of $A_2$ subsequent to the entrance into $A_1$. To this end, we define 
\begin{equation}
\label{eq:W-definition}
W = H_{A_1} + T_{A_2} \circ \vartheta_{H_{A_1}},
\end{equation}
namely the first exit time of the walk from the box $A_2$ after entering $A_1$, as well as
\begin{equation}
\label{eq:kappa_1Def}
\kappa_1 = \text{ the law of the stopped process $X_{(H_{A_1} + \, \cdot \,) \wedge W}$ under $\overline{P}_\sigma$}
\end{equation}
on the space $\Gamma(U^N)$, that is the aforementioned law of an ``(short) excursion of the tilted walk'' started from $\sigma$ (and recorded \textit{after} the first entrance time into $A_1$), and 
\begin{equation}
\label{eq:kappa_2Def}
\kappa_2 = \text{ the law of the stopped process $X_{\, \cdot \, \wedge T_{A_2}}$ under $P_{\widetilde{e}_{A_1}}$},
\end{equation}
on the space $\Gamma(U^N)$, which is the law of a ``(short) excursion of the simple random walk'' started from $\widetilde{e}_{A_1}$ (recall~\eqref{eq:normalized_eqmeasureDef}). Instead of directly comparing $\kappa_1$ and $\kappa_2$, we will devise below a domination in terms a modified version of $\kappa_2$ that assigns zero mass to excursions leaving $A_2$ only after an excessive time. For $\alpha \in (r_2,1)$, we define the measures on $\Gamma(U^N)$ 
\begin{equation}
\label{eq:DifferentKappaDef}
\begin{split}
\widetilde{\kappa}_j & = \kappa_j[ \, \cdot \, |  \mathcal{E}_\alpha], \qquad j \in \{1,2\}, \\
\overline{\kappa}_2 & = \kappa_2[\, \cdot \, \cap \mathcal{E}_\alpha]
\end{split}
\end{equation}
(recall the definition of the event $\mathcal{E}_\alpha$ in~\eqref{eq:NotTooLargeExit}). Note that $\overline{\kappa}_2$ is not a probability measure. The next proposition shows that $\kappa_1$ almost dominates $\overline{\kappa}_2$, and will make crucial use of Lemma~\ref{lem:HKLemma}.
\begin{prop}
\label{prop:Radon-Nikodym}
For large enough $N$, one has for any measurable $S \subseteq \Gamma(U^N)$ that
\begin{equation}
 \kappa_1[S] \geq \left( 1 - \frac{1}{N^c}  \right) \overline{\kappa}_2[S].
\end{equation}
\end{prop}
\begin{proof}
We consider the following auxiliary  measures on the space of trajectories $\Gamma(U^N)$
\begin{equation}
\begin{split}
\gamma & = \text{ the law of the stopped process $X_{\, \cdot \, \wedge T_{A_2}}$ under $\overline{P}_{\widetilde{e}_{A_1}}$}, \\
\widetilde{\gamma} & = \gamma[ \, \cdot \, |  \mathcal{E}_\alpha].
\end{split}
\end{equation}
Importantly, the trajectories under $\gamma$ have the same starting distribution $\widetilde{e}_{A_1}$ as under $\kappa_2$ but a different generator corresponding to the confined walk.
 By Lemma~\ref{lem:HKLemma}, we have for any measurable $S \subseteq \Gamma(U^N)$ depending on $X_{\cdot \wedge T_{A_2}}$,
\begin{equation}
\begin{split}
\label{eq:GammaBoundedBelow}
\gamma[S] & \geq \widetilde{\gamma}[S] \cdot \gamma[\mathcal{E}_\alpha] = \widetilde{\gamma}[S] \left(1 - \sum_{x \in \partial_{\mathrm{int}}A_1} \widetilde{e}_{A_1}(x)\overline{P}_x[\mathcal{E}_\alpha^c] \right)  \stackrel{\eqref{eq:ExitTimeDeviation}}{\geq} \widetilde{\gamma}[S]\left(1-C_4 \exp\Big(-C_5 N^{c_6} \Big) \right).
\end{split}
\end{equation}
We will now compare the measures $\widetilde{\gamma}$ and $\widetilde{\kappa}_2$. To that end, note that
\begin{equation}
\label{eq:BoundGammaByKappa_2}
\begin{split}
\widetilde{\gamma}[S] & = \sum_{x \in A_1 } \widetilde{e}_{A_1}(x) \overline{P}_x[S | \mathcal{E}_\alpha] \stackrel{\eqref{eq:RNAbsBound}}{\geq} \sum_{x \in A_1}  \widetilde{e}_{A_1}(x)  \left(1 - \frac{c_7(\alpha)}{N^{1 - r_2}} \right) P_x[S | \mathcal{E}_\alpha]\\
& = \left(1 - \frac{c_7(\alpha)}{N^{1 - r_2}} \right)\widetilde{\kappa}_2[S].
\end{split} 
\end{equation}
We can directly compare the measures $\gamma$ and $\kappa_1$, as we now explain. Indeed, by the strong Markov property applied at time $H_{A_1}$, we have that
\begin{equation}
\begin{split}
\frac{\mathrm{d} \kappa_1}{\mathrm{d}\gamma}  = \Phi(X_0),\qquad\text{where }\Phi(x) & = \frac{\overline{P}_\sigma[X_{H_{A_1}} = x]}{\widetilde{e}_{A_1}(x)} \mathbbm{1}_{\partial_{\mathrm{int}}A_1}(x).
\end{split} 
\end{equation}
By Proposition~\ref{prop:AnalogProp4.7}, we find that, for large $N$,
\begin{equation}
\frac{\mathrm{d} (\kappa_1 - \gamma)}{\mathrm{d}\gamma} = \Phi(X_0) - 1 \geq -\frac{1}{N^c}.
\end{equation}
We therefore have that for $S \subseteq \Gamma(U^N)$ measurable depending on $X_{\cdot \wedge T_{A_2}}$, that for large $N$,
\begin{equation}
\begin{split}
\kappa_1[S] & \geq \left(1 - \frac{1}{N^c}\right) \gamma[S]  \stackrel{\eqref{eq:GammaBoundedBelow}}{\geq}\left(1 - \frac{1}{N^c}\right) \left(1-C_4 \exp\Big(-C_5 N^{c_6} \Big) \right)\widetilde{\gamma}[S] \\
& \stackrel{\eqref{eq:BoundGammaByKappa_2}}{\geq} \left(1 - \frac{1}{N^{c'}}\right) \widetilde{\kappa}_2[S],
\end{split}
\end{equation}
for some appropriate $c' > 0$. Finally, note that
$$
\widetilde{\kappa}_2[S] = \frac{\kappa_2[S \cap \mathcal{E}_\alpha]}{\kappa_2[\mathcal{E}_\alpha]} \stackrel{\kappa_2[\mathcal{E}_\alpha] \leq 1}{\geq} \kappa_2[S \cap \mathcal{E}_\alpha] =  \overline{{\kappa}}_2[S],
$$
and the claim follows.
\end{proof}
We will now use the results of Propositions~\ref{prop:CouplingIndependentStartedfromSigme}
and~\ref{prop:Radon-Nikodym} to produce  further couplings to ultimately (in Theorem~\ref{thm:CrucialCoupling}) compare the occupation-time field in $A_2$ of the confined walk with the one produced by a Poisson point process of simple random walk trajectories started from $\widetilde{e}_{A_1}$. To that end, consider an auxiliary probability space $(\widetilde{\mathcal{O}},\widetilde{\mathcal{F}},\widetilde{Q})$, on which two independent Poisson point processes $\eta_1, \eta_2$ with values in $\Gamma(U^N)$ are defined such that 
\begin{equation}
\label{eq:IntensityMeasuresOfPPP}
\begin{split}
&\text{the intensity measure of $\eta_1$ is given by $(1 + \tfrac{\varepsilon}{3}) \varphi_N^2(x_0)\capa(A_1)\kappa_1$}, \\
&\text{the intensity measure of $\eta_2$ is given by $(1 + \tfrac{\varepsilon}{4}) \varphi_N^2(x_0)\capa(A_1)\kappa_2$}
\end{split}
\end{equation}
  (see~\eqref{eq:kappa_1Def} and~\eqref{eq:kappa_2Def} for the definitions of $\kappa_1$ and $\kappa_2$, respectively). We define the random fields $(L^{\eta_j}_x)_{x  \in A_2}$, $j \in \{1,2\}$, as
\begin{equation}
\label{eq:LetaDef}
L^{\eta_j}_x = \text{ the total time spent in $x$ by trajectories in $\eta_j$}, \qquad j \in \{1,2\}.
\end{equation}
In the next proposition, we provide a Poissonization of ``shortened versions'' of the $J-1$ independent trajectories obtained from the previous Proposition~\ref{prop:CouplingIndependentStartedfromSigme}.
\begin{prop}
\label{prop:PoissonizationCoupling}
For large $N$, there exists a probability space $(\mathcal{O}_2,\mathcal{F}_2,Q_2)$ on which we can define two random fields $(K'_x)_{x \in A_2}$ and $(K''_x)_{x \in A_2}$, with values in $[0,\infty)^{A_2}$, such that $(K_x')_{x \in A_2}$ has the same law as $\left(\sum_{j = 2}^J \overline{L}^j_x \right)_{x \in A_2}$, where for $j \geq 2$, the fields $(\overline{L}^j_x)_{x \in A_2}$ are distributed as i.i.d.~copies of
\begin{equation}
\left(\int_{H_{A_1}}^{W} \mathbbm{1}_{\{ Y_t = x\}}\mathrm{d}t \right)_{x \in A_2},
\end{equation}
with $Y$ having the same law as $X$ under $\overline{P}_\sigma$, and $(K_x'')_{x \in A_2}$ has the same law as $(L^{\eta_1}_x)_{x \in A_2}$ in~\eqref{eq:LetaDef} under $\widetilde{Q}$, and
\begin{equation}
\label{eq:SecondCoupling}
Q_2\Big[(K'_x)_{x \in A_2} \geq (K''_x)_{x \in A_2} \Big] \geq 1 - e^{-N^c}.
\end{equation}
\end{prop}
\begin{proof}
Let $\xi \sim Poisson((1+\frac{\varepsilon}{3})\varphi_N^2(x_0) \capa(A_1))$ and let $(\widehat{X}^{(j)} )_{j \in \mathbb{N}}$ be a sequence of i.i.d.~confined walks with the same law as $X$ under $\overline{P}_\sigma$, independent of $\xi$. We set 
\begin{equation}
\widehat{L}^j_x = \text{ the total time spent in $x$ by }\widehat{X}_{[H_{A_1},W)}^{(j)}, j \geq 2.
\end{equation}
We then define
\begin{equation}
K'_x = \sum_{j = 2}^{J} \widehat{L}^j_x, \qquad K''_x = \sum_{j = 2}^{\xi + 1} \widehat{L}^j_x, \qquad x \in A_2,
\end{equation}
which clearly have the respective required laws. By a standard estimate on concentration of Poisson random variables, we see that
\begin{equation}
Q_2\Big[(K'_x)_{x \in A_2} \geq (K''_x)_{x \in A_2} \Big] \geq Q_2[J \geq \xi + 1] \geq 1 - e^{-N^c},
\end{equation}
and the claim follows.
\end{proof}

In the following theorem, we will develop the pivotal last coupling needed, which compares the occupation time fields $(L^{\eta_1}_x)_{x \in A_2}$ and $(L^{\eta_2}_x)_{x \in A_2}$, where we make use of Proposition~\ref{prop:Radon-Nikodym}.

\begin{theorem}
\label{thm:CrucialCoupling}
For large $N$, there exists a probability space $(\mathcal{O}_3,\mathcal{F}_3,Q_3)$ on which we can define two random fields $(M_x)_{x \in A_2}$ and $(\widetilde{M}_x)_{x \in A_2}$ with values in $[0,\infty)^{A_2}$ such that $(M_x)_{x \in A_2}$ has the same law as the field $(L^{\eta_1}_x)_{x \in A_2}$ under $\widetilde{Q}$ and $(\widetilde{M}_x)_{x \in A_2}$ has the same law as $(L^{\eta_2}_x)_{x \in A_2}$ under $\widetilde{Q}$, such that
\begin{equation}
\label{eq:ThirdCoupling}
Q_3\Big[(M_x)_{x \in A_2} \geq (\widetilde{M}_x)_{x \in A_2} \Big] \geq 1 - e^{-N^c}.
\end{equation}
\end{theorem}
\begin{proof}
We consider for $N$ large enough a probability space $(\mathcal{O}_3,\mathcal{F}_3,Q_3)$ on which we can define three independent Poisson point processes $\theta$, $\beta_1$, and $\beta_2$ on $\Gamma(U^N)$ such that
\begin{equation}
\begin{split}
&\text{the intensity measure of $\theta$ is given by $(1 + \tfrac{\varepsilon}{4}) \varphi_N^2(x_0)\capa(A_1)\overline{\kappa}_2$}, \\
&\text{the intensity measure of $\beta_1$ is given by $ \varphi_N^2(x_0)\capa(A_1)((1 + \tfrac{\varepsilon}{3})\kappa_1 -  (1 + \tfrac{\varepsilon}{4})\overline{\kappa}_2)$}, \\
&\text{the intensity measure of $\beta_2$ is given by $(1 + \tfrac{\varepsilon}{4}) \varphi_N^2(x_0)\capa(A_1)(\kappa_2 - \overline{\kappa}_2)$}.
\end{split}
\end{equation}
Here, for $N$ large enough, Proposition~\ref{prop:Radon-Nikodym} guarantees that the signed measure $(1 + \tfrac{\varepsilon}{3})\kappa_1 -  (1 + \tfrac{\varepsilon}{4})\overline{\kappa}_2$ and hence also the intensity measure of $\beta_1$ is non-negative (the measure $\kappa_2 - \overline{\kappa}_2$ is non-negative by definition, see~\eqref{eq:DifferentKappaDef}). Denoting for $x \in A_2$,
\begin{equation}
\begin{split}
L^{\theta}_x & = \text{ the total time spent in $x$ by trajectories in $\theta$}, \\
L^{\beta_1}_x & = \text{ the total time spent in $x$ by trajectories in $\beta_1$}, \\
L^{\beta_2}_x & = \text{ the total time spent in $x$ by trajectories in $\beta_2$},
\end{split}
\end{equation}
we observe that the field $(M_x)_{x \in A_2}$ with $M_x \stackrel{\mathrm{def}}{=}L_{x}^\theta + L_x^{\beta_1}$ under $Q_3$ has the same law as $(L^{\eta_1}_x)_{x \in A_2}$ under $\widetilde{Q}$, whereas the field $(\widetilde{M}_x)_{x \in A_2}$ with $\widetilde{M}_x \stackrel{\mathrm{def}}{=} L_{x}^\theta + L_x^{\beta_2}$ under $Q_3$ has the same law as $(L^{\eta_2}_x)_{x \in A_2}$ under $\widetilde{Q}$, as required. \smallskip

Furthermore, note that we have (analogously to the calculation in~\eqref{eq:GammaBoundedBelow}), since $\alpha \in (r_2,1)$, that
\begin{equation}
\label{eq:ConvergenceTo0Ec}
 P_{\widetilde{e}_{A_1}}[\mathcal{E}_\alpha^c] \stackrel{\eqref{eq:ExitTimeDeviation}}{\leq}   C_4' \exp\Big(- C_5' N^{c_6'} \Big).
\end{equation}
Therefore, we obtain
\begin{equation}
\begin{split}
Q_3\Big[(M_x)_{x \in A_2} \geq (\widetilde{M}_x)_{x \in A_2} \Big]
&  \geq  Q_3[L_x^{\beta_2} = 0 \text{ for }x\in A_1] \geq Q_3[\beta_2(\Gamma(U^N)) = 0] \\
& = e^{- ( 1 + \varepsilon/4)\varphi_N^2(x_0)\capa(A_1)(\kappa_2[\Gamma(U^N)] - \overline{\kappa}_2[\Gamma(U^N)])} \\
& = e^{- ( 1 + \varepsilon/4)\varphi_N^2(x_0)\capa(A_1)P_{\widetilde{e}_{A_1}}[\mathcal{E}_\alpha^c] } \geq 1 - e^{-N^c}, 
\end{split}
\end{equation}
using also~\eqref{eq:capacityControlsBox}, the fact that $\varphi_N(x_0)$ is bounded above, see~\eqref{eq:Phi_N_TechnicalProp}, and $e^{-x} \geq 1-x$ for $x \geq 0$.
\end{proof}

\begin{remark}
\label{rem:LocalCoupling}
The results in this section can be interpreted as follows: Given a function $\varphi$ fulfilling the properties i) and ii) in Proposition~\ref{prop:quasi-minimizer}, in any mesoscopic box $A_1 = B(x_0,M)$, where $x_0 \in D^\delta_N$ and $M = \lfloor N^{r_1} \rfloor$ (with $r_1$ as in~\eqref{eq:r_jDef}) the occupation time field of the tilted random walk dominates (up to an error term of size at most $\exp(-c\log^2 N)$) the occupation times generated by a Poisson point process of trajectories with intensity measure roughly $\varphi^2(x_0/N)\capa(A_1)\widetilde{e}_{B(x_0,M)}$, and which are stopped at the first exit time of $A_2$. One can show (see Proposition~\ref{prop:concentration} below, which adapts~\cite[Proposition 2.2, Lemma 2.3]{sznitman2019bulk}, see also~\cite[Proposition 5.4, Section 9]{belius2013gumbel}) that these dominate the occupation times of continuous-time random interlacements at a locally constant level, roughly equal to $\varphi^2(x_0/N)$. Importantly, this coupling is uniform over all mesoscopic boxes centered at such $x_0$. 
\end{remark}

\section{Lower bound}
\label{sec:lowerbound}

In this section, we state and prove in Theorem~\ref{thm:MainThm} below our main result concerning the asymptotic lower bound on the probability of the event $\mathcal{A}^\nu_N(F)$, see~\eqref{eq:ExcessTypeEvent}. As pivotal steps, we will utilize the coupling results in Section~\ref{sec:coupling}, as well as the calculation of the relative entropy in Proposition~\ref{prop:RelEntropy}. We also need a (Poisson) concentration result (see Proposition~\ref{prop:concentration}), showing that for a carefully chosen Poisson point process of simple random walk trajectories on a mesoscopic scale, a downward deviation from its expectation is unlikely. Finally, in Subsection~\ref{subsec:Applications}, we give some applications of our main result.
\subsection{Proof of the lower bound}
\label{subsec:Maintheorem}
We recall the definition of the function $\vartheta(\cdot)$ from~\eqref{eq:def theta} as well as that of the energy $\cI_D(\cdot)$ of the minimizer from~\eqref{eq:I_D(nu)Def}, and now state the main result of the present article.

\begin{theorem}
\label{thm:MainThm}
Suppose that the local function $F$ fulfills Assumption~\ref{eq:RegularityCondF} and fix a set $D$ as above~\eqref{eq:DiscreteBlowup}. Then, for every $y \in \bbZ^d$ and every $\nu \in (0,\vartheta_\infty)$, where $\vartheta_\infty = \lim_{u \rightarrow \infty} \vartheta(u)$, one has
\begin{equation}
\label{eq:MainTheoremStatementSec6}
\begin{split}
\liminf_{N \rightarrow \infty} & \frac{1}{N^{d-2}} \log P_y[\cA^\nu_N(F)] \geq   -\cI_D(\nu) \\
&\left( = - \inf\bigg\{\frac{1}{2d}\int_{\bbR^d} |\nabla \phi|^2\,\De x\, : \,\phi\in C_0^\infty(\bbR^d)\,,\fint_D \vartheta(\phi^2)\,\De x > \nu\bigg\}\right).
\end{split}
\end{equation}
\end{theorem}

\begin{remark}
\label{eq:MatchingUpperBound}
If in addition to Assumption~\ref{eq:RegularityCondF}, the function $F$ is bounded (which implies in particular that $\vartheta_\infty < \infty$), one has a matching asymptotic upper bound to~\eqref{eq:MainTheoremStatementSec6} from~\cite[Corollary 5.11]{sznitman2019bulk}, which is in fact obtained by a similar result for random interlacements in the ``singular limit $u \rightarrow 0$''. Thus, we obtain that for bounded $F$ fulfilling Assumption~\ref{eq:RegularityCondF}, for every $y \in \bbZ^d$ and every $\nu \in (0,\vartheta_\infty)$, one has
\begin{equation}
\label{eq:CoincidingBounds}
\begin{split}
\lim_{N \rightarrow \infty} & \frac{1}{N^{d-2}} \log P_y[\cA^\nu_N(F)] =  -\cI_D(\nu) \\
&\left( = - \inf\bigg\{\frac{1}{2d}\int_{\bbR^d} |\nabla \phi|^2\,\De x\, : \,\phi\in C_0^\infty(\bbR^d)\,,\fint_D \vartheta(\phi^2)\,\De x > \nu\bigg\}\right).
\end{split}
\end{equation}
\end{remark}

We will prove Theorem~\ref*{thm:MainThm} by a change of probability method, this involves computing the relative entropy of the tilted walk $\widetilde{P}_{y,N}$ (recall~\eqref{eq:TiltedWalk}, with respect to the simple random walk $P_y$. In the next proposition we provide an upper bound on this relative entropy.
\begin{prop} 
\label{prop:RelEntropy} One has that for all $y \in \bbZ^d$ that (with $\varepsilon > 0$ as chosen above~\eqref{eq:time_horizon})
    \begin{equation}
        \limsup_{N\to \infty} \frac{1}{N^{d-2}} \cH(\widetilde{P}_{y,N}|P_y) \leq (1 + \varepsilon) \frac{1}{d} \mathcal{E}_{\bbR^d}(\varphi,\varphi)
    \end{equation}
    (note that the right-hand side is well-defined since $\varphi \in C^\infty_0(\bbR^d)$ by Proposition~\ref{prop:quasi-minimizer}, i)).
\end{prop}
\begin{proof}
Note that for $N$ large enough, we can guarantee that $y \in U^N ( = (NB_\cR) \cap \bbZ^d$), which we assume throughout the remainder of the proof. By the definition~\eqref{eq:RelEntropy}, we obtain for any $y \in U^N$
\begin{equation}
\label{eq:RelEntropyCalc}
\begin{split}
\cH(\widetilde{P}_{y,N}|P_y) & \stackrel{\eqref{eq:MeasureUpToT},\eqref{eq:TiltedWalk}}{=} \widetilde{E}_{y,N}\left[\log  M_{S_N}\right] \\
& \stackrel{\eqref{eq:MartingaleDef},\eqref{eq:ConfinedTiltedCoincide}}{=}  \overline{E}_y\left[\int_0^{t_\star} v_\varphi(X_s) \mathrm{d}s   \right] + \overline{E}_y\left[\int_{t_\star}^{S_N} v_\varphi(X_s) \right] \\
& \qquad \qquad  +  \overline{E}_y\left[ \log f(X_{S_N}) - \log f(X_{0}) \right],
\end{split}
\end{equation}
where $\widetilde{E}_{y,N}$ and $\overline{E}_y$ denote the expectations under $\widetilde{P}_{y,N}$ and $\overline{P}_y$, respectively, and we recall that $v_\varphi$ and $f$ are defined in~\eqref{eq:v_phiDef} and~\eqref{eq:varphi_N_Def}. We bound the three summands in~\eqref{eq:RelEntropyCalc} separately. Importantly, we stress that the main contribution comes from the second summand and will crucially use that $S_N = (1 + \varepsilon) \| \varphi_N\|_{\ell^2(\bbZ^d)}^2$. By~\eqref{eq:vUpperBound}, we find that for the first summand, we have
\begin{equation}
\label{eq:TermIBound}
\overline{E}_y\left[\int_0^{t_\star} v_\varphi(X_s) \mathrm{d}s   \right] \leq t_\star \sup_{x \in U^N} v_\varphi(x) \leq c\log^2 N,
\end{equation}
and the third summand can be bounded by
\begin{equation}
\label{eq:TermIIIBound}
\begin{split}
\overline{E}_y\left[ \log f(X_{S_N}) \right] - \overline{E}_y\left[ \log f(X_{0}) \right] & \leq \log \max_{z \in U^N} f(z) - \log \min_{z \in U^N} f(z) \\
& \leq c \log N.
\end{split}
\end{equation}
We now replace the second summand in~\eqref{eq:RelEntropyCalc} by term bringing into play the stationary measure $\pi$. To this end, we use Lemma~\ref{lem:TVtoStationary} to write
\begin{equation}
\label{eq:TermIIBound}
\begin{split}
& \bigg\vert \overline{E}_y\bigg[\int_{t_\star}^{S_N} v_\varphi(X_s) \mathrm{d}s - (S_N - t_\star) \sum_{z \in U^N} v_\varphi(z)\pi(z) \bigg] \bigg\vert \\
& \leq |U^N| S_N \cdot \sup_{t \in [t_\star,S_N], x,z \in U^N} \bigg| \overline{P}_x[X_t = z] -  \pi(z)\bigg| \cdot \max_{z \in U^N} |v_\varphi(z)| \\
& \leq CN^{2d-2}e^{-c'\log^2N},
\end{split}
\end{equation}
using also~\eqref{eq:vUpperBound} and the fact that $S_N \leq CN^d$, see~\eqref{eq:S_N_volume_order}. 
by Lemma~\ref{lem:TechnicalProperties}, ii). We now calculate the contribution coming from the stationary measure. To that end, note that due to~\eqref{eq:time_horizon}, we have
\begin{equation}
\label{eq:MainContributionBoundEntropy}
\begin{split}
\sum_{z \in U^N} v_\varphi(z)\pi(z) & \stackrel{\eqref{eq:piDef},\eqref{eq:v_phiDef}}{=} -\frac{1 + \varepsilon}{S_N} \sum_{z \in \bbZ^d} \varphi_N(z) \Delta_{\bbZ^d} \varphi_N(z) \\
& \stackrel{\eqref{eq:DiscreteGaussGreen}}{=} \frac{1 + \varepsilon}{S_N} \mathcal{E}_{\bbZ^d}(\varphi_N,\varphi_N).
\end{split} 
\end{equation}
Upon combining~\eqref{eq:TermIBound},~\eqref{eq:TermIIIBound} and~\eqref{eq:TermIIBound}, we therefore see that
\begin{equation}
\label{eq:RelativeEntropyAlmostDone}
\begin{split}
\cH(\widetilde{P}_{y,N}|P_y) & \leq c\log^2N + e^{-c'\log^2N} + \sum_{z \in U^N} v_\varphi(z) \pi(z) + c\log N \\
& \stackrel{\eqref{eq:MainContributionBoundEntropy}}{\leq} (1 + \varepsilon) \mathcal{E}_{\bbZ^d}(\varphi_N,\varphi_N) + C \log^2 N,
\end{split}
\end{equation} 
for large enough $N$. By a standard Riemann sum argument (recall that $\varphi$ is smooth and has compact support, see Proposition~\ref{prop:quasi-minimizer}, i)), we find that
\begin{equation}
\begin{split}
\limsup_{N \rightarrow \infty} \frac{1}{N^{d-2}} \mathcal{E}_{\bbZ^d}(\varphi_N,\varphi_N) \leq \frac{1}{d} \mathcal{E}_{\bbR^d}(\varphi,\varphi).
\end{split}
\end{equation}
The claim then immediately follows from~\eqref{eq:RelativeEntropyAlmostDone}.
\end{proof}

Before starting the proof of Theorem~\ref{thm:MainThm}, we state the following proposition which is a simpler case of~\cite[Proposition 2.2]{sznitman2019bulk}. We give the proof in the Appendix for the reader's convenience.
\begin{prop}\label{prop:concentration}
Let $B = B(z,N)$ and $U = B(z,N^\gamma)$, for $z \in \bbZ^d$ and $\gamma>1$. Let $\eta$ be a Poisson point process defined on some probability space $(\overline{\mathcal{O}},\overline{\mathcal{F}},\overline{Q})$ with values in $\Gamma(\bbZ^d)$ and with intensity measure given by $a (1 + \Delta) \capa(B)\kappa$, where $a,\Delta>0$, and $\kappa$ denotes the law of the stopped process $X_{\cdot \wedge T_U}$ under $P_{\widetilde{e}_B}$. Assume that $F$ satisfies Assumption \ref{eq:RegularityCondF}. Then for all $N$ large enough, 
    \begin{equation}
    \begin{split}
        \overline{Q}\Big[\frac{1}{|B|}\sum_{x\in B} & F((L^\eta_{x+z})_{z\in B(0,\mathfrak{r})})  \leq \vartheta(a) \Big]\leq C\exp(- N^{c} (a\wedge 1/a)), \text{ where} \\
L^{\eta}_x & = \textnormal{ the total time spent in $x$ by trajectories in $\eta$}.
        \end{split}
    \end{equation}
\end{prop}

\begin{proof}[Proof of Theorem~\ref{thm:MainThm}] 
Let us fix $\delta > 0$, $\cR > 0$ as in~\eqref{eq:FixedChoiceR} and consider $\varphi\in C_0^\infty(\bbR^d)$ as constructed in Proposition~\ref{prop:quasi-minimizer}. We recall the tilted measure $\widetilde{P}_{y,N}$ for $N$ large enough (such that $y \in U^N$), defined in~\eqref{eq:TiltedWalk} where we choose $f$ as in~\eqref{eq:varphi_N_Def}. An application of the entropy lower bound~\eqref{eq:EntropyIneq}, yields
\begin{equation}
\label{eq:Penultimate_LowerBound}
   \log P_y[\cA^\nu_N(F)] \geq \log \widetilde{P}_{y,N}[\cA^\nu_N(F)] -  \frac{1}{\widetilde{P}_{y,N}[\cA^\nu_N(F)]} \Big\{\cH(\widetilde{P}_{y,N}|P_y)+\frac{1}{\rm{e}}\Big\}.
\end{equation}
We now claim that
\begin{equation}
\label{eq:TiltedProbTo1}
    \lim_{N\to \infty} \widetilde{P}_{y,N}[\cA^\nu_N(F)] = 1.
\end{equation}
We admit~\eqref{eq:TiltedProbTo1} for the time being and explain how the proof of~\eqref{eq:MainTheoremStatementSec6} is then concluded. By Proposition~\ref{prop:RelEntropy}, we have that 
\begin{equation}
\label{eq:RelEntropyBoundSec6}
\limsup_{N \rightarrow \infty} \frac{1}{N^{d-2}}\cH(\widetilde{P}_{y,N}|P_y) \leq (1+\varepsilon) \frac{1}{d} \cE_{\bbR^d}(\varphi,\varphi) \stackrel{\text{Proposition~\ref{prop:quasi-minimizer} iii)},~\eqref{eq:last}}{ \leq} (1+\varepsilon)\cI_D(\nu(1 + 2\delta)).
\end{equation}
Upon dividing~\eqref{eq:Penultimate_LowerBound} by $N^{d-2}$ and   applying~\eqref{eq:RelEntropyBoundSec6} and~\eqref{eq:TiltedProbTo1}, we readily find that
\begin{equation}
   \liminf_{N\to\infty} \frac{1}{N^{d-2}} \log P_y[\cA^\nu_N(F)] \geq -(1+\varepsilon)\cI_D(\nu(1 + 2\delta)).
\end{equation}
Theorem \ref{thm:MainThm} now follows by letting $\delta \downarrow 0$, $\varepsilon \downarrow 0$, and by using the continuity of the rate function $\cI_D(\cdot)$ established in Proposition~\ref{prop:continuityofI}. \smallskip

We are left with proving~\eqref{eq:TiltedProbTo1}. To that end, we introduce an intermediate scale $M = \lfloor N^{r_1}\rfloor$ with $r_1\in (0,\tfrac{1}{4})$ as in~\eqref{eq:r_jDef} and consider $M$-boxes 
\begin{equation}
B_x \stackrel{\mathrm{def}}{=} B(x,M) (= x + [-M,M]^d\cap \bbZ^d), \qquad \text{with centers }x \in (2M+1)\bbZ^d.
\end{equation} 
For an $M$-box $B$, we denote by $x_B$ the unique point in $(2M+1) \bbZ^d$ such that $B = B_{x_B}$. Importantly, by taking $N$ large enough, a box $B$ intersecting $D_N$ can be identified with the box $A_1$ with center $x_0 = x_{B}$ (see~\eqref{eq:AjBoxesDef}), where we can assume that $x_B \in D^\delta_N$ as well as $A_6 \subseteq D^\delta_N$ (see~\eqref{eq:zetaAndx_0}). Since $\varphi$ is smooth and compactly supported (see Proposition~\ref{prop:quasi-minimizer}), in view of the Lipschitz continuity of $\vartheta$ (see Lemma~\ref{lem:PropertiesTheta}),~\eqref{eq:last}, of the fact that $r_1 < 1$, and recalling the notation $\varphi_N(\cdot) = \varphi(\cdot/N)$, it follows from a Riemann sum approximation argument that for any $N$ large enough 
\begin{equation}
    \nu (1 + \delta) \leq \fint_D \vartheta(\varphi^2)\,\De x \leq  (1+\delta)^{1/2}\sum_{B: B\subseteq D_N} \frac{|B|}{|D_N|} \vartheta(\varphi_N(x_B)^2),
\end{equation}
where the sum is over all boxes $M$-boxes $B$ contained in $D_N$. Therefore, for all $N$ sufficiently large
\begin{equation}
    \begin{aligned}
        \Big\{\frac{1}{|D_N|}&\sum_{x\in D_N} F\Big((L_{x+z})_{z \in B(0,\mathfrak{r})}\Big) \leq \nu\Big\}\\
        &\subseteq \Big\{\frac{1}{|D_N|}\sum_{x\in D_N} F\Big((L_{x+z})_{z \in B(0,\mathfrak{r})}\Big) \leq (1+\delta)^{-1}\fint_D \vartheta(\varphi^2)\,\De x \Big\}\\
        &\subseteq  \Big\{\sum_{x\in D_N} F\Big((L_{x+z})_{z \in B(0,\mathfrak{r})}\Big) \leq (1+\delta)^{-1/2}\sum_{B: B\subseteq D_N} |B|\vartheta(\varphi_N(x_B)^2) \Big\}\\
        &\subseteq\bigcup_{B: B\subseteq D_N} \Big\{\frac{1}{|B|}\sum_{x\in B} F\Big((L_{x+z})_{z \in B(0,\mathfrak{r})}\Big) \leq \frac{\vartheta(\varphi_N(x_B)^2)}{(1+\delta)^{1/2}} \Big\}.
    \end{aligned}
\end{equation}
As a consequence, with a union bound, we obtain the following estimate for $N$ large enough
\begin{equation}\label{eq:unionbound}
    \begin{aligned}
        \widetilde{P}_{y,N} &[\cA^\nu_N(F)^c]  \\ &\leq C N^{d(1-r_1)} \sup_{B: B\subseteq D_N} \widetilde{P}_{y,N}\bigg[\frac{1}{|B|}\sum_{x\in B} F\Big((L_{x+z})_{z \in B(0,\mathfrak{r})}\Big) \leq \frac{\vartheta(\varphi_N(x_B)^2)}{(1+\delta)^{1/2}}  \bigg].
    \end{aligned}
\end{equation}
We will now estimate the probabilities on the right hand side of~\eqref{eq:unionbound} uniformly on the $M$-boxes $B$ such that $B\subseteq D_N$. To that end we fix an arbitrary such box $A_1 = B$ centered at $x_0 = x_B$. We then choose $\eta \in (0, \cR/100)$ and $\widetilde{\delta}$ sufficiently small to guarantee that for $N$ large enough, $\overline{D^\delta} \subseteq B_{\cR - \eta}$ and $A_6 \subseteq D^\delta_N$ hold, and all results of Section~\ref{sec:coupling} can be applied with constants that are uniform in $x_B$. Moreover, we will repeatedly use that for $N$ large enough (since $A_2 = B(x_0,\lfloor N^{r_2} \rfloor)$ with $r_2 > r_1$, see~\eqref{eq:r_jDef} and $\mathfrak{r}$ is constant),
\begin{equation}
x \in A_1, z \in B(0,\mathfrak{r}) \qquad \text{implies} \qquad x + z \in A_2.
\end{equation}
 We now see that (recalling again that $B = A_1$), for large enough $N$,
\begin{equation}
\label{eq:JustLookingUntilS_N}
\begin{split}
\widetilde{P}_{y,N}\bigg[\frac{1}{|A_1|}\sum_{x\in A_1} & F\Big((L_{x+z})_{z \in B(0,\mathfrak{r})}\Big) \leq \frac{\vartheta(\varphi_N(x_0)^2)}{(1+\delta)^{1/2}}  \bigg] \\
&   \stackrel{\eqref{eq:ConfinedTiltedCoincide}}{\leq} \overline{P}_y \bigg[\frac{1}{|A_1|}\sum_{x\in A_1} F\Big( \Big(\int_{R_2}^{S_N} \mathbbm{1}_{\{X_t = x + z \}} \mathrm{d}t  \Big)_{z \in B(0,\mathfrak{r})}\Big) \leq \frac{\vartheta(\varphi_N(x_0)^2)}{(1+\delta)^{1/2}}  \bigg],
\end{split}
\end{equation}
where we used that $F$ is increasing. By writing $\overline{P_2^J } = \bigotimes_{j = 2}^J \overline{P}_\sigma$ the law of $J-1$ independent copies of $\overline{P}_\sigma$, we now see that (using again that $F$ is increasing)
\begin{equation}
\label{eq:ApplicationFirstCoupling}
\begin{split}
 \overline{P}_y \bigg[\frac{1}{|A_1|} & \sum_{x\in A_1} F\Big( \Big(\int_{R_2}^{S_N} \mathbbm{1}_{\{X_t = x + z \}} \mathrm{d}t  \Big)_{z \in B(0,\mathfrak{r})}\Big) \leq \frac{\vartheta(\varphi_N(x_0)^2)}{(1+\delta)^{1/2}}  \bigg] \\
 & \stackrel{\eqref{eq:EnoughTrajectories},\eqref{eq:FirstCoupling}}{\leq}  \overline{P_2^J }\bigg[\frac{1}{|A_1|}  \sum_{x\in A_1} F\Big( \Big( \sum_{j = 2}^J \widetilde{L}^j_{x+z}  \Big)_{z \in B(0,\mathfrak{r})}\Big) \leq \frac{\vartheta(\varphi_N(x_0)^2)}{(1+\delta)^{1/2}}  \bigg] + e^{-c\log^2 N},
\end{split}
\end{equation}
where the independent fields $(\widetilde{L}^j_x)_{x \in A_2}$ are defined in Proposition~\ref{prop:CouplingIndependentStartedfromSigme}.
Finally, we see upon using Proposition~\ref{prop:PoissonizationCoupling} and Theorem~\ref{thm:CrucialCoupling} (and the fact that $F$ is increasing) that
\begin{equation}
\label{eq:Application2-3Coupling}
\begin{split}
\overline{P_2^J }\bigg[\frac{1}{|A_1|}  & \sum_{x\in A_1} F\Big( \Big( \sum_{j = 2}^J \widetilde{L}^j_{x+z}  \Big)_{z \in B(0,\mathfrak{r})}\Big) \leq \frac{\vartheta(\varphi_N(x_0)^2)}{(1+\delta)^{1/2}}  \bigg] \\
& \stackrel{ V_1 \geq W }{\leq} {Q}'\bigg[\frac{1}{|A_1|}  \sum_{x\in A_1} F\Big(\Big(  \sum_{j = 2}^J \overline{L}^j_{x+z}  \Big)_{z \in B(0,\mathfrak{r})}\Big) \leq \frac{\vartheta(\varphi_N(x_0)^2)}{(1+\delta)^{1/2}}  \bigg] + e^{-N^c},\\
& \stackrel{\eqref{eq:SecondCoupling},\eqref{eq:ThirdCoupling}}{\leq} \widetilde{Q}\bigg[\frac{1}{|A_1|}  \sum_{x\in A_1} F\Big(\Big( L^{\eta_2}_{x+z} \Big)_{z \in B(0,\mathfrak{r})}\Big) \leq \frac{\vartheta(\varphi_N(x_0)^2)}{(1+\delta)^{1/2}}  \bigg] + e^{-N^c},
\end{split}
\end{equation}
where the i.i.d.~fields $(\overline{L}^j_x)_{x \in A_2}$ are defined in Proposition~\ref{prop:PoissonizationCoupling} (on some auxiliary probability space $(\mathcal{O}', \mathcal{F}',Q')$), $(L^{\eta_2}_x)_{x \in A_2}$ is defined above~\eqref{eq:IntensityMeasuresOfPPP}, and $V_1$ and $W$ are defined in~\eqref{eq:Stopping-times-R-V-def} and~\eqref{eq:W-definition}, respectively. By Proposition~\ref{prop:concentration}  (with $z = x_0$, $B = A_1$, $U = A_2$, $\gamma = r_2/r_1 > 1$) the probability in the last display is bounded above by $e^{-N^{c'}}$. Combining this with~\eqref{eq:JustLookingUntilS_N},~\eqref{eq:ApplicationFirstCoupling}, and~\eqref{eq:Application2-3Coupling}, insertion into~\eqref{eq:unionbound} shows that
\begin{equation}
\widetilde{P}_{y,N}[A_N(F)^c] \leq C N^{d(1-r_1)} (e^{-N^{c}} + e^{-c'\log^2 N}),
\end{equation}
which implies~\eqref{eq:TiltedProbTo1} and therefore finishes the proof.
\end{proof}
\subsection{Applications of Theorem~\ref{thm:MainThm}}
\label{subsec:Applications}

In this short subsection, we give some applications of our main result. \smallskip

As one pertinent example that our result applies to, one can consider the case $F = F_2$ in~\eqref{eq:F_examples_introduction}, namely the (bounded) function $F_2(\ell) = \mathbbm{1}_{\{\ell > 0 \}}$ (and $\mathfrak{r} = 0$). With this choice,~\eqref{eq:CoincidingBounds} yields that for every $y \in \bbZ^d$ and $\nu \in (0,1)$,
\begin{equation}
\label{eq:CoveringResultSec6}
\begin{split}
\lim_{N \rightarrow \infty} & \frac{1}{N^{d-2}} \log P_y\big[|D_N \cap \{X_t \,: \, t \geq 0 \}| > \nu |D_N| \big] \\
& = - \min\bigg\{\frac{1}{2d}\int_{\bbR^d} |\nabla \phi|^2\,\De x\, : \,\phi\in D^{1,2}(\bbR^d)\,,\fint_D (1 - e^{-\phi^2/g(0,0)})\,\De x = \nu\bigg\},
\end{split}
\end{equation} 
and the minimizer $\varphi_{\mathrm{min}}$ in the second line of~\eqref{eq:CoveringResultSec6} can be chosen to be non-negative, see Remark~\ref{rem:ChooseMinimizerNonnegative}.
Remarkably, the strategy used to obtain the lower bound in~\eqref{eq:CoveringResultSec6} involving the tilted walk suggests that in order to cover a macroscopic fraction of a box $D = [-1,1]^d$, the random walk tends to create on mesoscopic scales $M \ll N$ an occupation-time profile roughly coinciding with that of random interlacements at a locally constant level (given by $\varphi_{\mathrm{min}}^2(\cdot/N)$), and therefore the occupation-time profile is potentially positive also (far) away from $D_N$. 
\begin{remark}
\label{rem:SwissCheese} 
Largely motivated by the study of downwards moderate deviations for the volume of the Wiener sausage in~\cite{van2001moderate} (see also its adaptation to the case of the random walk in~\cite{phetpradap2011intersections}), similar questions concerning the deviant behavior of the range of random walks as studied in the present article have appeared, see~\cite{asselah2017moderate,asselah2020extracting,
asselah2020nature,asselah2021two,erhard2023uniqueness}.
It is plausible that the lower bounds derived in this work might give insight into the ``folding behavior'' exhibited by the (discrete-time) simple random walk up to a finite time horizon; see, e.g.,~Theorem 1.4 and~Proposition 5.1 in~\cite{asselah2020extracting}. In a similar direction, it is also natural to study fine properties of the minimizer on the right-hand side of~\eqref{eq:CoveringResultSec6}, in particular its uniqueness (we refer to~\cite{erhard2023uniqueness} for a uniqueness property concerning the aforementioned moderate deviations for the volume of the range of the simple random walk up to a finite time horizon).  In our context, insight into the uniqueness roughly would address whether there are different ``near-optimal'' strategies to enforce the covering event under the probability on the left-hand side of~\eqref{eq:CoveringResultSec6}. \end{remark}
We now explain a different application of our main result relating to the question of the excessive presence of points in a box that are disconnected from the boundary of an enclosing box, a problem that was addressed for random interlacements in~\cite{sznitman2019bulk,sznitman2021excess,sznitman2023cost}. To that end, define for $u \geq 0$ the function
\begin{equation}\label{eq:theta-0}
\vartheta_0(u) = \bbP[0 \stackrel{\cV^u}{\centernot \longleftrightarrow} \infty ],
\end{equation}
where the event under the probability refers to the absence of an arbitrarily long nearest-neighbor paths of pairwise distinct vertices starting in $0$ that remain in the vacant set $\cV^{u} \stackrel{\mathrm{def}}{=} \bbZ^d \setminus \cI^u$ of random interlacements. The function $\vartheta_0$ is known to be non-decreasing, left-continuous, and identically equal to $1$ on $(u_\ast,\infty)$, where $u_\ast$ is the critical level for the phase transition of $\cV^u$, and in fact continuous away from $u_\ast$ (with a possible but not expected jump at $u_\ast$), see~\cite{teixeira2009uniqueness}. We also denote by $\overline{\vartheta}_0$ the right-continuous modification of $\vartheta_0$, and let for $N \geq 1$ (with $S(y,r) = \{x  \in \bbZ^d \, : \, |x - y|_\infty = r \}$ for $r \geq 0$),
\begin{equation}
\begin{split}
\mathcal{C}_N = \text{ the connected component in $S(0,N)$ in }S(0,N) \cup (\bbZ^d \setminus \{X_t \, : \, t  \geq 0 \}  )
\end{split}
\end{equation}
(note that $S(0,N) \subseteq \mathcal{C}_N$ by definition). As we now explain, the results obtained in the present work imply that for $\nu \in (0,1)$, 
\begin{equation}
\label{eq:LowerBoundNumberDiscPoints}
\begin{split}
\liminf_{N \rightarrow \infty} &\frac{1}{N^{d-2}} \log \bbP[|B(0,N) \setminus \mathcal{C}_N| \geq \nu |B(0,N)| ] \geq \\
&  - \inf\bigg\{\frac{1}{2d}\int_{\bbR^d} |\nabla \phi|^2\,\De x\, : \,\phi\in C_0^\infty(\bbR^d)\,,\fint_D \overline{\vartheta}_0(\phi^2)\,\De x > \nu\bigg\}. 
\end{split}
\end{equation}
Indeed, one can for $\mathfrak{r} \geq 1$ consider the case $F = F_3$ in~\eqref{eq:F_examples_introduction}, for which $\vartheta_{\mathfrak{r}}(u) = \bbP[0 \stackrel{\cV^u}{\centernot \longleftrightarrow} S(0,{\mathfrak{r}}) ]$ (with hopefully obvious notation) depends analytically on $u$ (see, e.g.,~\cite[(2.8)]{sznitman2019bulk}), and denote
\begin{equation}
\cD^{\mathfrak{r}} = \{x \in \bbZ^d \, : \, \text{ every path between $x$ and $S(x,\mathfrak{r})$ visits $\{X_t \, : \, t \geq 0 \}$\}},
\end{equation}
i.e.~the sites disconnected by the range of the simple random walk within sup-distance $\mathfrak{r}$. With these definitions, one can then use~\eqref{eq:MainTheoremStatementSec6}  (with the choice $F = F_3$ as mentioned previously) to find that for every $\nu \in (0,1)$,
\begin{equation}
\begin{split}
\liminf_{N \rightarrow \infty} & \frac{1}{N^{d-2}} \log P_y[|\cD^{\mathfrak{r}} \cap D_N| > \nu |D_N| ] \geq - K_{0,\mathfrak{r}}(\nu) \\
& \stackrel{\mathrm{def}}{=} - \inf\bigg\{\frac{1}{2d}\int_{\bbR^d} |\nabla \phi|^2\,\De x\, : \,\phi\in C_0^\infty(\bbR^d)\,,\fint_D \vartheta_{\mathfrak{r}}(\phi^2)\,\De x > \nu\bigg\}.
\end{split}
\end{equation}
One can then repeat the steps in~\cite[Remark 6.6, 2)]{sznitman2019bulk} in combination with~\cite[Theorem 2]{sznitman2021c} to obtain~\eqref{eq:LowerBoundNumberDiscPoints}.
The lower bound so obtained matches the upper bound in~\cite[Corollary 5.3]{sznitman2023cost} upon combination with the recently obtained sharpness of the phase transition for the vacant set of random interlacements, in the series of works~\cite{duminil2023characterization,duminil2023finite,duminil2023phase}, 
(see in particular~\cite[(1.21)]{duminil2023phase}), see also~\cite[Remark 5.4]{sznitman2023cost}. \medskip

\textbf{Acknowledgements.} Parts of this project were carried out within the framework of an \textit{Oberwolfach Research Fellowship} in October 2022 (OWRF, proposal number R2212). The authors would like to thank the institute for its support and hospitality. The authors thank two anonymous reviewers for their careful reading of the manuscript and helpful suggestions. While this work was written,~A.~C.~was associated to INdAM (Istituto Nazionale di Alta Matematica ``Francesco Severi'') and GNAMPA.

\appendix

\begin{appendix}
\section{Proofs of Propositions~\ref{prop:ClosenessToQSD} and~\ref{prop:concentration}}\label{sec:Appendix} %% if no title is needed, leave empty \section*{}.
In this appendix, we provide proofs for technical results used throughout the article. Notably, we provide the main steps of the proof of Proposition~\ref{prop:ClosenessToQSD}, as well as some technical lemmas used both in its proof as well as in the proof of Propositions~\ref{prop:ComparisonEqmeasure} and~\ref{prop:TiltedWalkExcursionNumber}. We also prove Proposition~\ref{prop:concentration} that was used in Section~\ref{sec:lowerbound}. \smallskip

We consider a fixed near-minimizer $\varphi$ as in~\eqref{eq:FixedPhi} and fixed $\delta,\cR$ as in~\eqref{eq:FixedChoiceR}, and all constants may depend implicitly on these choices. Moreover, we also use the definitions of the boxes $A_j$, $1 \leq j \leq 6$ as in~\eqref{eq:AjBoxesDef}. We first collect some hitting time estimates for the confined walk $(\overline{P}_x)_{x \in U^N}$, where throughout the Appendix, we let $f$ be given as in~\eqref{eq:varphi_N_Def} and let $P^{\overline{\mu}}_x$, $x \in \bbZ^d$ stand for the measure governing the random walk with conductances given by $\overline{\mu}$ in~\eqref{eq:NonConstantConductances} and speed measure~\eqref{eq:ChoiceOfSpeedMeasurePi}.
For the following lemma, we also recall that $t_\star = N^2\log^2N$, as defined in~\eqref{eq:RegenerationTimeDef}.
\begin{lemma}
For large $N$, one has
\begin{equation}
\label{eq:HittingUpperBounds}
\begin{split}
\overline{P}_x[H_{A_1} < t_\star] \leq \frac{1}{N^c}, \qquad x \in U^N \setminus A_2, \\
\overline{P}_x[H_{A_2} < t_\star] \leq \frac{1}{N^{c'}}, \qquad x \in U^N \setminus A_3.
\end{split}
\end{equation}
Moreover, uniformly for $x \in \partial_{\mathrm{int}}A_1$, we have
\begin{equation}
\label{eq:SomeBoundsEqmeasure}
\begin{split}
\frac{\overline{e}_{A_1}(x)}{\overline{M}_x}  & \leq P^{\overline{\mu}}_x[T_{A_3} < \widetilde{H}_{A_1}] \leq P^{\overline{\mu}}_x[T_{A_2} < \widetilde{H}_{A_1}] \leq \frac{\overline{e}_{A_1}(x)}{\overline{M}_x}\left(1 + \frac{1}{N^{c''}} \right).
\end{split}
\end{equation}
\end{lemma}
\begin{proof}
The statement of the lemma is the same as that of~\cite[Lemma 3.3]{li2017lower}, but with $\overline{e}_{A_1}/\overline{M}$ in place of $e_{A_1}$ for~\eqref{eq:SomeBoundsEqmeasure}. We first establish~\eqref{eq:HittingUpperBounds}. We aim at following the steps of the proof in~\cite[(3.13)--(3.25)]{li2017lower}. To this end, we will consider a slight modification of $P^{\overline{\mu}}_x$ which is tailor-made to allow to study the confined walk ``up to time $T_{A_6}$''. We let
\begin{equation}
\varphi_N^{A_6}(x) = \varphi_N(x) \qquad x \in B(x_0,\lfloor \tfrac{\widetilde{\delta}}{100}N \rfloor+1), \text{ extended periodically to }\bbZ^d,
\end{equation}
and set
\begin{equation}
\label{eq:NonConstantConductancesA6}
\widehat{\mu}_{x,y} = \frac{1}{2d} \varphi_N^{A_6}(x)\varphi_N^{A_6}(y), \text{ for } x\sim y, x,y \in \mathbb{Z}^d.
\end{equation}
We also define the speed measure
\begin{equation}
\label{eq:ChoiceOfSpeedMeasureA6}
\widehat{\nu}_x = (\varphi_N^{A_6})^2(x), \qquad x \in \mathbb{Z}^d.
\end{equation}
With this, one can consider the random walk $P^{\widehat{\mu};\widehat{\nu}}_x$ using the generator $\mathscr{L}^{\widehat{\mu};\widehat{\nu}}$, see~\eqref{eq:VeryGeneralGenerator} and below it. We abbreviate $\widehat{P}_x =P^{\widehat{\mu};\widehat{\nu}}_x$ for $x \in \bbZ^d$. To start with, note that the conductances~\eqref{eq:NonConstantConductancesA6} fulfill, for large enough $N$,
\begin{equation}
c(\widetilde{\delta}) \varphi_N^2(x_0) \leq \widehat{\mu}_{x,y} \leq C(\widetilde{\delta})\varphi_N^2(x_0),
\end{equation}
and the speed measure is bounded from above and below, $\widehat{\nu} \in [\underline{c}(\widetilde{\delta}),\overline{C}(\widetilde{\delta})]$. It follows that we have the Gaussian heat kernel bounds
\begin{equation}
\label{eq:ModifiedWalkGaussianHKbounds}
\frac{c}{t^{d/2}} e^{- \frac{c'|x-y|^2}{t}} \leq  q^{\widehat{\mu}}_t(x,y) \leq \frac{C}{t^{d/2}} e^{- \frac{C'|x-y|^2}{t}}, \qquad t \geq 1 \vee \widetilde{c}|x-y|, \ x,y \in \mathbb{Z}^d
\end{equation}
(with constants depending on $\widetilde{\delta}$ and $\varphi$), which follows in the same way as Lemma~\ref{lem:NewHK}.  \smallskip

As in~\cite[(3.13)--(3.25)]{li2017lower}, to show~\eqref{eq:HittingUpperBounds} we only need to establish that, for large enough $N$, 
\begin{equation}
\label{eq:(3.16eq)}
\sup_{x \in \partial A_2} \widehat{P}_x[H_{A_1} < T_{A_6}] \leq \frac{1}{N^c},
\end{equation}
as well as, for large enough $N$,
\begin{equation}
\label{eq:(3.19eq)}
\sup_{x \in \partial A_2} \widehat{P}_x[T_{A_6} \leq N^2/\log N] \leq \frac{1}{N^c}.
\end{equation}
We first show~\eqref{eq:(3.16eq)}, which can be seen by decomposing 
\begin{equation}
\label{eq:UpperBoundHittingTiltedProof}
\widehat{P}_x[H_{A_1} < T_{A_6}] = \widehat{P}_x[H_{A_1} < \infty] - \widehat{E}_x\left[ \widehat{P}_{X_{T_{A_6}}}[H_{A_1} < \infty]\right],
\end{equation}
where we used the strong Markov property at time $T_{A_6}$ ($<\infty$, $\widehat{P}_x$-a.s.). Since $d_\infty(A_1,\partial A_6) \geq cN$ and $d_\infty(x,A_1) > cN^{r_2}$ for large enough $N$, we can apply a modification of~\eqref{eq:LastExitDecompisitionTilted} and the arguments below it to show the required polynomial decay. We now argue for~\eqref{eq:(3.19eq)}.  Let $x \in \partial A_2$, then by~\cite[Lemma 5.22, Lemma 5.21 and proof]{barlow2017random}, we have in view of the Gaussian heat kernel bounds~\eqref{eq:ModifiedWalkGaussianHKbounds}
\begin{equation}
\begin{split}
\widehat{P}_x[T_{A_6} \leq N^2/\log N] & \leq C \exp\bigg(-c\bigg(1-\frac{\widetilde{\delta}}{100} \bigg)^2 \frac{N^2}{(N^2 / \log N)}\bigg) \leq \frac{1}{N^c},
\end{split}
\end{equation}
for $N$ large enough, which yields the claim. 
 \smallskip

We now turn to the proof of~\eqref{eq:SomeBoundsEqmeasure}, where the only non-trivial part of the statement is the last inequality. To this end, we adapt the arguments below~\cite[(3.25)]{teixeira2011fragmentation} to the framework of the random walk among conductances $\overline{\mu}$. We first recall~\eqref{eq:HittingOfBoxEstimateII}, so for large enough $N$ also $ P^{\overline{\mu}}_y[H_{A_1} = \infty] \geq c (>0)$ when $y \in \partial A_2$. We then see that
\begin{equation}
\begin{split}
P^{\overline{\mu}}_x[T_{A_2} <\widetilde{H}_{A_1}, \widetilde{H}_{A_1} < \infty] & \stackrel{\eqref{eq:HittingOfBoxEstimateII}}{\leq}  N^{-c} P^{\overline{\mu}}_x[T_{A_2} < \widetilde{H}_{A_1}] \\
& \leq N^{-c} P^{\overline{\mu}}_x[T_{A_2} < \widetilde{H}_{A_1}] \inf_{y \in \partial A_2} \frac{ P^{\overline{\mu}}_{y}[H_{A_1}  =\infty] }{c} \\
& \leq N^{-c} \frac{P^{\overline{\mu}}_{x}[\widetilde{H}_{A_1} = \infty]}{c} = \frac{1}{cN^c}  \frac{\overline{e}_{A_1}(x)}{\overline{M}_x}, \qquad x \in \partial_{\mathrm{int}}A_1.
\end{split}
\end{equation}
Therefore, we easily conclude that for $N$ large enough,
\begin{equation}
\begin{split}
P^{\overline{\mu}}_x[T_{A_2} <\widetilde{H}_{A_1}] & = P^{\overline{\mu}}_x[ \widetilde{H}_{A_1} = \infty] + P^{\overline{\mu}}_x[T_{A_2} <\widetilde{H}_{A_1} < \infty]  \\
& \leq \left(1 + \frac{1}{N^{c''}}\right)\frac{\overline{e}_{A_1}(x)}{\overline{M}_x}, \qquad  x \in \partial_{\mathrm{int}}A_1,
\end{split}
\end{equation}
giving us the required bounds.
\end{proof}

We state some further hitting time estimates that involve the stationary distribution $\pi$ (see~\eqref{eq:piDef}) of the confined walk on $U^N$ (recall that $f$ and $\varphi_N$ is chosen as in~\eqref{eq:varphi_N_Def}).

\begin{lemma}
\label{lem:Morehittingtimebounds}
For large enough $N$, we have the following estimates:
\begin{equation}
\begin{split}
\label{eq:InverseHittingEstimates}
\left(1 - \frac{1}{N^{c}} \right)\frac{(1+\varepsilon) \capa(A_1)}{S_N} \varphi_N^2(x_0) & \leq \frac{1}{\overline{E}_\pi[H_{A_1}]} \\
& \leq \left(1 + \frac{1}{N^{c'}} \right) \frac{(1+\varepsilon)\capa(A_1)}{S_N} \varphi_N^2(x_0),
\end{split}
\end{equation}
and
\begin{equation}
\begin{split}
\label{eq:ImportantHittingTimeEstimate}
\frac{1}{\overline{E}_\pi[H_{A_j}]} & \leq  \frac{C}{N^{d - (d-2)r_j}}, \qquad j \in \{1,2\}.
\end{split}
\end{equation}
Moreover, if we let $f_{A_1}$ stand for the function
\begin{equation}
\label{eq:fA1Def}
f_{A_1}(x) = 1 - \frac{\overline{E}_x[H_{A_1}]}{\overline{E}_\pi[H_{A_1}]}, \qquad x \in U^N,
\end{equation}
then we have for $N$ large enough that
\begin{equation}
\label{eq:fA1Bound}
f_{A_1}(x) \geq -\frac{1}{N^c},\qquad x \in U^N.
\end{equation}
\end{lemma}
\begin{proof}
The arguments are exactly as in the proofs of~\cite[Propositions 3.5 and 3.7, Lemma 3.6]{li2017lower}, and we explain the necessary changes to adapt the proofs to our context. \smallskip

We use the Dirichlet form of the confined walk given by 
\begin{equation}
\overline{\cE}(g,g) = \frac{1}{2}\sum_{x,y \in U^N, x \sim y} \frac{f(x)f(y)}{2d} \Big(g(x) - g(y) \Big)^2, \qquad g : U^N \rightarrow \bbR.
\end{equation}
Inserting the function 
\begin{equation}
g_{A_1,A_2}(x) = \overline{P}_x[H_{A_1} < T_{A_2}] = P^{\overline{\mu}}_x[H_{A_1} < T_{A_2}], \qquad x \in U^N,
\end{equation}
which is $(-\mathscr{L}^{\overline{\mu}})$-harmonic in $A_2 \setminus A_1$, equal to $1$ on $A_1$ and equal to $0$ outside of $A_2$, a straightforward adaptation of the proof of~\cite[Proposition 3.4]{li2017lower} yields
\begin{equation}
\begin{split}
\left(1 - \frac{1}{N^{c'}} \right) \frac{(1+\varepsilon)\varphi_N^2(x_0)}{S_N} \sum_{x \in \partial_{\mathrm{int}}A_1} & P^{\overline{\mu}}_x[T_{A_2} < \widetilde{H}_{A_1}]  \leq \overline{\cE}(g_{A_1,A_2},g_{A_1,A_2})  \\
& \leq \left(1 + \frac{1}{N^{c'}} \right) \frac{(1+\varepsilon)\varphi_N^2(x_0)}{S_N} \sum_{x \in \partial_{\mathrm{int}}A_1} P^{\overline{\mu}}_x[T_{A_2} < \widetilde{H}_{A_1}],
\end{split}
\end{equation}
where we used that by~\eqref{eq:time_horizon}, $S_N =(1+\varepsilon) \|\varphi_N\|_{\ell^2(\bbZ^d)}^2$ and the smoothness of $\varphi$ (see Proposition~\ref{prop:quasi-minimizer} i)). We then use~\eqref{eq:SomeBoundsEqmeasure},~\eqref{eq:MandVarphi2Close},
and the bounds~\eqref{eq:CapaComparison} to obtain
\begin{equation}
\label{eq:TiltedDirichletArgument}
\left(1 - \frac{1}{N^{c}} \right) \frac{(1+\varepsilon)\capa(A_1)}{S_N}\varphi_N^2(x_0) \leq \overline{\cE}(g_{A_1,A_2},g_{A_1,A_2}) \leq \left(1 + \frac{1}{N^{c'}} \right) \frac{(1+\varepsilon)\capa(A_1)}{S_N}\varphi_N^2(x_0).  
\end{equation}
The second inequality of~\eqref{eq:InverseHittingEstimates} and~\eqref{eq:ImportantHittingTimeEstimate} then follow exactly as in the proof of~\cite[Proposition 3.5]{li2017lower} (where we use~\eqref{eq:S_N_volume_order} in place of~\cite[(2.20) 5.]{li2017lower}). \smallskip

The proof of~\eqref{eq:fA1Bound} follows line by line as that of the same statement in~\cite[Lemma 3.6]{li2017lower}, using the second inequality of~\eqref{eq:InverseHittingEstimates} in place of~\cite[(3.32)]{li2017lower}, the relaxation estimate~\eqref{lem:TVtoStationary} in place of~\cite[(2.64)]{li2017lower}, and the estimate $\frac{1}{2d}\sqrt{\pi(x)\pi(y)} \in (cN^{-(4+d)},1]$, which follows from Lemma~\ref{lem:TechnicalProperties}, ii) and the argument in~\cite[(2.59)]{li2017lower}. \smallskip

Finally, the first inequality of~\eqref{eq:InverseHittingEstimates} is a straightforward adaptation of the proof of~\cite[Proposition 3.7]{li2017lower}, using~\eqref{eq:TiltedDirichletArgument},~\eqref{eq:fA1Bound} in place of~\cite[(3.36)]{li2017lower}, as well as~\eqref{eq:HittingUpperBounds} (in place of~\cite[(3.10)]{li2017lower}) and~\eqref{eq:ImportantHittingTimeEstimate} in place of~\cite[(3.33)]{li2017lower}.
\end{proof}

We now turn to the proof of Proposition~\ref{prop:ClosenessToQSD}, which follows that of~\cite[Proposition 4.5]{li2017lower} (see also~\cite[Lemma 3.9]{teixeira2011fragmentation}).

\begin{proof}[Proof of Proposition~\ref{prop:ClosenessToQSD}]
We need some controls on spectral quantities attached to the generator~$\mathscr{L}^{U^N \setminus A_2}$ and the semigroup $(\mathscr{H}^{U^N \setminus A_2}_t)_{t \geq 0}$, see~\eqref{eq:SemigroupDef} and below, which we simply denote by $\mathscr{L}$ and the semigroup $(\mathscr{H}_t)_{t \geq 0}$ in this section. The eigenvalues of $-\mathscr{L}$ are denoted by
\begin{equation}
0 \leq \lambda_i \leq \lambda_{i+1}, \qquad 1 \leq i \leq |U^N \setminus A_2|,
\end{equation}
and we let $(f_i)_{1 \leq i \leq |U^N \setminus A_2|}$ stand for an $\langle \cdot,\cdot\rangle_{\ell^2(U^N\setminus A_2, \pi^{U^N\setminus A_2})}$-orthonormal basis of eigenfunctions associated with $(\lambda_i)_{1 \leq i \leq |U^N \setminus A_2|}$ (so $\lambda_1 =  \lambda_1^{U^N\setminus A_2}$ in the notation above~\eqref{eq:qsd_Def}). We also recall the definition of the quasi-stationary distribution $\sigma$ from~\eqref{eq:qsd_Def}. In what follows, we need to establish the control for large enough $N$, 
\begin{equation}
\label{eq:MinimumOfSigmaBound}
\min_{x \in U^N \setminus A_2} \sigma(x) \geq \frac{1}{N^c,}
\end{equation}
as well as the polynomial bounds
\begin{equation}
\label{eq:PolynomialControlsEigenfunction}
\frac{1}{N^{c'}} \leq \min_{x \in U^N \setminus A_2} f_1(x) \leq \max_{x \in U^N \setminus A_2} f_1(x) \leq N^{c''}.
\end{equation}
These controls are analogues of~\cite[(4.20)--(4.21)]{li2017lower}, and we only point out the relevant changes in our set-up due to the choice of a different tilting function $\varphi_N$. It suffices to show that for all $x,x' \in U^N \setminus A_2$, for large enough $N$
\begin{equation}
\label{eq:(4.24)analogue}
\overline{P}_x[H_{x'} < H_{A_2}] \geq\frac{1}{N^c}.
\end{equation}
The claim will then follow by using the fact that
\begin{equation}
\sigma(y) \geq \frac{1}{N^c} \overline{P}_y[H_x < H_{A_2}] \sigma(x), \qquad x,y \in U^N \setminus A_2,
\end{equation}
which is exactly~\cite[Lemma 4.1]{li2017lower} (and its proof remains unchanged, since we can apply Lemma~\ref{lem:TechnicalProperties}, ii) in the same way). As in~\cite{li2017lower}, we decompose the proof of~\eqref{eq:(4.24)analogue} into two cases: \smallskip

\textit{Case I: $x' \in A_4 \setminus A_2$.} One can first see that by using~\eqref{eq:LastExitDecompisitionTilted}, for large enough $N$,
\begin{equation}
\overline{P}_x[H_{\partial A_5} < H_{A_2}] \geq \frac{1}{N^c}.
\end{equation}
One can then introduce
\begin{equation}
l(x) = \overline{P}_x[H_{x'} < H_{A_2}],
\end{equation}
and use the strong Markov property at $H_{\partial A_5}$ to show that 
\begin{equation}
l(x) \geq \frac{1}{N^c}\min_{y \in \partial A_5} l(y), \qquad x \in U^N \setminus A_2.
\end{equation}
Next, we proceed as in~\cite[(4.27)--(4.28)]{li2017lower}, but noting that $l$ is $(-\mathscr{L})$-harmonic in $U^N \setminus (A_2 \cup \{x'\})$ and by employing the Harnack inequality on the weighted graph $(\mathbb{Z}^d,\mathbb{E}^d,\overline{\mu})$ with uniformly elliptic weights given in~\eqref{eq:NonConstantConductances} (see, e.g.,~\cite[Théorème 1]{delmotte1997inegalite}), which implies that
\begin{equation}
\min_{z \in \partial A_5} l(z) \geq c' \max_{z \in \partial A_5}l(z), 
\end{equation}
since $\partial A_5 \subseteq B(x_0,2N^{r_5}) \setminus B(x_0,\frac{1}{2}
N^{r_5}) \subseteq U^N\setminus A_2$ for $N$ large enough. Therefore the proof is reduced to showing that 
\begin{equation}
\label{eq:LowerBoundHittingSpecificPoint}
\overline{P}_{y'}[X_{T_{B}} = x'] \geq CN^{-c},
\end{equation}
for $B = B(y',|y'-x'|_\infty-1)$ and $y' \in \partial A_5$ is a point of least distance in $\ell^\infty$-norm to $x'$, sharing $(d-1)$ common coordinates with $x'$ (so $x' \in \partial B$) (for definiteness, we choose the lexicographically smallest such point $y'$). The validity of~\eqref{eq:LowerBoundHittingSpecificPoint} can then be shown by a similar argument as for the case of constant conductances (see, e.g.,~\cite[Lemma 6.3.7, p.~158--159]{lawler2010random}). This concludes the proof for Case I. \smallskip

\textit{Case II: $x' \in U_N  \setminus A_4$.} The proof carries over without any changes, as long as 
\begin{equation}
\max_{y \in \partial A_3} \overline{P}_y[H_x < H_{A_2}] \geq \frac{1}{N^c}
\end{equation}
holds for large enough $N$, which itself can be proved as in~\cite[Lemma 4.2]{li2017lower}. \smallskip

Finally, we point out that the proof of~\eqref{eq:PolynomialControlsEigenfunction} is a repetition of~\cite[(4.34)--(4.38)]{li2017lower}, using $|U^N\setminus A_2| \leq cN^d$, as well as Lemma~\ref{lem:TechnicalProperties} ii) and~\eqref{eq:MinimumOfSigmaBound} in place of (2.38)4.~and (4.20) of~\cite{li2017lower}. \smallskip

With these preparations at hand, we now prove~\eqref{eq:ConditionalDistributionCloseToqsd}.
To this end, we write 
\begin{equation}
\label{eq:QuotientWithSemigroups}
\overline{P}_x[X_{t_\star} = y | H_{A_2} > t_{\ast}] = \frac{(\mathscr{H}_{t_\star} \delta_y)(x)}{(\mathscr{H}_{t_\star} \boldsymbol{1})(x)}, \qquad x,y \in U^N \setminus A_2.
\end{equation}
see also~\cite[(4.47)]{li2017lower} (or~\cite[(A.8)]{teixeira2011fragmentation} for a similar argument). Upon decomposing $\delta_y$ into its components with respect to the orthonormal basis $(f_i)_{1 \leq i \leq |U^N \setminus A_2|}$, we see that
\begin{equation}
\label{eq:SemigroupAppliedToDelta}
(\mathscr{H}_{t_\star} \delta_y)(x) = e^{-\lambda_1 t_\star} \left(f_1(x)f_1(y)\pi^{U^N\setminus A_2}(y) + \sum_{i = 2}^{|U^N \setminus A_2|} e^{(\lambda_1 - \lambda_i)t_{\ast}}f_i(x)f_i(y) \pi^{U^N\setminus A_2}(y)  \right).
\end{equation}
On one hand, we have that for large enough $N$
\begin{equation}
\label{eq:f1f2piLargeenough}
f_1(x)f_1(y)\pi^{U^N\setminus A_2}(y) \stackrel{\text{Lemma~\ref{lem:TechnicalProperties}, ii)},\eqref{eq:PolynomialControlsEigenfunction}}{\geq} \frac{1}{N^c}.
\end{equation}
We now explain that the terms in the sum~\eqref{eq:SemigroupAppliedToDelta} are negligible. This will follow immediately once we show that for large $N$,
\begin{equation}
\label{eq:EVdifferenceLargeEnough}
\lambda_2 - \lambda_1 \geq cN^{-2}.
\end{equation}
Indeed, admitting~\eqref{eq:EVdifferenceLargeEnough} we can follow the proof of as~\cite[(4.51)--(4.54)]{li2017lower} to see that for large enough $N$
\begin{equation}
\begin{split}
\left\vert\sum_{i = 2}^{|U^N \setminus A_2|} e^{(\lambda_1 - \lambda_2)t_{\ast}} f_i(x)f_i(y)\pi^{U^N\setminus A_2}(y) \right\vert & \leq cN^d e^{-c\log^2N} \sup_{2 \leq i \leq |U^N\setminus A_2|} | f_i(x)f_i(y)\pi^{U^N\setminus A_2}(y)|  \\
& \stackrel{\text{Lemma~\ref{lem:TechnicalProperties}, ii),} \eqref{eq:PolynomialControlsEigenfunction}}{\leq} c'N^{d +2c''}e^{-c\log^2N} \leq e^{-\tilde{c}\log^2N},
\end{split}
\end{equation}
which in turn implies that for large enough $N$,
\begin{equation}
\left\vert \frac{(\mathscr{H}_{t_\star}\delta_y)(x) }{e^{-\lambda_1 t_\star} f_1(x)f_1(y) \pi^{U^N\setminus A_2}(y)} - 1 \right\vert \leq e^{-c\log^2N}, \qquad x,y \in U^N \setminus A_2,
\end{equation}
where we used~\eqref{eq:f1f2piLargeenough}. By the same arguments one shows that
\begin{equation}
\left\vert \frac{(\mathscr{H}_{t_\star}\boldsymbol{1})(x) }{e^{-\lambda_1 t_\star} f_1(x)\langle f_1, \boldsymbol{1}\rangle_{\ell^2(U^N\setminus A_2, \pi^{U^N\setminus A_2})}} -1  \right\vert \leq e^{-c\log^2N}, \qquad x \in U^N \setminus A_2.
\end{equation}
Using~\eqref{eq:QuotientWithSemigroups}, we obtain
\begin{equation}
\begin{split}
|\overline{P}_x[X_{t_\star} = y | H_{A_2} > t_{\ast}] - \sigma(y)| & = \left\vert \frac{(\mathscr{H}_{t_\star} \delta_y)(x)}{(\mathscr{H}_{t_\star} \boldsymbol{1})(x)} - \sigma(y) \right\vert \\
& \leq e^{-c\log^2 N}\sigma(y) \leq e^{-c\log^2 N}, \qquad x,y \in U^N \setminus A_2.
\end{split}
\end{equation}
It remains to show~\eqref{eq:EVdifferenceLargeEnough}. However, this is exactly the statement of~\cite[Lemma 4.4]{li2017lower}, and its proof carries over to our framework without any changes (using~\eqref{eq:ImportantHittingTimeEstimate} in the process).
\end{proof}

\begin{proof}[Proof of Proposition~\ref{prop:concentration}] First of all, by Poisson concentration, we have,
\begin{equation}
\label{eq:PoissonConcentrationDelta}
    \overline{Q}[\eta(\Gamma(\bbZ^d)) < a(1+\Delta/2)\capa(B)]\leq 2 \exp(-c a\Delta^2\capa(B)).
\end{equation}
Let us set $m \stackrel{\mathrm{def}}{=} \lfloor a(1+\Delta/2)\capa(B)] \rfloor$, and an i.i.d.~sequence of independent excursions  $(Z^\ell)_{\ell \geq 1} \in \Gamma(\bbZ^d)$ sampled according to $\kappa$. Denote the joint law by $\cQ$. We denote for $n = 0,\ldots,m$ and all $x\in B^\frr$, where $B^\frr =\{x+z:x\in B,\,|z|_\infty\leq \frr\}$ is the discrete $\frr$-neighborhood of $B$,
\begin{equation}
    \begin{aligned}
        L^n_x& \stackrel{\mathrm{def}}{=} \sum_{\ell=1}^n L_x(Z^\ell), \qquad L_x(Z^\ell) \stackrel{\mathrm{def}}{=} \int_0^{T_U(Z^\ell)} \IND_{\{Z^\ell_t = x\}}\,\De t,\\
        F_{n} & \stackrel{\mathrm{def}}{=}\sum_{x\in B} F((L^n_{x+z})_{z\in B(0,\mathfrak{r})}),
    \end{aligned}
\end{equation}
where it is understood that the empty sum is equal to zero. We also introduce a truncation of the excursions, and for that we define for $\ell\geq 1$  the stopping times
\begin{equation}
    \sigma_\ell = \inf\Big\{ s\geq 0 :\, |Z^\ell_{[0,s]}\cap B^\frr| \geq N^{2+\frac{1}{4}}\text{ or }\int_0^{s \wedge T_U(Z^\ell)} \IND_{\{Z^\ell_t \in B^\frr\}}\,\De t \geq N^{2+\frac{1}{4}}\Big\},
\end{equation}
(where $Z_{[0,s]}^\ell = \{ Z_t^\ell \, : \, t \in [0,s] \}$ is the range of $Z^\ell$) and the truncated quantities
\begin{equation}
    \begin{aligned}
        \widetilde{L}^n_x&\stackrel{\mathrm{def}}{=} \sum_{\ell=1}^n \widetilde{L}_x(Z^\ell), \qquad \widetilde{L}_x(Z^\ell) \stackrel{\mathrm{def}}{=} \int_0^{T_U(Z^\ell)\wedge \sigma_\ell} \IND_{\{Z^\ell_t = x\}}\,\De t,\\
        \widetilde{F}_{n} &\stackrel{\mathrm{def}}{=} \sum_{x\in B} F((\widetilde{L}^n_{x+y})_{y\in B(0,\mathfrak{r})}).
    \end{aligned}
\end{equation}
The same martingale argument as in~\cite[(2.49)--(2.53)]{sznitman2019bulk} using the Azuma-Hoeffding inequality for bounded martingale differences shows that 
\begin{equation}
\label{eq:AzumaHoeffding}
    \cQ\Big[|\widetilde{F}_{m} - E_\cQ[\widetilde{F}_m]| \geq N^{d - \frac{1}{10}} \Big] \leq 2 \exp\Big( -\frac{c(F)}{a(1 + \Delta/2)} N^{\frac{1}{4}} \Big), \qquad N \geq 1.
\end{equation}

To conclude the proof, we establish the following controls on the expectation of $\widetilde{F}_m$ and $F_m$ under $\cQ$.
\begin{lemma} 
For any $\beta > 0$, we have for $N$ large enough,
\begin{align}
\label{eq:TwoExpectationsClose}
   0 & \leq E_{\cQ}[F_{\lfloor \beta \capa(B) \rfloor}] -  E_{\cQ}[\widetilde{F}_{\lfloor \beta \capa(B) \rfloor}] \leq c\beta e^{-c'N^{\frac{1}{4}}}, \ \text{and}\\
   \label{eq:ExpectationCloseToTheta}
   \lim_{N \rightarrow \infty} \frac{1}{|B|} &  E_{\cQ}[F_{\lfloor \beta \capa(B) \rfloor}] = \vartheta(\beta)
\end{align}
(recall the definition of $\vartheta$ from~\eqref{eq:def theta}).
\end{lemma}
\begin{proof}
This follows from a straightforward adaptation of the proof of~\cite[Lemma 2.3]{sznitman2019bulk} to our purposes, as we now briefly explain. Indeed,~\eqref{eq:TwoExpectationsClose} follows from~\cite[(2.56)--(2.61)]{sznitman2019bulk}, which actually yields an upper bound of the form
\begin{equation}
 E_{\cQ}[F_{\lfloor \beta \capa(B) \rfloor}] -  E_{\cQ}[\widetilde{F}_{\lfloor \beta \capa(B) \rfloor}] \leq C\beta \capa(B) (N + \mathfrak{r})^2 e^{- \frac{cN^{2 + \frac{1}{4}}}{(N+\mathfrak{r})^2}} + C'\beta(N+\mathfrak{r})^d e^{- \frac{c'N^{2 + \frac{1}{4}}}{(N+\mathfrak{r})^2}}.
\end{equation}
The claim~\eqref{eq:ExpectationCloseToTheta} follows in the same way as the argument in~\cite[(2.63)--(2.68)]{sznitman2019bulk}, noting that the quantity $a$ in (2.67) of the same reference can in fact be bounded above by $c(1 + \mathfrak{r})^{d-2}/ N^{\gamma(d-2)}$.
\end{proof}
We now explain how the proof of Proposition~\ref{prop:concentration} can be concluded. Indeed, we observe that by~\eqref{eq:PoissonConcentrationDelta}, it follows that
\begin{equation}
    \overline{Q} \Big[\frac{1}{|B|}\sum_{x\in B} F((L^\eta_{x+y})_{y\in B(0,\mathfrak{r})}) \leq \vartheta(a) \Big]\leq  \cQ\Big[F_m \leq \vartheta(a)|B| \Big] + 2 \exp(- a c\Delta^2\capa(B))
\end{equation}
(by decomposing into the events that the number of trajectories is either smaller than or larger or equal to $a(1 + \Delta/2)\capa(B)$). The proof then follows by combining~\eqref{eq:AzumaHoeffding},~\eqref{eq:TwoExpectationsClose}, and~\eqref{eq:ExpectationCloseToTheta}.
\end{proof}
\end{appendix}

\bibliographystyle{plain}
\bibliography{literature}

\end{document}